\documentclass[12pt]{amsart}
\usepackage{txfonts}      
\usepackage{amssymb}
\usepackage{eucal}
\usepackage{amsmath}
\usepackage{amscd}
\usepackage{xcolor}
\usepackage{multicol}
\usepackage[all]{xy}           
\usepackage{graphicx}
\usepackage{color}
\usepackage{colordvi}
\usepackage{xspace}
\usepackage{tikz}
\usepackage{makecell}
\usepackage{appendix}
\usepackage{amsthm}

\usepackage{ifpdf}
\ifpdf
\usepackage[colorlinks,final,backref=page,hyperindex]{hyperref}
\else
\usepackage[colorlinks,final,backref=page,hyperindex,hypertex]{hyperref}
\fi

\usepackage[active]{srcltx} 

\begin{document}


\newtheorem{thm}{Theorem}[section]
\newtheorem{lem}[thm]{Lemma}
\newtheorem{cor}[thm]{Corollary}
\newtheorem{pro}[thm]{Proposition}
\theoremstyle{definition}
\newtheorem{defi}[thm]{Definition}
\newtheorem{ex}[thm]{Example}
\newtheorem{rmk}[thm]{Remark}
\newtheorem{pdef}[thm]{Proposition-Definition}
\newtheorem{condition}[thm]{Condition}

\renewcommand{\labelenumi}{{\rm(\alph{enumi})}}
\renewcommand{\theenumi}{\alph{enumi}}

\newcommand {\emptycomment}[1]{} 

\newcommand{\nc}{\newcommand}
\newcommand{\delete}[1]{}

\nc{\todo}[1]{\tred{To do:} #1}

\nc{\tred}[1]{\textcolor{red}{#1}}
\nc{\tblue}[1]{\textcolor{blue}{#1}}
\nc{\tgreen}[1]{\textcolor{green}{#1}}
\nc{\tpurple}[1]{\textcolor{purple}{#1}}
\nc{\tgray}[1]{\textcolor{gray}{#1}}
\nc{\torg}[1]{\textcolor{orange}{#1}}
\nc{\tmag}[1]{\textcolor{magenta}}
\nc{\btred}[1]{\textcolor{red}{\bf #1}}
\nc{\btblue}[1]{\textcolor{blue}{\bf #1}}
\nc{\btgreen}[1]{\textcolor{green}{\bf #1}}
\nc{\btpurple}[1]{\textcolor{purple}{\bf #1}}

    \nc{\mlabel}[1]{\label{#1}}  
    \nc{\mcite}[1]{\cite{#1}}  
    \nc{\mref}[1]{\ref{#1}}  
    \nc{\meqref}[1]{\eqref{#1}}  
    \nc{\mbibitem}[1]{\bibitem{#1}} 

\delete{
    \nc{\mcite}[1]{\cite{#1}{\small{\tt{{\ }(#1)}}}}  
    \nc{\mref}[1]{\ref{#1}{\small{\tred{\tt{{\ }(#1)}}}}}  
    \nc{\meqref}[1]{\eqref{#1}{{\tt{{\ }(#1)}}}}  
    \nc{\mbibitem}[1]{\bibitem[\bf #1]{#1}} 
}


\nc{\cm}[1]{\textcolor{red}{Chengming:#1}}
\nc{\yy}[1]{\textcolor{blue}{Yanyong: #1}}
\nc{\li}[1]{\textcolor{purple}{#1}}
\nc{\lir}[1]{\textcolor{purple}{Li:#1}}


\nc{\tforall}{\ \ \text{for all }}
\nc{\hatot}{\,\widehat{\otimes} \,}
\nc{\complete}{completed\xspace}
\nc{\wdhat}[1]{\widehat{#1}}

\nc{\ts}{\mathfrak{p}}
\nc{\mts}{c_{(i)}\ot d_{(j)}}

\nc{\NA}{{\bf NA}}
\nc{\LA}{{\bf Lie}}
\nc{\CLA}{{\bf CLA}}

\nc{\cybe}{CYBE\xspace}
\nc{\nybe}{NYBE\xspace}
\nc{\ccybe}{CCYBE\xspace}
\nc{\enybe}{ENYBE\xspace}
\nc{\genybe}{GENYBE\xspace}

\nc{\ndend}{pre-Novikov\xspace}
\nc{\calb}{\mathcal{B}}
\nc{\rk}{\mathrm{r}}
\newcommand{\g}{\mathfrak g}
\newcommand{\h}{\mathfrak h}
\newcommand{\pf}{\noindent{$Proof$.}\ }
\newcommand{\frkg}{\mathfrak g}
\newcommand{\frkh}{\mathfrak h}
\newcommand{\Id}{\rm{Id}}
\newcommand{\gl}{\mathfrak {gl}}
\newcommand{\ad}{\mathrm{ad}}
\newcommand{\add}{\frka\frkd}
\newcommand{\frka}{\mathfrak a}
\newcommand{\frkb}{\mathfrak b}
\newcommand{\frkc}{\mathfrak c}
\newcommand{\frkd}{\mathfrak d}
\newcommand {\comment}[1]{{\marginpar{*}\scriptsize\textbf{Comments:} #1}}


\nc{\zhushi}[1]{}

\nc{\hr}{\hat{r}}
\nc{\hs}{\hat{s}}
\nc{\haa}{\hat{\alpha}}
\nc{\hbb}{\hat{\beta}}
\nc{\hA}{\hat{A}}

\nc{\otherbeta}{\bm{\beta}}

\nc{\bbvar}{P_{\beta}}
\nc{\Tvar}{P_{T}}

\nc{\tA}{\mathcal{A}}     
\nc{\taa}{\mathcal{P}}
\nc{\tbb}{\mathcal{Q}}  
\nc{\tga}{\gamma^{\mathcal{A}}}   
\nc{\tT}{\tilde{T}}

\nc{\caa}{\check{\alpha}}   \nc{\cbb}{\check{\beta}}
\nc{\chr}{\check{r}} \nc{\crp}{\check{r_+}}  \nc{\crm}{\check{r_-}}
\nc{\chcalp}{\check{\mathcal{P}}}

\nc{\qsan}{\circ_\Delta}

\nc{\postn}{\text{Post-Novikov algebra}}
\nc{\postns}{\text{Post-Novikov algebras}}
\nc{\aax}{\alpha(x)}   
\nc{\aay}{\alpha(y)}   
\nc{\aau}{\alpha(u)}
\nc{\aav}{\alpha(v)}


\nc{\vspa}{\vspace{-.1cm}}
\nc{\vspb}{\vspace{-.2cm}}
\nc{\vspc}{\vspace{-.3cm}}
\nc{\vspd}{\vspace{-.4cm}}
\nc{\vspe}{\vspace{-.5cm}}


\nc{\disp}[1]{\displaystyle{#1}}
\nc{\bin}[2]{ (_{\stackrel{\scs{#1}}{\scs{#2}}})}  
\nc{\binc}[2]{ \left (\!\! \begin{array}{c} \scs{#1}\\
    \scs{#2} \end{array}\!\! \right )}  
\nc{\bincc}[2]{  \left ( {\scs{#1} \atop
    \vspace{-.5cm}\scs{#2}} \right )}  
\nc{\ot}{\otimes}
\nc{\sot}{{\scriptstyle{\ot}}}
\nc{\otm}{\overline{\ot}}
\nc{\ola}[1]{\stackrel{#1}{\la}}

\nc{\scs}[1]{\scriptstyle{#1}} \nc{\mrm}[1]{{\rm #1}}

\nc{\dirlim}{\displaystyle{\lim_{\longrightarrow}}\,}
\nc{\invlim}{\displaystyle{\lim_{\longleftarrow}}\,}

\nc{\bfk}{{\bf k}} \nc{\bfone}{{\bf 1}}
\nc{\rpr}{\circ}
\nc{\dpr}{{\tiny\diamond}}
\nc{\rprpm}{{\rpr}}

\nc{\mmbox}[1]{\mbox{\ #1\ }} \nc{\ann}{\mrm{ann}}
\nc{\Aut}{\mrm{Aut}} \nc{\can}{\mrm{can}}
\nc{\twoalg}{{two-sided algebra}\xspace}
\nc{\colim}{\mrm{colim}}
\nc{\Cont}{\mrm{Cont}} \nc{\rchar}{\mrm{char}}
\nc{\cok}{\mrm{coker}} \nc{\dtf}{{R-{\rm tf}}} \nc{\dtor}{{R-{\rm
tor}}}
\renewcommand{\det}{\mrm{det}}
\nc{\depth}{{\mrm d}}
\nc{\End}{\mrm{End}} \nc{\Ext}{\mrm{Ext}}
\nc{\Fil}{\mrm{Fil}} \nc{\Frob}{\mrm{Frob}} \nc{\Gal}{\mrm{Gal}}
\nc{\GL}{\mrm{GL}} \nc{\Hom}{\mrm{Hom}} \nc{\hsr}{\mrm{H}}
\nc{\hpol}{\mrm{HP}}  \nc{\id}{\mrm{id}} \nc{\im}{\mrm{im}}

\nc{\incl}{\mrm{incl}} \nc{\length}{\mrm{length}}
\nc{\LR}{\mrm{LR}} \nc{\mchar}{\rm char} \nc{\NC}{\mrm{NC}}
\nc{\mpart}{\mrm{part}} \nc{\pl}{\mrm{PL}}
\nc{\ql}{{\QQ_\ell}} \nc{\qp}{{\QQ_p}}
\nc{\rank}{\mrm{rank}} \nc{\rba}{\rm{RBA }} \nc{\rbas}{\rm{RBAs }}
\nc{\rbpl}{\mrm{RBPL}}
\nc{\rbw}{\rm{RBW }} \nc{\rbws}{\rm{RBWs }} \nc{\rcot}{\mrm{cot}}
\nc{\rest}{\rm{controlled}\xspace}
\nc{\rdef}{\mrm{def}} \nc{\rdiv}{{\rm div}} \nc{\rtf}{{\rm tf}}
\nc{\rtor}{{\rm tor}} \nc{\res}{\mrm{res}} \nc{\SL}{\mrm{SL}}
\nc{\Spec}{\mrm{Spec}} \nc{\tor}{\mrm{tor}} \nc{\Tr}{\mrm{Tr}}
\nc{\mtr}{\mrm{sk}}

\nc{\ab}{\mathbf{Ab}} \nc{\Alg}{\mathbf{Alg}}

\nc{\BA}{{\mathbb A}} \nc{\CC}{{\mathbb C}} \nc{\DD}{{\mathbb D}}
\nc{\EE}{{\mathbb E}} \nc{\FF}{{\mathbb F}} \nc{\GG}{{\mathbb G}}
\nc{\HH}{{\mathbb H}} \nc{\KK}{{\mathbb K}} \nc{\LL}{{\mathbb L}}
\nc{\NN}{{\mathbb N}}
\nc{\QQ}{{\mathbb Q}} \nc{\RR}{{\mathbb R}} \nc{\BS}{{\mathbb{S}}} \nc{\TT}{{\mathbb T}}
\nc{\VV}{{\mathbb V}} \nc{\ZZ}{{\mathbb Z}}


\nc{\calao}{{\mathcal A}} \nc{\cala}{{\mathcal A}}
\nc{\calc}{{\mathcal C}} \nc{\cald}{{\mathcal D}}
\nc{\cale}{{\mathcal E}} \nc{\calf}{{\mathcal F}}
\nc{\calfr}{{{\mathcal F}^{\,r}}} \nc{\calfo}{{\mathcal F}^0}
\nc{\calfro}{{\mathcal F}^{\,r,0}} \nc{\oF}{\overline{F}}
\nc{\calg}{{\mathcal G}} \nc{\calh}{{\mathcal H}}
\nc{\cali}{{\mathcal I}} \nc{\calj}{{\mathcal J}}
\nc{\call}{{\mathcal L}} \nc{\calm}{{\mathcal M}}
\nc{\caln}{{\mathcal N}} \nc{\calo}{{\mathcal O}}
 \nc{\calr}{{\mathcal R}}
\nc{\calt}{{\mathcal T}} \nc{\caltr}{{\mathcal T}^{\,r}}
\nc{\calu}{{\mathcal U}} \nc{\calv}{{\mathcal V}}
\nc{\calw}{{\mathcal W}} \nc{\calx}{{\mathcal X}}
\nc{\CA}{\mathcal{A}}



\nc{\fraka}{{\mathfrak a}} \nc{\frakB}{{\mathfrak B}}
\nc{\frakb}{{\mathfrak b}} \nc{\frakd}{{\mathfrak d}}
\nc{\oD}{\overline{D}}
\nc{\frakF}{{\mathfrak F}} \nc{\frakg}{{\mathfrak g}}
\nc{\frakm}{{\mathfrak m}} \nc{\frakM}{{\mathfrak M}}
\nc{\frakMo}{{\mathfrak M}^0} \nc{\frakp}{{\mathfrak p}}
\nc{\frakS}{{\mathfrak S}} \nc{\frakSo}{{\mathfrak S}^0}
\nc{\fraks}{{\mathfrak s}} \nc{\os}{\overline{\fraks}}
\nc{\frakT}{{\mathfrak T}}
\nc{\oT}{\overline{T}}
\nc{\frakX}{{\mathfrak X}} \nc{\frakXo}{{\mathfrak X}^0}
\nc{\frakx}{{\mathbf x}}
\nc{\frakTx}{\frakT}      
\nc{\frakTa}{\frakT^a}        
\nc{\frakTxo}{\frakTx^0}   
\nc{\caltao}{\calt^{a,0}}   
\nc{\ox}{\overline{\frakx}} \nc{\fraky}{{\mathfrak y}}
\nc{\frakz}{{\mathfrak z}} \nc{\oX}{\overline{X}}

\font\cyr=wncyr10


\title[Extended $\mathcal{O}$-operators]{Extended $\mathcal{O}$-operators, Novikov Yang-Baxter equations and post-Novikov algebras}

\author{Jianfeng Yu}
\address{School of Mathematics, Hangzhou Normal University,
Hangzhou 311121, PR China}
\email{yujianfeng1411@163.com}

\author{Yanyong Hong}
\address{School of Mathematics, Hangzhou Normal University,
Hangzhou 311121, PR China}
\email{yyhong@hznu.edu.cn}

\subjclass[2010]{17A30, 
17B38, 
17A60, 
17D25  
}

\keywords{Novikov algebra, Novikov bialgebra, Novikov Yang-Baxter equation, $\calo$-operator}

\begin{abstract}
In this paper, we introduce the definition of extended $\mathcal{O}$-operators on a Novikov algebra $(A,\circ)$ associated to an $A$-bimodule Novikov algebra which is a generalization of the definition of $\mathcal{O}$-operators and show that there are new Novikov algebra structures on the $A$-bimodule Novikov algebra obtained from extended $\mathcal{O}$-operators. We also introduce the definition of post-Novikov algebras and show that there is a close relationship between post-Novikov algebras and $\mathcal{O}$-operators of weight $\lambda$. The tensor form of extended $\mathcal{O}$-operators is also investigated which leads to the definition of extended Novikov Yang-Baxter equations, which is a generalization of the notion of Novikov Yang-Baxter equations. The relationships between  extended $\mathcal{O}$-operators, Novikov Yang-Baxter equations, extended Novikov Yang-Baxter equations and generalized Novikov Yang-Baxter equations are established.
\end{abstract}

\maketitle

\vspace{-1.2cm}


\allowdisplaybreaks

\section{Introduction}
Novikov algebras appeared in the study of Hamiltonian
operators in the formal variational calculus \mcite{GD1, GD2} and Poisson brackets of hydrodynamic type \mcite{BN}. Moreover, by \cite{X1}, Novikov algebras correspond to a class of Lie conformal algebras which describe the singular part of operator product expansion of chiral fields in conformal field
theory \mcite{K1}.  Novikov algebras are also an important subclass of pre-Lie algebras
which have close relationships with many fields in mathematics and physics such as convex homogeneous cones \mcite{V}, affine manifolds and affine
structures on Lie groups \mcite{Ko}, deformation of
associative algebras \mcite{Ger}, vertex algebras \mcite{BLP, BK} and so on.

The definition of Novikov bialgebras  was introduced in \cite{HBG}, whose affinization can produce infinite-dimensional completed Lie bialgebras. Similar to the classical Lie bialgebra theory \cite{Dr, CP}, the notion of Novikov Yang-Baxter equations was introduced in \cite{HBG}, and the skew-symmetric solutions of the Novikov Yang-Baxter equation (NYBE) can produce Novikov bialgebras.
Moreover, the notion of $\mathcal{O}$-operators on Novikov algebras was introduced in \cite{HBG} to provide skew-symmetric solutions of the NYBE. Therefore, NYBE and $\mathcal{O}$-operators on Novikov algebras are important in the theory of Novikov bialgebras. These connections can be summarized in the following diagram.
\begin{equation*}
    \begin{split}
        \xymatrix{
        \txt{  $\mathcal{O}$-operators on\\ Novikov algebras}\ar[r] &
            \text{skew-symmetric solutions of the NYBE} \ar[r]  & \text{Novikov bialgebras}
        \vspace{-.1cm}}
    \end{split}
    \mlabel{eq:Lieconfdiag}
\end{equation*}

Motivated by the results of \cite{HBG}, it is natural to consider the general (may not skew-symmetric) solutions of the NYBE and the relationship between these solutions and their operator forms (may some generalizations of $\mathcal{O}$-operators on Novikov algebras). It should be pointed out that such problems for Lie algebras and associative algebras have been investigated in \cite{BGNlaxpair} and \cite{BGNASS} in a broad content respectively. For the case of Lie algebras,  the authors in \cite{BGNlaxpair} introduced the notion of extended $\mathcal{O}$-operators on Lie algebras in order to study double Lie algebra structures and nonabelian generalized Lax pairs. Note that  extended $\mathcal{O}$-operators on Lie algebras are generalizations of Rota-Baxter operators and $\mathcal{O}$-operators on Lie algebras \cite{STS, Ku1, Bai}. Moreover, the relationship between extended $\mathcal{O}$-operators and extended classical Yang-Baxter equations which are generalizations of classical Yang-Baxter equations was presented. By these results, it is easy to see that a general solution of the classical Yang-Baxter equation is equivalent to some extended $\mathcal{O}$-operator (see \cite{BGNlaxpair, KM}). Moreover, it was shown in \cite{BGNGCYBE} that extended $\mathcal{O}$-operators are also related to generalized classical Yang-Baxter equations which also naturally appeared in the study of Lie bialgebras. Therefore, extended $\mathcal{O}$-operators on Lie algebras are important in the theory of Lie algebras and Lie bialgebras. Similar results for associative algebras were presented in \cite{BGNASS}. Based on these results, in this paper, we plan to introduce the definitions of extended $\mathcal{O}$-operators on Novikov algebras, extended Novikov Yang-Baxter equations and generalized Novikov Yang-Baxter equations, and investigate the relationships between them. We also hope that by studying extended $\mathcal{O}$-operators on Novikov algebras, we can also obtain some new results for the NYBE as  in \cite{BGNlaxpair} for the case of Lie algebras.

Let $(A,\circ)$ be a Novikov algebra. In this paper, we first introduce the definitions of $A$-bimodule Novikov algebras, $\mathcal{O}$-operators of weight $\lambda$ on $(A,\circ)$ associated to an $A$-bimodule Novikov algebra and post-Novikov algebras. The relationships between them are investigated. We show that a post-Novikov algebra can naturally provide a bimodule Novikov algebra on a new Novikov algebra and an $\mathcal{O}$-operator of weight $\lambda$ on $(A,\circ)$ associated to an $A$-bimodule Novikov algebra can produce a post-Novikov algebra structure on the vector space of the $A$-bimodule Novikov algebra. Then we introduce the definition of extended $\mathcal{O}$-operators on $(A,\circ)$ associated to an $A$-bimodule Novikov algebra, which is a generalization of the definition of $\mathcal{O}$-operators of weight $\lambda$. It is shown that   extended $\mathcal{O}$-operators can produce new Novikov algebra structures on  the vector space of the $A$-bimodule Novikov algebra  and are also related to $A$-bimodule Novikov algebras (see Theorems \ref{delta +-} and \ref{thm:r+-}). Then we investigate the tensor forms of extended $\mathcal{O}$-operators. We introduce the definition of extended Novikov Yang-Baxter equations and investigate the relationships between extended $\mathcal{O}$-operators and extended Novikov Yang-Baxter equations in several cases (see Theorem \ref{thm:r-ENYBE}, Proposition \ref{dual exo} and Theorem \ref{thm:taa exo}). In particular, we show that  a general solution of the NYBE is equivalent to some extended $\mathcal{O}$-operator (see Corollaries \ref{cor:o-ENYBE}, \ref{cor-qN} and \ref{cor-gN}). Moreover, we introduce the definition of generalized Novikov Yang-Baxter equations which naturally appear in the study of Novikov bialgebras in \cite{HBG} and the relationship between  extended $\mathcal{O}$-operators and generalized Novikov Yang-Baxter equations is presented (see Theorem \ref{Goper con}). These results show that the study of extended $\mathcal{O}$-operators on Novikov algebras is meaningful  in the theory of
Novikov algebras and Novikov bialgebras.

This paper is  organized as follows. In Section 2, the definitions of $A$-bimodule Novikov algebras, post-Novikov algebras and extended $\mathcal{O}$-operators are introduced. The relationship between post-Novikov algebras and $\mathcal{O}$-operators of weight $\lambda$ is established. Moreover, we show that there are new Novikov algebra structures derived from extended $\mathcal{O}$-operators. Section 3 establishes the relationships between extended $\mathcal{O}$-operators, Novikov Yang-Baxter equations and extended Novikov Yang-Baxter equations. In Section 4, we introduce the definition of generalized Novikov Yang-Baxter equations and investigate the relationship between extended $\mathcal{O}$-operators and generalized Novikov Yang-Baxter equations.

Throughout this paper, let  $\bf k$ be a field of characteristic zero. All tensors over ${\bf k}$ are denoted by $\otimes$. We assume that all vector spaces or algebras in this paper are finite-dimensional, although many results hold in the infinite-dimensional case. Denote the identity map by $\id$.  For a vector space $A$, let
\begin{equation*}
\tau:A\otimes A\rightarrow A\otimes A,\qquad a\otimes b\mapsto b\otimes a,\;\;\;\; a,b\in A,
\end{equation*}
be the flip operator.

\section{Post-Novikov algebras, extended $\calo$-operators and new Novikov algebra structures}\mlabel{sec:o}
In this section, we introduce the definitions of $A$-bimodule Novikov algebras, post-Novikov algebras and extended $\calo$-operators, where $(A,\circ)$ is a Novikov algebra. The relationship between post-Novikov algebras and $\mathcal{O}$-operators of weight $\lambda$ is presented. Furthermore, we show that extended $\calo$-operators can produce new Novikov algebra structures.

\subsection{$A$-bimodule Novikov algebras}\mlabel{ss:bimod-nov}
We start with the classical notions of Novikov algebras.
\vspace{-.1cm}
\begin{defi}
\cite{GD1} A {\bf Novikov algebra} is a vector space $A$ with a binary
    operation $\circ$ satisfying
\begin{eqnarray}
                                                          \mlabel{lef} \notag
 (a\circ b)\circ c-a\circ (b\circ c)&=&(b\circ a)\circ c-b\circ (a\circ c),\\
                                                            \mlabel{Nov} \notag
 (a\circ b)\circ c&=&(a\circ c)\circ b ,  \;\;\;a, b,c\in A.
\vspace{-.1cm}
\end{eqnarray}

A Novikov algebra $(A, \circ)$ is called {\bf trivial} if for all $a$, $b\in A$, $a\circ b=0$.
\end{defi}
Let $(A,\circ)$ be a Novikov algebra. Define another binary operation $\star$ on $A$ by
\begin{eqnarray} \notag
    a\star b:=a\circ b+b\circ a,\quad \;\;\; ~~a,~~b\in A.
\end{eqnarray}
Let $L_A$, $R_A:
A\rightarrow {\rm End}_{\bf k}(A)$ be the linear maps defined by
\begin{eqnarray*}
    L_A(a)(b)\coloneqq a\circ b,\quad R_A(a)(b)\coloneqq b\circ a,\quad \;\;\;  a,~~b\in A.
\end{eqnarray*}
Define $L_{A,\star}: A\rightarrow {\rm End}_{\bf k}(A)$ by $L_{A,\star}=L_A+R_A$.

\begin{defi}
\cite{O3} A {\bf bimodule} of a Novikov algebra $(A,\circ)$ is a triple $(V, l_A,r_A)$, where $V$ is a vector space and   $l_A$, $r_A: A\rightarrow {\rm End}_{\bf k}(V)$ are linear maps satisfying
\begin{eqnarray}
                                                           \mlabel{lef-mod1}
&l_A(a\circ b-b\circ a)v=l_A(a)l_A(b)v-l_A(b)l_A(a)v,&\\
                                                           \mlabel{lef-mod2}
&l_A(a)r_A(b)v-r_A(b)l_A(a)v=r_A(a\circ b)v-r_A(b)r_A(a)v,&\\
                                                           \mlabel{Nov-mod5}
&l_A(a\circ b)v=r_A(b)l_A(a)v,
&\\
                                                           \mlabel{Nov-mod6}
&r_A(a)r_A(b)v=r_A(b)r_A(a)v, &
\hspace{-2.7cm} \quad  \;\;\;  a, b\in A, v\in V.
\vspc
\end{eqnarray}
\end{defi}

Note that $(A, L_A, R_A)$ is a bimodule of $(A,\circ)$.

\begin{defi}                              \label{bimod}
Let $(A,\circ)$ be a Novikov algebra and  $(  M,\cdot )$ be another Novikov algebra. Let $l_A$, $r_A$ : $A \rightarrow \text{End}_{\bf k}(M)$ be two linear maps. We call $M$
(or the quadruple $(M,\cdot,l_A,r_A) $) an {\bf $A$-bimodule Novikov algebra } if $(M,l_A,r_A)$ is a bimodule of $(A,\circ)$ and the following equalities hold.
\begin{eqnarray}
                                                           \label{lef-mod3}
&(l_A(a)v)\cdot w-l_A(a)(v\cdot w)=(r_A(a)v)\cdot w-v\cdot (l_A(a)w),&\\
                                                           \label{lef-mod4}
&r_A(a)(v \cdot w)-v\cdot (r_A(a)w)=r_A(a)(w\cdot v)-w\cdot (r_A(a)v),&\\
                                                           \label{Nov-mod7}
&(l_A(a)v)\cdot w=(l_A(a)w)\cdot v, &\\
                                                           \label{Nov-mod8}
&r_A(a)(v\cdot w)=(r_A(a)v)\cdot w, &
\hspace{-2.7cm} \quad \;\;\;  a\in A, v, w\in M.
\end{eqnarray}

\end{defi}
\begin{rmk}\label{rmk-bi-m}
Let $(A,\circ)$ be a Novikov algebra. Then $(A,\circ,L_A,R_A)$ is an $A$-bimodule Novikov algebra. Moreover, a bimodule $(V, l_A,r_A)$ of $(A, \circ)$ can also be seen as an $A$-bimodule Novikov algebra if we regard the vector space $V$ as the trivial Novikov algebra.  \delete{We still denote such  $A$-bimodule Novikov algebra by $(V, l_A,r_A)$.} For convenience, in the sequel, we say that $(V, l_A,r_A)$ is an $A$-bimodule Novikov algebra, which means that $(V, l_A,r_A)$ is a bimodule of $(A, \circ)$ and the Novikov algebra structure on $V$ is trivial.
\end{rmk}
\begin{pro}                 \mlabel{pro:semi}
Let $(A,\circ)$ be a Novikov algebra, $M$ be a vector space with a binary operation $\cdot$
and $l_A$, $r_A: A\to {\rm End}_{\bf k}(M)$ be linear maps.
Define a binary operation $\bullet$ on $A\oplus M$ by
\vspa
    $$(a+u)\bullet(b+v)\coloneqq a\circ b+l_A(a)v+r_A(b)u +u\cdot v,\quad\;\;\;  a,b\in A,u,v\in M.$$
Then $(M,\cdot, l_A,r_A)$ is an $A$-bimodule Novikov algebra if and only
if $(A\oplus M,\bullet)$ is a Novikov algebra.
\end{pro}
\begin{proof}
It follows directly from \cite[Theorem 3.2]{Hong}.
\end{proof}

\begin{rmk}
By Proposition \ref{pro:semi}, if the binary operation $\cdot$ on $M$ is trivial, $(M, l_A, r_A)$ is a bimodule of $(A, \circ)$ if and only if  $(A\oplus M, \bullet)$ is a Novikov algebra. In this case, we call $(A\oplus M, \bullet)$ the
{\bf semi-direct product} of $A$ by $(M, l_A, r_A)$ and denoted by
$A\ltimes_{l_A,r_A} M$.
\end{rmk}

Let $(A,\circ)$ be a Novikov algebra and $V$ be a vector space.
For a linear map $\varphi:A\rightarrow\mathrm{End}_{
\bf k}(V)$, define a linear map
$\varphi^{*}:A\rightarrow\mathrm{End}_{\bf k}(V^{*})$ by
\vspa
$$\langle\varphi^{*}(a)f,v\rangle=-\langle f,\varphi(a)v\rangle, \;\;\; a\in A, f\in V^{*},v\in V,$$
where $\langle\cdot, \cdot\rangle$ is the usual pairing between $V$ and $V^*$.

\begin{pro}\cite[Proposition 3.3]{HBG}                        \mlabel{pp:dualrep}
Let $(A,\circ)$ be a Novikov algebra and $(V, l_A, r_A)$ be a
bimodule of $(A,\circ)$. Then $(V^\ast, l_A^\ast+r_A^\ast, -r_A^\ast)$ is
a bimodule of $(A,\circ)$.
\end{pro}
\delete{\begin{proof}
It is straightforward.
\end{proof}
}
\begin{rmk}                                                \mlabel{ex:dualrep}
By Proposition \ref{pp:dualrep}, the bimodule $(A, L_A, R_A)$ of a Novikov algebra $(A,\circ)$ gives the bimodule
$(A^\ast, L_{A,\star}^\ast, -R_A^\ast)$. By Remark \ref{rmk-bi-m}, $(A^\ast, L_{A,\star}^\ast, -R_A^\ast)$  is also an $A$-bimodule Novikov algebra.
\end{rmk}

\subsection{$\calo$-operators of weight $\lambda$ and post-Novikov algebras}
\mlabel{ss:PostNov}
We first introduce the notion of $\calo$-operators of weight $\lambda$ on Novikov algebras.
\begin{defi}                                             \label{def-o}
Let $(A,\circ)$ be a Novikov algebra, $(M,\cdot,l_A,r_A)$ be an $A$-bimodule Novikov algebra and $\lambda \in \bf{k}$. A linear map $\alpha: M \rightarrow A$ is called an {\bf $\calo$-operator of weight $\lambda$} on $(A, \circ)$ associated to $(M,\cdot,l_A,r_A)$ if $\alpha$ satisfies
\begin{eqnarray}                                       \label{operation}
\alpha(u)\circ\alpha(v)=\alpha \Big( l_A(\alpha(u))v+r_A(\alpha(v)) u +\lambda u\cdot v \Big), \quad  \;\;\; u,v\in M.
\end{eqnarray}

An $\calo$-operator $T:A\rightarrow A$ of weight $\lambda$ on $(A,\circ)$ associated to  $(A,\circ,L_A,R_A)$ is called a {\bf Rota-Baxter operator of weight $\lambda$}, i.e.,
\begin{equation}  \label{T:Rota-Baxter}
T(x)\circ T(y)=T \Big( T(x)\circ y+ x\circ T(y)+\lambda x \circ y \Big)
, \quad x,y\in A.
\end{equation}
\end{defi}
\begin{rmk}
If $(M, \cdot)$ is trivial, then an $\calo$-operator of weight $\lambda$ on $(A,\circ)$ associated to $(M,\cdot,l_A,r_A)$ is just an $\calo$-operator on $(A, \circ)$ associated to the bimodule $(M, l_A, r_A)$ defined in \cite{HBG}.
\delete{when $M=V$ becomes a vector space, i.e., considering the special $A$-bimodule ($V,l,r$) with the zero multiplication, a linear map $\alpha:V \rightarrow A$ is an $\calo$-operator (of any weight $\lambda$) if $$\alpha(u)\circ \alpha(v)=\alpha \Big( l(\alpha(u))v+r(\alpha(v)) u \Big), \quad \;\;\;\; u,v \in V.$$}
\end{rmk}
\begin{rmk}
When $\lambda \neq 0$, $\alpha$ is an $\calo$-operator of weight $\lambda$ on $(A, \circ)$ associated to $(M,\cdot,l_A,r_A)$ if and only if $\alpha/\lambda$ is an $\calo$-operator of weight $1$ on $(A, \circ)$ associated to $(M,\cdot,l_A,r_A)$ if and only if $\alpha$ is an $\calo$-operator of weight $1$ on $(A, \circ)$ associated to $(M,\cdot_\lambda,l_A,r_A)$, where $\cdot_\lambda$ is defined by $u\cdot_\lambda v =\lambda u \cdot v$ for all $u$, $v \in M$.
\end{rmk}

Then  we introduce the definition of post-Novikov algebras.

\begin{defi}
A {\bf post-Novikov algebra} is a quadruple $(A,\circ,\lhd,\rhd)$, where $(A, \circ)$ is a Novikov algebra, and $\lhd$ and $\rhd$ are binary operations such that the following compatibility conditions hold.
\begin{eqnarray}
&&(a\rhd c) \lhd b =(a\rhd b +a\lhd b +a\circ b)\rhd c,     \label{ND1}
\\
&&a\rhd (c\lhd b)-(a\rhd c)\lhd b+(c\lhd a)\lhd b =c\lhd (a\lhd b+a\rhd b +a\circ b) ,                                                  \label{ND2}
\\
&&(a\rhd b +a\lhd b + a\circ b)\rhd c -(b\rhd a +b\lhd a + b\circ a)\rhd c=
a\rhd(b\rhd c)-b\rhd(a\rhd c),                              \label{ND3}
\\
&& (a\lhd b)\lhd c=(a\lhd c)\lhd b  ,                       \label{ND4}
\\
&&(a\rhd b)\circ c-a\rhd (b\circ c)=(b\lhd a)\circ c- b\circ (a\rhd c),
                                                            \label{post7}
\\
&&(b\circ c)\lhd a -b\circ(c\lhd a)=(c \circ b)\lhd a -c\circ (b\lhd a),
                                                            \label{post8}
\\
&&(a\rhd b) \circ c =(a\rhd c)\circ b,                      \label{post9}
\\
&&(a\circ b)\lhd c=(a\lhd c)\circ b,                     \label{post10}
\vspd \;\;\;a, b,c\in A.
\end{eqnarray}
\end{defi}

\begin{rmk}    \mlabel{post-2nov}
\begin{enumerate}
\item [(\romannumeral1)] If the binary operation $\circ$ in the definition of post-Novikov algebras is trivial, then a post-Novikov algebra is a pre-Novikov algebra defined in \cite{HBG}.
\item [(\romannumeral2)] \delete{Eqs~(\mref{lef}) and (\mref{Nov}) mean that $A$ is a Novikov algebra with operation $\circ$, and we denote by $(\calh(A),\circ)$. } Let $(A, \circ, \lhd, \rhd)$ be a post-Novikov algebra. It is straightforward to verify that $(A,  \circledcirc) $ is also a Novikov algebra with  the following binary operation:
\begin{eqnarray}                    \mlabel{postqquan}
a \circledcirc b := a \rhd b +a\lhd b +a\circ b,~ a,b\in A.
\end{eqnarray}
We say that $(A,\circledcirc )$ {\bf has a compatible post-Novikov algebra structure $(A,\circ,\lhd,\rhd)$} and $(A,\circledcirc )$ is {\bf the associated Novikov algebra of $(A,\circ,\lhd,\rhd)$}.
\end{enumerate}
\delete{\item [(\romannumeral3)] A {\bf homomorphism between two Post-Novikov algebras} is defined as a linear map between the two post-Novikov algebras that preserves the corresponding operations.}
\end{rmk}

Recall \cite{Lod1} that a {\bf commutative dendriform trialgebra} $(A,\cdot, \circ)$ is a
vector space $A$ with two binary operations $\cdot$ and $\circ$, where $(A,\cdot)$ is a commutative associative algebra and they satisfy the following compatibility conditions:
\begin{eqnarray}
&&(a \circ b)\circ c = a\circ(b\circ c + c\circ b+  b\cdot c),\\
&&(a \circ b)\circ c = (a \circ c)\circ b,\\
&&(a\circ b)\cdot c= a\cdot (c  \circ b),\\
&&(a\cdot b)\circ c= (a\circ c)\cdot b,
\;\;a,b,c\in A.
\end{eqnarray}

Next, we present a construction of post-Novikov algebras from  commutative dendriform trialgebras with a derivation.
\begin{pro}
Let $(A,\cdot,\circ)$ be a commutative dendriform trialgebra with a derivation $D$. Define binary operations $\ast$, $\lhd$ and $\rhd: A\times
A\rightarrow A$ by
$$a\ast b =a\cdot D(b),\;\; a\lhd b\coloneqq a\circ D(b) ,\;\;
a\rhd b\coloneqq  D(b)\circ a\;\;\; a,b\in A.$$
Then $(A,\ast,\lhd,\rhd)$ is a post-Novikov algebra.
\end{pro}
\begin{proof}
\delete{Take Eq.~(\mref{ND2}) as an example : $a\rhd (c\lhd b)-(a\rhd c)\lhd b+(c\lhd a)\lhd b =c\lhd (a\lhd b+a\rhd b +a\ast b) $.
\begin{eqnarray*}
&&a\rhd (c\lhd b)=a\rhd(c\circ D(b))=D(c\circ D(b))\circ a=(D(c)\circ D(b))\circ a + (c\circ D^2(b))\circ a.\\
&&-(a\rhd c)\lhd b=- (D(c)\circ a)\lhd b=-(D(c)\circ a)\circ D(b),\\
&&(c\lhd a)\lhd b=(c\circ D(a))\circ D(b),\\
\\
&&c\lhd (a\lhd b)=c\lhd(a\circ D(b))=c\circ D(a\circ D(b))=c\circ (D(a)\circ D(b))+c\circ(a\circ D^2(b))\\
&&c\lhd (a\rhd b)=c\lhd(D(b)\circ a)=c\circ D(D(b)\circ a)
=c\circ(D^2(b)\circ a)+c\circ(D(b)\circ D(a)),\\
&&c\lhd(a\ast b)=c\lhd (a\cdot D(b))=c\circ D(a\cdot D(b))=c\circ(D(a)\cdot D(b))+c\circ(a\cdot D^2(b)).\\
\\
&&(c\circ D^2(b))\circ a= c\circ(D^2(b)\circ a) +c\circ(a\circ D^2(b))+c\circ(a\cdot D^2(b)),(by~Eq.~(70))\\
&&(c\circ D(a))\circ D(b)=c\circ (D(a)\circ D(b))+c\circ(D(b)\circ D(a))+c\circ(D(a)\cdot D(b)) (by~Eq.~(70))
\end{eqnarray*}}
It is straightforward.

\end{proof}

\begin{lem}  \mlabel{post-abimod}
Let $(A,\circ,\lhd,\rhd)$ be a post-Novikov algebra. Define two linear maps $L_\rhd$, $R_\lhd: A\rightarrow {\rm End}_{\bf k}(A)$ by $L_\rhd(a)b:=a\rhd b$ and $R_\lhd(b)a:=a\lhd b$ for all $a$, $b\in A$. Then $(A,\circ,L_\rhd,R_\lhd)$ is a $(A,\circledcirc)$-bimodule Novikov algebra.
\end{lem}
\begin{proof}
It is straightforward.
\delete{The proof follows from a direct checking. Eqs.~(\mref{ND1})-(\mref{ND4}) show that $(\calh(A),l, r)$ is a representation of $(\calf(A),\circledcirc )$ and further Eqs.~(\mref{post7})-(\mref{post10}) mean that $(\calh(A),\circ,l,r)$ is a $(\calf(A),\circledcirc )$-bimodule Novikov algebra.}
\end{proof}

\begin{thm}                        \label{yibanpost}
Let $(A,\circ)$ be a Novikov algebra and $(M,\cdot,l_A,r_A ) $ be an $A$-bimodule Novikov algebra. Let $\alpha: M\rightarrow A$ be an $\calo$-operator of weight
$\lambda $ on $(A, \circ)$ associated to $(M,\cdot,l_A,r_A)$. Then the following conclusions hold. \delete{ i.e., it satisfies
$$\alpha(x)\circ{?}\alpha(y)=\alpha \Big( l(\alpha(x))y+r(\alpha(y))x +\lambda x\cdot{?} y \Big), \quad  \;\;\; x,y\in M.$$ }
\begin{enumerate}
\item [(\romannumeral1)] There is a post-Novikov algebra structure on $M$, where the three binary operations are defined by
\begin{eqnarray}
 x \circ y: =\lambda x \cdot y, \;\;\ x\rhd y:=l_A(\alpha(x))y, \;\;\ x\lhd y:=r_A(\alpha(y))x, \;\;\; x,y\in M.
\end{eqnarray}
Denote this post-Novikov algebra by $(M,\circ,\lhd,\rhd)$.
\item [(\romannumeral2)] $\alpha$ is a Novikov algebra homomorphism from $(M, \circledcirc)$ to $(A,\circ)$, where $\circledcirc$ is defined by Eq.~(\mref{postqquan}).
\item [(\romannumeral3)] If \text{Ker}($\alpha$) is an ideal of $(M, \cdot)$, then there exists an induced post-Novikov algebra structure on $\alpha(M)$ given by
{\small \begin{align}   \mlabel{impost}
~~~\alpha(x)\circ_\alpha \alpha(y):=\lambda \alpha(x\cdot y),~
\alpha(x)\rhd_\alpha \alpha(y):=\alpha( l_A(\alpha(x))y),~
\alpha(x)\lhd_\alpha \alpha(y):=\alpha( r_A(\alpha(y))x),~x,y\in M.
\end{align}
}
Furthermore, $\alpha$ is a homomorphism of post-Novikov algebras from $(M, \circ, \lhd, \rhd)$ to
$(\alpha(M), \circ_\alpha,$ $ \lhd_\alpha, \rhd_\alpha)$.
\end{enumerate}
\end{thm}

\begin{proof}
It is straightforward.
\delete{(\romannumeral1) It is a straightforward check and thus we just take Eq.~(\mref{ND2}) for an example. In fact, By Eq.~(\mref{lef-mod2}), for any $x,y,z\in M$, we  have
\begin{eqnarray*}
&&x\rhd (z\lhd y)-(x\rhd z)\lhd y+(z\lhd x)\lhd y -z\lhd (x\lhd y+x\rhd y +x\odot y)  \\
&=&l(\alpha(x))(r(\alpha(y))z)- r(\alpha(y))(l(\alpha(x))y)+r(\alpha(y))(r(\alpha(x))z) -r\Big(\alpha\big(l(\alpha(x))y+r(\alpha(y))x +\lambda x\cdot y\big)\Big)z\\
&=&r\Big(\alpha(x)\circ\alpha(y) \Big)z-r\Big(\alpha\big(l(\alpha(x))y+r(\alpha(y))x +\lambda x\cdot y\big)\Big)z=0.
\end{eqnarray*}
(\romannumeral2) For any $x,y\in M$, we obtain that
\begin{eqnarray*}
\alpha(x\circledcirc y)=\alpha(x \rhd y +x\lhd y +x \odot y)
=\alpha\big( l(\aax)y +r(\aay)x +\lambda x\cdot y\big)
=\aax\circ \aay.
\end{eqnarray*}
(\romannumeral3) We start with a necessary check that above multiplications defined by Eq.~(\mref{impost}) are well-defined. In fact, it follows the same argument  as one in \cite[Theorem 5.4(\romannumeral3)]{BGNlaxpair} for Lie algebra. On the other hand, we shall prove $(\alpha(M),\circ_\alpha, \lhd_\alpha, \rhd_\alpha)$ is a \postn. However, it naturally holds since $(M,\odot,\lhd,\rhd)$ is a \postn. Furthermore, the definition of Eq.~(\mref{impost}) means that last assertion holds immediately.
\\
\nc{\aaxo}{\alpha(x_1)} \nc{\aaxtw}{\alpha(x_2)}
\nc{\aayo}{\alpha(y_1)} \nc{\aaytw}{\alpha(y_2)}
Let $x_1,x_2,y_1,y_2\in M$ such that $\aaxo=\aaxtw$ and $\aayo=\aaytw$.
Since $x_1-x_2,y_1-y_2\in Ker(\alpha)$, $Ker(\alpha)$ is an ideal of
$(M,\circ)$,
\begin{eqnarray*}
&&\aaxo\rhd_\alpha \aayo =\alpha\Big(l(\aaxo)y_1\Big)
=\alpha\Big(l\big(\alpha(x_2 +x_1-x_2)\big)(y_2+y_1-y_2)\Big)\\
&=&\alpha l(\alpha(x_2)y_2)+ \alpha(l(\alpha(x_2))(y_1-y_2))\\
&=&\alpha l(\alpha(x_2)y_2)+ \alpha(x_2)\circ\alpha(y_1-y_2)-\alpha(r(\alpha(y_1-y_2))(x_2)) -\lambda \alpha((x_2)\cdot(y_1-y_2))\\
&=&\alpha l(\alpha(x_2)y_2) =\aaxtw\rhd_\alpha \aaytw.\;\;\;\aaxo\lhd_\alpha \aayo, \\
&&\aaxo\circ_\alpha \aayo =\alpha(x_2 +x_1-x_2)\big)\circ_\alpha
\alpha(y_2+y_1-y_2)=\aaxtw\circ_\alpha \aaytw.
\end{eqnarray*}}
\end{proof}

\begin{cor}
Let $(A,\cdot)$ be a Novikov algebra. Then there is a compatible post-Novikov algebra structure on $A$ if and only if there exists an $A$-bimodule Novikov algebra $(M,\cdot,l_A,r_A ) $ and an invertible $\calo$-operator of weight $\lambda=1$ on $(A, \cdot)$ associated to $(M,\cdot,l_A,r_A)$.
\end{cor}
\begin{proof}
Suppose that $(A, \cdot)$ has a compatible post-Novikov algebra structure given by $(A,\circ,\lhd,\rhd)$. Then the associated Novikov algebra $(A,\circledcirc )$ of $(A,\circ,\lhd,\rhd)$ equals to $(A,\cdot)$.
Therefore, by Lemma \mref{post-abimod}, $(A,\circ,L_\rhd,R_\lhd)$ is an $(A,\cdot)$-bimodule Novikov algebra. Since
$a \cdot b = a \rhd b +a\lhd b +a\circ b$ for all $a$, $b\in A$, we have $\id(a) \cdot \id(b) = \id\Big(L_\rhd(\id(a)) b +R_\lhd(\id(b)) a +\id(a)\circ \id(b)\Big)$. Therefore, $\id:(A,\circ,L_\rhd,R_\lhd) \rightarrow (A,\cdot)$ is an $\calo$-operator of weight $\lambda=1$ on $(A,\cdot)$ associated to $(A,\circ,L_\rhd,R_\lhd)$.

On the other hand, suppose that $\alpha$ is an invertible $\calo$-operator of weight $\lambda=1$ on $(A, \cdot)$ associated to $(M,\cdot,l_A,r_A)$. Then by Theorem \mref{yibanpost}, there exists a post-Novikov algebra structure on $\alpha(M)=A$ given by Eq.~(\mref{impost}) for $\lambda=1$. Therefore, $(\alpha(M)=A,\circ_\alpha, \lhd_\alpha, \rhd_\alpha)$ is a compatible post-Novikov algebra structure on $(A,\cdot)$.
\end{proof}

\begin{cor}         \mlabel{post-Rota}
Let $(A,\circ)$ be a Novikov algebra and $T:A\rightarrow A$ be a Rota-Baxter operator of weight $\lambda$. Then there is a post-Novikov algebra structure on $A$ given by
\begin{align}
x\odot y:=\lambda x\circ y,~
x\rhd y:= T(x)\circ y,~
x\lhd y:= x\circ T(y),~x,y\in A.
\end{align}
If $T$ is invertible in addition,  then there is a compatible post-Novikov structure on $A$ defined by
\small{\begin{align}   \mlabel{}
x\circ_T y:=\lambda T (T^{-1}(x)\circ T^{-1}(y)),~
x\rhd_T y:=T( (x   \circ (T^{-1}(y)),~
x\lhd_T y:=T ( (T^{-1}(x) \circ y),~x,y\in A.
\end{align}}
\end{cor}
\begin{proof}
It follows from Theorem \ref{yibanpost} directly.
\end{proof}
\subsection{Extended $\calo$-operators}
                                                        \mlabel{ss:operator}
\delete{For exploring the connections between extended $\calo$-operators and Novikov bimodule structures, we will first state the following concepts about $\calo$-operators on Novikov algebras.}

\begin{defi}     \label{def:beta and operator}
Let $(A,\circ)$ be a Novikov algebra and $(M,\cdot,l_A,r_A) $ be an $A$-bimodule Novikov algebra. Let $\kappa,\mu$ and $\lambda$ be constants in $\bf k$.
\begin{enumerate}
\item[(\romannumeral1)] A linear map $\beta:M \rightarrow A$ is called {\bf balanced } if
\begin{equation}\label{b:con1}
 l_A(\beta(u))v = r_A(\beta(v))u,\quad   u, v \in M.
\end{equation}
\item[(\romannumeral2)] A linear map $\beta:M \rightarrow A$ is called {\bf $A$-invariant of mass $\kappa$} if
\begin{equation}\label{b:con2}
\kappa \Big( x \circ \beta(u) \Big) = \kappa \beta\Big( l_A(x) u \Big), \quad \kappa \Big( \beta(u) \circ x \Big) = \kappa \beta\Big( (r_A(x) u ) \Big) ,           \quad   x \in A, u\in M.
\end{equation}
In particular, if $\beta$ is $A$-invariant of mass $\kappa\neq 0$, then $\beta$ is an $A$-bimodule homomorphism from $(M, l_A,r_A)$ to $(A, L_A, R_A)$. If $\beta$ is balanced and $\beta$ is an $A$-bimodule homomorphism from $(M, l_A,r_A)$ to $(A, L_A, R_A)$, we call that $\beta$ is a {\bf balanced $A$-bimodule homomorphism}.
\item[(\romannumeral3)] A linear map $\beta:M \rightarrow A$ is called {\bf equivalent  of mass  $\mu$} if
\begin{equation}\label{b:con3}
\mu l_A\Big( \beta(u \cdot v) \Big) w = \mu \Big( l_A(\beta(u)v) \Big)\cdot w ,
\quad
\mu r_A\Big( \beta(v \cdot w) \Big) u=  \mu u\cdot \Big(r_A(\beta(w)v) \Big) ,
\quad   u, v, w \in M.
\end{equation}
\item[(\romannumeral4)] Let $\alpha$, $\beta:M \rightarrow A$ be linear maps and $\beta$ be balanced, $A$-invariant of mass $\kappa$, and equivalent of mass $\mu$.  $\alpha$ is called an {\bf extended $\calo$-operator of weight $\lambda$ with extension $\beta$ of mass $(\kappa,\mu)$} on $(A, \circ)$ associated to $(M,\cdot,l_A,r_A)$ if
\begin{equation}  \label{a ext b}
\alpha(u)\circ \alpha(v)-\alpha \Big( l_A(\alpha(u))v+r_A(\alpha(v)) u+\lambda u\cdot v \Big)
=\kappa \beta(u) \circ \beta(v)+  \mu \beta(u \cdot v)
,\quad  u,v\in M.
\end{equation}
\delete{A {\bf Rota-Baxter Novikov algebra} of weight $\lambda$ is a Novikov algebra equipped with a Rota-Baxter operator of weight $\lambda$.}
\end{enumerate}
\begin{rmk}
The parameters $\kappa$, $\mu$ and $\lambda$ in the definition are included in order to
uniformly treat the different cases when the parameters vary.

If $\beta=0$, then an extended $\calo$-operator of weight $\lambda$ with extension $\beta$ of mass $(\kappa,\mu)$ is just the $\calo$-operator of weight $\lambda$ defined in Definition \mref{def-o}.

Note that if $(M, \cdot, l_A, r_A)=(A, \circ, L_A, R_A)$, then $\beta=\id$ is balanced, $A$-invariant of mass $\kappa$ and equivalent of mass  $\mu$.

It should be also pointed out that if $\beta$ is an $A$-bimodule homomorphism from $(M, l_A,r_A)$ to $(A, L_A, R_A)$, then Eq. (\ref{b:con2}) also holds for any $\kappa$. Therefore, in this case, $\beta$ is also $A$-invariant of mass $\kappa$.
\end{rmk}
\end{defi}


Let $(A,\circ)$ be a Novikov algebra and  $(M,\cdot,l_A,r_A) $ be an $A$-bimodule Novikov algebra. Let $\delta_\pm:M \rightarrow A$ be two linear maps and $\lambda \in \bf k$. Define a new binary operation on $M$ as follows.
\begin{equation}                           \mlabel{double Nov}
u\diamond v:=l_A(\delta_+(u))v+r_A(\delta_-(v))u+\lambda u\cdot v,
\quad   u, v \in M.
\end{equation}
Set
\begin{equation}                               \mlabel{a and b}
\alpha:=\frac{\delta_++\delta_-}{2}, \quad \beta=\frac{\delta_+-\delta_-}{2},
\end{equation}
which is called the {\bf symmetrizer} and {\bf antisymmetrizer} of $\delta_\pm$ respectively. Note that $\delta_\pm$ can be recovered from $\alpha$ and $\beta$ by $\delta_+=\alpha+\beta$ and $\delta_-=\alpha-\beta$.

\begin{pro}                                             \label{pro:*}
Let $(A,\circ)$ be a Novikov algebra and $(M,\cdot,l_A,r_A)  $ be an $A$-bimodule Novikov algebra. Let $\alpha : M \rightarrow A $ be a linear map and $\lambda \in \bf k$. Then the following binary operation
\begin{equation} \mlabel{*Nov}
u*v:=l_A(\alpha(u))v+r_A(\alpha(v))u+\lambda u\cdot v,\quad   u, v \in M,
\end{equation}
gives a new Novikov algebra structure on $M$ if and only if the following equalities hold:
\begin{eqnarray}                                        \mlabel{*con1}
&&l_A\Big(\alpha(u*v)-\alpha(u)\circ\alpha(v)\Big)w=l_A\Big(\alpha(u*w)-\alpha(u)\circ\alpha(w)\Big)v,\\            \mlabel{*con2}
&&l_A\Big( \alpha(u*v)- \alpha(u)\circ\alpha(v)\Big)w  -
l_A \Big( \alpha(v*u)- \alpha(v)\circ\alpha(u)\Big) w
\\&&\qquad= \notag
r_A\Big(   \alpha(v*w) - \alpha(v)\circ\alpha(w)  \Big) u-
r_A\Big(   \alpha(u*w) - \alpha(u)\circ\alpha(w)  \Big) v,
\qquad  u, v, w \in M.
\end{eqnarray}
\end{pro}

\begin{proof}
For all $u$, $v$, $w \in M$, we have
\begin{eqnarray*}                    
(u*v)*w&=&l_A\big(\alpha(u*v) \big)w+r_A(\alpha(w)) (u*v)+\lambda(u*v)\cdot w
\\
&=&l_A\big(\alpha(u*v) \big)w+
  r_A(\alpha(w))\big(l_A(\alpha(u))v\big)+  r_A(\alpha(w) ) \big( r( \alpha(v) ) u  \big)                           \\
   &&\quad+\lambda    r_A(\alpha(w)) \big( u\cdot v\big)
     +\lambda  \big(l_A(\alpha(u))v\big)\cdot w+
     \lambda  \big( r_A(\alpha(v))  u     \big)\cdot w+
     \lambda^2   \big( u\cdot v\big)\cdot w,\\               
u*(v*w)&=&l_A(\alpha(u))(v*w) + r_A\big(\alpha(v*w) \big)  u     +\lambda u\cdot(v*w)
                  \nonumber\\
 &=&l_A(\alpha(u))\big(l_A(\alpha(v))w\big)   +   l_A(\alpha(u))\big(  r_A(\alpha(w)) v  \big)  +   \lambda l_A(\alpha(u))(v\cdot w)       \nonumber\\
    &&\quad+     r_A\big(\alpha(v*w) \big)      u
   +\lambda u\cdot \big(l_A(\alpha(v))w\big)
   +\lambda u\cdot \big(  r_A(\alpha(w)) v    \big)
    +\lambda^2u\cdot\big( v\cdot w\big) .
\end{eqnarray*}

\delete{Similarly, we get
\begin{eqnarray*}                  
(v*u)*w&=&l_A\big(\alpha(v*u) \big)w+
r_A(\alpha(w)) \big(l_A(\alpha(v))u\big)  +  r_A(\alpha(w))  \big(  r_A(\alpha(u)) v \big)
                    \nonumber\\
&&\quad +\lambda  r_A(\alpha(w))    \big( v\cdot u\big)
+\lambda\big(l_A(\alpha(v))u\big)\cdot w+
\lambda\big(  r_A(\alpha(u)) v     \big)\cdot w+
\lambda^2\big( v\cdot u\big)\cdot w  ,\\                       
v*(u*w)&=&l_A(\alpha(v))\big(l_A(\alpha(u))w\big)+ l_A(\alpha(v))  \big( r_A(\alpha(w)) u  \big)+\lambda l_A(\alpha(v))(u\cdot w)  \nonumber\\
&&\quad+   r_A\big(\alpha(u*w) \big) v
+\lambda v\cdot \big(l_A(\alpha(u))w\big)
+\lambda v\cdot \big(  r_A(\alpha(w))u   \big)
+\lambda^2v\cdot \big( u\cdot w\big) .
\end{eqnarray*}}

Therefore, by Eqs. (\mref{lef-mod2}), (\mref{lef-mod3}),  (\mref{lef-mod4}), and (\mref{lef-mod1}), we obtain
\small{
\begin{eqnarray*}
&&(u*v)*w-u\ast(v\ast w)-(v*u)*w+v*(u*w)
\\
&&\quad=\Bigg( r_A(\alpha(w))\big(l_A(\alpha(u))v\big) - l_A(\alpha(u))\big(  r_A(\alpha(w)) v  \big)   -r_A(\alpha(w))  \big(  r_A(\alpha(u)) v \big)                           \Bigg)
  +\Bigg( r_A(\alpha(w)) \big(r_A(\alpha(v) u)\big)
\\
  &&\quad \quad\quad-r_A(\alpha(w))        \big(l_A(\alpha(v))u\big) + l_A(\alpha(v))\big(r_A(\alpha(w)) u\big)               \Bigg)
  +\lambda \Bigg(  \big(l_A(\alpha(u))v\big)\cdot w- l_A(\alpha(u))(v\cdot w)
\\
  &&\quad\quad\quad-\big(  r_A(\alpha(u)) v\big)\cdot w + v\cdot \big(l_A(\alpha(u))w\big)      \Bigg)
  +\lambda \Bigg( r_A(\alpha(w)) \big( u\cdot v\big) - u\cdot \big(  r_A(\alpha(w)) v    \big) -r_A(\alpha(w))    \big( v\cdot u\big)
\\
  &&\quad\quad\quad  +v\cdot\big(  r_A(\alpha(w))u   \big)                                                                        \Bigg)+\lambda \Bigg( \big( r_A(\alpha(v))  u\big)\cdot w -u\cdot \big(l_A(\alpha(v))w\big)  -\big(l_A(\alpha(v))u\big)\cdot w  +l_A(\alpha(v))(u\cdot w )
\Bigg)
\\
  &&\quad\quad\quad +\lambda^2 \Bigg( (u\cdot v)\cdot w - u\cdot (v\cdot w)-(v\cdot u)\cdot w+ v\cdot(u\cdot w)                                 \Bigg)
-\Bigg( l_A(\alpha(u))\big(l_A(\alpha(v))w\big)
\\
 &&\quad\quad\quad  -l_A(\alpha(v))\big(l_A(\alpha(u))w\big)                  \Bigg) + l_A\big(\alpha(u*v) \big)w - l_A\big(\alpha(v*u) \big)w
        -r_A\big(\alpha(v*w) \big)u   +r_A\big(\alpha(u*w) \big) v
\\
&&\quad=-r_A\big(\alpha(u)\circ \alpha(w) \big) v + r_A\big(\alpha(v)\circ \alpha(w) \big)  u
-l_A\Bigg( \alpha(u)\circ \alpha(v) -  \alpha(v)\circ \alpha(u)               \Bigg)w
+ l_A\big(\alpha(u*v) \big)w
\\
 &&\quad\quad\quad - l_A\big(\alpha(v*u) \big)w
        -r_A\big(\alpha(v*w) \big)u   +r_A\big(\alpha(u*w) \big) v
\\
&&\quad=l_A\Big( \alpha(u*v)- \alpha(u)\circ\alpha(v)\Big)w  -
l_A \Big( \alpha(v*u)- \alpha(v)\circ\alpha(u)\Big) w -
r_A\Big(   \alpha(v*w) - \alpha(v)\circ\alpha(w)  \Big) u
\\
&&\quad\quad\quad +r_A\Big(   \alpha(u*w) - \alpha(u)\circ\alpha(w)  \Big) v.
\end{eqnarray*}}

\delete{On the other hand, we can similarly have
\begin{align*}                    
(u*v)*w=&l_A\big(\alpha(u*v) \big)w+
r_A(\alpha(w))  \big(l_A(\alpha(u))v\big)+r_A(\alpha(w)) \big(r_A(\alpha(v))  u \big) \\
&+\lambda  r_A(\alpha(w)) \big( u\cdot v\big)
+\lambda\big(l_A(\alpha(u))v\big)\cdot w+
\lambda\big(r_A(\alpha(v))u\big)\cdot w+
\lambda^2\big( u\cdot v\big)\cdot w  ,
\end{align*}

Thus,
\begin{align*}                
(u*w)*v=&l_A\big(\alpha(u*w) \big)v+
 r_A(\alpha(v))  \big(l_A(\alpha(u))w\big) + r_A(\alpha(v)) \big( r_A(\alpha(w))  u \big) \\
&+\lambda   r_A(\alpha(v))  \big( u\cdot w\big)
+\lambda  \big(l_A(\alpha(u))w\big)\cdot v+
\lambda     \big(     r_A(\alpha(w)) u \big)\cdot v+
\lambda^2     \big( u\cdot w\big)\cdot v.
\end{align*}

\zhushi{
\begin{eqnarray*}                
(u*w)*v&&=l\big(\alpha(u*w) \big)v+
 r(\alpha(v))  \big(l(\alpha(u))w\big) + r(\alpha(v)) \big( r(\alpha(w))  u \big) \\
&&+\lambda   r(\alpha(v))  \big( u\cdot{?}w\big)
+\lambda  \big(l(\alpha(u))w\big)\cdot{?}v+
\lambda     \big(     r(\alpha(w)) u \big)\cdot{?}v+
\lambda^2     \big( u\cdot{?}w\big)\cdot{?}v.
\end{eqnarray*}
}}
Similarly, by Eqs. (\mref{Nov-mod6}), (\mref{Nov-mod8}), (\mref{Nov-mod7}), and (\mref{Nov-mod5}), we obtain
\small{
\begin{eqnarray*}
&&(u*v)*w-(u*w)*v=\Bigg( r_A(\alpha(w)) \big(r_A(\alpha(v))  u \big)-r_A(\alpha(v)) \big( r_A(\alpha(w))  u \big)  \Bigg)
+\lambda\Bigg( r_A(\alpha(w)) \big( u\cdot v\big) \\&&-  \big(r_A(\alpha(w)) u \big)\cdot v\Bigg)
+\lambda\Bigg( \big(l_A(\alpha(u))v\big)\cdot w -  \big(l_A(\alpha(u))w\big)\cdot v  \Bigg)
+\lambda\Bigg( \big(r_A(\alpha(v))u\big)\cdot w -r_A(\alpha(v))  \big( u\cdot w\big)          \Bigg)\\
&&+\lambda^2\Big( \big( u\cdot v\big)\cdot w -\big( u\cdot w\big)\cdot v          \Big)
+l_A\big(\alpha(u*v) \big)w +  r_A(\alpha(w))  \big(l_A(\alpha(u))v\big)-
l_A\big(\alpha(u*w) \big)v \\&&- r_A(\alpha(v))  \big(l_A(\alpha(u))w\big)
=l_A\big(\alpha(u*v) \big)w +  l_A\Big(\alpha(u) \circ \alpha(w)\Big)v
-l_A\big(\alpha(u*w) \big)v - l_A\Big(\alpha(u) \circ \alpha(v)\Big)w.
\end{eqnarray*}}
Then the proof is completed.
\end{proof}

\begin{rmk}
If $\alpha$ is an $\calo$-operator of weight $\lambda$ on $(A, \circ)$ associated to $(M,\cdot , l_A, r_A)$, then Eqs. (\mref{*con1}) and (\mref{*con2}) hold naturally, which means that $(M, \ast)$ defined by Eq.~(\mref{*Nov}) is a Novikov algebra.
\end{rmk}

\begin{cor}   
Let $\left(A,\circ \right)$ be a Novikov algebra and $(M,\cdot,l_A,r_A)  $ be an $A$-bimodule Novikov algebra. Let $\delta_\pm:M \rightarrow A$ be two linear maps   and  $\alpha$, $\beta$ be the linear maps defined by Eq.~(\mref{a and b}). If $\beta$ is balanced, then $(M,\diamond)$ defined by Eq. (\ref{double Nov}) is a Novikov algebra if and only if $\alpha$ satisfies Eqs. (\mref{*con1}) and (\mref{*con2}).
\end{cor}

\begin{proof}
Since $\beta$ is balanced, for all $u, v \in M$, we have
\begin{eqnarray*}
&&u\diamond v=l_A(\delta_+(u))v+r_A(\delta_-(v)) u+\lambda u\cdot v
\\
&=&l_A(\alpha(u))v+r_A(\alpha(v))  u+\lambda  u\cdot v +l_A(\beta(u))v - r_A(\beta(v))u
=l_A(\alpha(u))v+r_A(\alpha(v))  u+\lambda  u\cdot v.
\end{eqnarray*}
Then this conclusion follows directly from Proposition \mref{pro:*}.
\vspb
\end{proof}

\begin{thm}                            \mlabel     {delta +-}
Let $(A,\circ )$ be a Novikov algebra and $(M,\cdot,l_A,r_A)  $ be an $A$-bimodule Novikov algebra. Let $\delta_\pm:  M \rightarrow A$ be two linear maps, $\lambda$, $\kappa$, $\mu\in {\bf k}$, and  $\alpha$, $\beta$ be the linear maps defined by Eq. (\mref{a and b}). Then the following conclusions hold.
\begin{enumerate}
\item[(\romannumeral1)] Suppose that $\alpha$ is an extended $\calo$-operator of weight $\lambda$ with extension $\beta$ of mass $(\kappa,\mu)$ on $(A, \circ)$ associated to $(M,\cdot,l_A,r_A)$. Then $(M, \ast)$ defined by Eq. (\ref{*Nov}) is  a Novikov algebra.

\item[(\romannumeral2)] Suppose that $\beta$ is an $A$-bimodule homomorphism from $(M, l_A, r_A)$ to $(A, L_A, R_A)$. Then $\alpha$ satisfies Eq. (\mref{a ext b}) for $\kappa = -1$ and $\mu = \pm \lambda$ if and only if the following equalities hold:
\begin{equation}
\delta_\pm(u)\circ \delta_\pm(v)- \delta_\pm ( u*v)=0,\quad   u, v\in M.
\end{equation}
\end{enumerate}
\end{thm}

\begin{proof}
(\romannumeral1) By Proposition \mref{pro:*}, we only need to check that Eqs.~(\mref{*con1}) and (\mref{*con2}) hold. For Eq.~(\mref{*con1}), it is enough to prove that
$$l_A\Big( \kappa\beta(u)\circ\beta(v) + \mu \beta(u\cdot v)\Big)w
=l_A\Big( \kappa\beta(u)\circ\beta(w) + \mu \beta(u\cdot w)\Big)v.$$
In fact, we will prove that
$$\kappa l_A\Big( \beta(u)\circ\beta(v) \Big)w=\kappa l_A\Big( \beta(u)\circ\beta(w) \Big)v,\;\;\;\mu l_A\Big(  \beta(u\cdot v)\Big)w  =\mu l_A\Big( \beta(u\cdot w)\Big)v.$$
For all $u, v, w \in M$, by Eqs. (\mref{Nov-mod5}), (\mref{b:con1}) and (\mref{Nov-mod6}), we have
\begin{align*}
\kappa l_A\Big( \beta(u)\circ \beta(v) \Big)w
&=\kappa r_A(\beta(v) )   \Big( l_A(\beta(u))w \Big)
=\kappa r_A(\beta(v) )   \Big( r_A(\beta(w)) u\Big)
\\
&=\kappa r_A(\beta(w) ) \Big(   r_A(\beta (v) ) u \Big)
=\kappa r_A(\beta(w) )  \Big( l_A(\beta(u))v \Big)
=\kappa l_A\Big( \beta(u)\circ \beta(w) \Big)v    .
\end{align*}

Moreover, by Eqs.~(\mref{b:con3}) and (\ref{Nov-mod7}), we obtain
\begin{eqnarray*}
\mu l_A\Big(  \beta(u\cdot v)\Big)w
=\mu \Big(l_A  (\beta(u))   v \Big)\cdot w
=\mu \Big(l_A  (\beta(u))   w \Big)\cdot v
=\mu l_A\Big(  \beta(u\cdot w)\Big)v       .
\end{eqnarray*}
Therefore Eq.~(\mref{*con1}) holds.

Next we prove Eq.~(\mref{*con2}). In fact, it is enough to prove
\begin{eqnarray*}
l_A\bigg(   \Big( \kappa\beta(u)\circ\beta(v) + \mu \beta(u\cdot v)\Big)  -
 \Big( \kappa\beta(v)\circ\beta(u) + \mu \beta(v\cdot u)\Big)   \bigg)w         =\\ \nonumber
r_A\Big(   \kappa\beta(v)\circ\beta(w) + \mu \beta(v\cdot w)  \Big)  u-
r_A\Big(   \kappa\beta(u)\circ\beta(w) + \mu \beta(u\cdot w) \Big)  v     .
\end{eqnarray*}
Similarly, we will prove that
$$
\kappa l_A  \Big( \beta(u)\circ\beta(v) \Big)w  -
\kappa l_A\Big( \beta(v)\circ\beta(u) \Big)   w         = 
\kappa r_A\Big(   \beta(v)\circ\beta(w)   \Big)u-
\kappa r_A\Big(   \beta(u)\circ\beta(w)  \Big)v  ,
$$
and
$$\mu l_A\Big(    \beta(u\cdot v)\Big) w -
\mu l_A\Big(   \beta(v\cdot u)\Big)   \Big)w=         
\mu r_A\Big(      \beta(v\cdot w)  \Big)u-
 \mu r_A\Big(    \beta(u\cdot w) \Big)v     .
$$
For all $u, v, w \in M$, by Eqs. (\mref{lef-mod1}), (\mref{b:con1}) and (\mref{b:con2}), we have
\begin{eqnarray*}
&&\kappa l_A  \Big( \beta(u)\circ\beta(v) \Big)w  -
\kappa l_A \Big( \beta(v)\circ\beta(u) \Big)   w
=\kappa l_A(\beta(u)) ( l_A(\beta(v))w )-\kappa l_A(\beta(v))(l_A(\beta(u))w  )
\\
&&=\kappa  r_A \bigg(\beta  \Big(l_A(\beta(v))w  \Big)  \bigg)  u- \kappa r_A \bigg(\beta  \Big(l_A(\beta(u))w  \Big)  \bigg)  v    
=\kappa r_A\Big(   \beta(v)\circ\beta(w)   \Big)   u-
\kappa   r_A\Big(   \beta(u)\circ\beta(w)  \Big)  v         .
\end{eqnarray*}
Moreover, by Eqs.~(\mref{b:con3}), (\mref{b:con1}) and (\mref{lef-mod3}), we have
\begin{eqnarray*}
&& \mu l_A\Big(    \beta(u\cdot v)\Big) w -  \mu l_A\Big(   \beta(v\cdot u)\Big)   \Big)  w        
=\mu \Big(l_A( \beta(u)) v\Big)\cdot w- \mu \Big( l_A(\beta(v)) u\Big)\cdot w
\\&&=\mu \Big(l_A(\beta(u))  v\Big)\cdot w- \mu \Big(  r_A(\beta(u))  v \Big)\cdot w
=\mu l_A(\beta(u)) (v \cdot w)  - \mu v \cdot(   l_A(\beta(u)) w  )
\\&&= \mu r_A\Big(      \beta(v\cdot w)  \Big) u-\mu v\cdot\Big( r_A(\beta(w)) u  \Big)
= \mu r_A\Big(    \beta(v\cdot w)  \Big) u- \mu r_A\Big(    \beta(u\cdot w) \Big) v.
\end{eqnarray*}
Therefore, Eq.~(\mref{*con2}) holds and the proof of Item (\romannumeral1) is completed.

(\romannumeral2) By Eq.~(\mref{b:con2}), we obtain
\begin{eqnarray*}
&&(\alpha\pm\beta)(u)\circ (\alpha\pm\beta)(v)- (\alpha\pm\beta) ( u\ast v )
\\&=&\alpha(u)\circ \alpha(v) +     \beta(u)\circ \beta(v)
- \alpha \Big( l_A(\alpha(u))v+r_A(\alpha(v))u+\lambda u\cdot v \Big) \\
&& \mp \bigg(  \beta \Big( l_A(\alpha(u))v\Big)     - \alpha(u)\circ \beta(v) \Bigg)
\mp \bigg(  \beta \Big( r_A(\alpha(v)) u \Big)      - \beta(u)\circ \alpha(v) \Bigg)
\mp \beta \Big(\lambda u\cdot v \Big)
      \\
&=&\alpha(u)\circ \alpha(v)-          \alpha \Big( l_A(\alpha(u))v+r_A(\alpha(v))u+\lambda u\cdot v \Big)+       \beta(u)\circ \beta(v)
\mp\lambda \beta \Big( u\cdot v \Big).
\end{eqnarray*}
Thus Item (\romannumeral2) holds.
\end{proof}

\begin{thm}                         \mlabel{thm:r+-}
Let $\left(A,\circ \right)$ be a Novikov algebra and  $\left(M,\cdot,l_A,r_A\right)  $ be an $A$-bimodule Novikov algebra. Let $\delta_\pm:M \rightarrow A$ be two linear maps and $\alpha$, $\beta$ be the linear maps defined by Eq.~(\mref{a and b}). Suppose that $\beta$ is balanced, $A$-invariant of mass $\kappa \neq 0$ and equivalent of mass $\lambda$. Then the following conclusions hold.
\begin{enumerate}
\item[(\romannumeral1)] $(M_\pm,\cdot_\pm,l_A,r_A) $ are $A$-bimodule Novikov algebras, where $(M_\pm,\cdot_\pm ) $ are the new Novikov algebra structures on $M$ defined by
\begin{equation}                               \mlabel{+-def}
u \cdot_\pm v := \lambda u \cdot v \mp 2l_A \left( \beta(u) \right)v,
\qquad  u,v \in M.
\end{equation}  
\item[(\romannumeral2)] $\alpha$ is an extended $\calo$-operator of weight $\lambda$ with extension $\beta$ of mass $(-1,\pm \lambda)$  if and only if $\delta_\pm:M_\pm \rightarrow A$ is an $\calo$-operator of weight $1$ on $(A, \circ)$ associated to $(M_\pm,\cdot_\pm,l_A,r_A) $, where $\cdot_\pm$ are defined by Eq. (\mref{+-def}).
\end{enumerate}
\end{thm}

\begin{proof}
(\romannumeral1) First we prove that $(M_\pm, \cdot_\pm)$ are  Novikov algebras. In fact, for all $u, v, w\in M$, by Eqs.~(\mref{b:con1}), (\mref{Nov-mod7}), (\mref{b:con3}) and (\mref{b:con2}), we have
\begin{align*}                    
(u \cdot_\pm v)\cdot_\pm w
&= \lambda^2(u \cdot v)\cdot w   \mp    2\lambda \Big( l_A(\beta(u) )v \Big)\cdot w
   \mp    2\lambda l_A \Big(  \beta(u \cdot v) \Big)w
   +4  l_A \bigg( \beta \Big( l_A(   \beta(u)  )v    \Big) \bigg)w
\\&=\lambda^2(u \cdot w)\cdot v   \mp   2\lambda \Big(l_A (\beta(u) )w \Big)\cdot v
   \mp    2\lambda  \Big(  l_A(\beta(u))v \Big)   \cdot w
   +4  l_A \Big( \beta (u) \circ \beta (v) \Big)w      
\\&=\lambda^2(u \cdot v)\cdot w   \mp   2\lambda \Big(l_A (\beta(u) )w \Big)\cdot v
   \mp    2\lambda  l_A \Big(  \beta(u \cdot w) \Big)v
   +4 l_A \Big( \beta (u) \circ \beta (w) \Big)v       \\ 
&=(u \cdot_\pm w)\cdot_\pm v.
\end{align*}

Similarly, by Eqs.~(\mref{b:con1}) and (\mref{b:con3}), we have
\begin{align*}                    
(v \cdot_\pm u)\cdot_\pm w
&= \lambda^2(v \cdot u)\cdot w  \mp    4\lambda \Big( l_A (\beta(v) )u \Big)\cdot w
   +4   l_A \Big( \beta(v)\circ \beta(u)     \Big)w\\
&=\lambda^2(v \cdot u)\cdot w  \mp  4\lambda \Big( r_A (\beta(u) )v \Big)\cdot w
   +4  l_A \Big( \beta(v)\circ \beta(u)     \Big)w.
\end{align*}

Moreover, by Eqs.~(\mref{b:con1}), (\mref{b:con3}) and (\mref{b:con2}), we get
\begin{align*}                     
u\cdot_\pm(v \cdot_\pm w)
&= \lambda^2 u\cdot(v \cdot w) \mp    2\lambda u \cdot \Big( l_A (\beta(v) )w \Big)
   \mp    2\lambda l_A(\beta(u)) (  v \cdot w)
   +4 l_A(\beta(u))  \Big( l_A(   \beta(v)  )w    \Big)       \\
&=\lambda^2 u\cdot(v \cdot w) \mp    2\lambda u \cdot \Big( r_A (\beta(w) )v \Big)
   \mp    2\lambda l_A(\beta(u)) (  v \cdot w)
   +4   l_A(\beta(u))  \Big( l_A(   \beta(v)  )w    \Big)      \\
&=\lambda^2 u\cdot(v \cdot w) \mp    2\lambda  r_A\Big( \beta(v\cdot w)\Big)u
   \mp    2\lambda l_A(\beta(u)) (  v \cdot w)
   +4  l_A(\beta(u))  \Big( l_A(   \beta(v)  )w    \Big)       \\
&=\lambda^2 u\cdot(v \cdot w) \mp    4\lambda l_A(\beta(u)) (v\cdot w)
   +4  l_A(\beta(u))  \Big( l_A(   \beta(v)  )w    \Big)  .
\end{align*}

Similarly, we have
\begin{align*}                    
v\cdot_\pm(u \cdot_\pm w)
&=\lambda^2 v\cdot(u \cdot w)   \mp  4\lambda l_A(\beta(v)) (  u \cdot w)
   +4  l_A(\beta(v))  \Big( l_A(   \beta(u)  )w    \Big) \\
&=\lambda^2 v\cdot(u \cdot w)   \mp  4\lambda v \cdot \Big( l_A (\beta(u) )w \Big)
   +4 l_A(\beta(v))  \Big( l_A(   \beta(u)  )w    \Big) .
\end{align*}

Therefore, since $\beta$ is balanced, $A$-invariant of mass $\kappa\neq 0$ and equivalent of mass $\lambda$, we obtain
{\small
\begin{eqnarray*}
&&(u \cdot_\pm v)\cdot_\pm w -u\cdot_\pm(v \cdot_\pm w) - (v\cdot_\pm u) \cdot_\pm w +v\cdot_\pm(u \cdot_\pm w)
\\
&&=\lambda^2 \Bigg((u \cdot v)\cdot w-u\cdot(v \cdot w)-  (v \cdot u)\cdot w
            +v\cdot(u \cdot w) \Bigg)
\mp  4\lambda   \Bigg( \Big(l_A (\beta(u) )v \Big)\cdot w- l_A(\beta(u)) (v\cdot w)
\\
&&\quad  -\Big( r_A (\beta(u) )v \Big)\cdot w+v \cdot \Big( l_A (\beta(u) )w \Big)     \Bigg)
+4\lambda \Bigg( l_A \Big( \beta (u) \circ \beta (v) \Big)w
           -l_A \Big( \beta(v)\circ \beta(u)     \Big)w
\\&&\quad           - l_A(\beta(u))  \Big( l_A(   \beta(v)  )w    \Big)
          + l_A(\beta(v))  \Big( l_A(\beta(u))w    \Big) \Bigg)=0,
\end{eqnarray*}}
where the last equality follows from Eqs.~(\mref{lef-mod3}) and (\mref{lef-mod1}). Therefore, $(M_\pm,\cdot_\pm)$ are Novikov algebras.

Next, we shall prove that $(M_\pm,\cdot_\pm,l_A,r_A) $ are $A$-bimodule Novikov algebras. Since $(M,\cdot,l_A,r_A) $ is an $A$-bimodule Novikov algebra, $(M_\pm, l_A,r_A)$ are naturally bimodules of $(A,\circ)$. Hence, we just need to prove that $(M_\pm,\cdot_\pm,l_A,r_A) $ satisfy Eqs.~(\mref{lef-mod3})-(\mref{Nov-mod8}).
In fact, for all $a \in A, v, w \in M$, by Eqs.~(\mref{b:con1}), (\mref{Nov-mod7}), (\mref{Nov-mod5}) and (\mref{b:con2}), we have
\begin{align*}
(l_A(a)v)\cdot_\pm w    &=\lambda (l_A(a)v)\cdot w  \mp 2l_A\Big( \beta(l_A(a)v)\Big)w\\
      &=\lambda (l_A(a)v)\cdot w  \mp 2r_A( \beta(w) ) (l_A(a)v)
              \\ 
      &=\lambda (l_A(a)w)\cdot v  \mp 2l_A\Big( a\circ \beta(w) \Big)v
             \\ 
      &=\lambda (l_A(a)w)\cdot v  \mp 2l_A\Big( \beta(l_A(a)w) \Big)v
              \\ 
&=(l_A(a)w)\cdot_\pm v.
\end{align*}
Moreover, by Eqs.~(\mref{Nov-mod7}), (\mref{Nov-mod5}) and (\mref{b:con2}), we have
\begin{align*}
r_A(a)(v\cdot_\pm w)&=\lambda r_A(a)(v\cdot w)  \mp  2r_A(a)\Big(  l_A(\beta(v))w \Big)\\
       &=\lambda (r_A(a)v)\cdot w  \mp  2l_A\Big( \beta(v)\circ a \Big)w
              \\
       &=\lambda (r_A(a)v)\cdot w  \mp  2l_A\Big(  \beta(r_A(a)v) \Big)w
               \\
&=(r_A(a)v)\cdot_\pm w.
\end{align*}

Furthermore, by Eqs.~(\mref{lef-mod3}), (\mref{lef-mod2}), (\mref{b:con1}) and  (\mref{b:con2}), we obtain
{\small
\begin{eqnarray*}
&&(l_A(a)v)\cdot_\pm w  - l_A(a)(v\cdot_\pm w) -(r_A(a)v)\cdot_\pm w +v\cdot_\pm (l_A(a)w)
\\
&=&\lambda \Big((l_A(a)v)\cdot w -l_A(a)(v\cdot w) -(r_A(a)v)\cdot w   +v\cdot (l_A(a)w)
\Big) \mp 2\Bigg( l_A\Big(\beta(l_A(a)v)\Big)w
      \\&&\quad -l_A(a)\Big(l_A(\beta(v))w\Big)
                -l_A\Big( \beta(r_A(a)v)\Big)w   + l_A(\beta(v))(l_A(a)w)         \Bigg)\\
&=&\mp 2\Bigg(  r_A(\beta(w)) (l_A(a)v) -l_A(a)\Big(r_A(\beta(w))v\Big)-r_A(\beta(w))(r_A(a)v)
+r\Big(\beta(l_A(a)w)  \Big)v\Bigg)     \\
&=&\mp 2\Bigg(  r_A(\beta(w)) (l_A(a)v) -l_A(a)\Big(r_A(\beta(w))v\Big)-r_A(\beta(w))(r_A(a)v)
+r_A\Big(a\circ \beta(w)  \Big)v\Bigg) =0,  
\end{eqnarray*}
and
\begin{eqnarray*}
&&r_A(a)(v\cdot_\pm w) -v\cdot_\pm(r_A(a)w)-r_A(a)(w\cdot_\pm v) +w\cdot_\pm (r_A(a)v)
\\
&=&\lambda \Big( r_A(a)(v\cdot w) -v\cdot(r_A(a)w)  -r_A(a)(w\cdot v) +w\cdot (r_A(a)v)
\Big)\mp 2\Bigg( r_A(a)\Big( l_A(\beta(v))w\Big)
\\
&& \quad  -l_A(\beta(v))(r_A(a)w)-r_A(a)\Big(l_A(\beta(w))v\Big) +l_A(\beta(w))(r_A(a)v)      \Bigg) \\
&=&\mp 2\Bigg( r_A(a)\Big( l_A(\beta(v))w\Big) - l_A(\beta(v))(r_A(a)w)
           -r_A(a)\Big(r_A(\beta(v))w\Big) + r_A\Big( \beta(r_A(a)v)   \Big)w      \Bigg)
           \\
&=&\mp 2\Bigg( r_A(a)\Big( l_A(\beta(v))w\Big) - l_A(\beta(v))(r_A(a)w)
           -r_A(a)\Big(r_A(\beta(v))w\Big) +   r_A(\beta(v)\circ a)w     \Bigg)=0.
\end{eqnarray*}}
Then the proof of Item (\romannumeral1) is completed.


\noindent
(\romannumeral2) The last conclusion follows from Theorem \mref{delta +-} (\romannumeral2) and the following equality:
{\small
\begin{eqnarray*}
&&\delta_\pm(u)\circ \delta_\pm(v)-
\delta_\pm \Big( l_A(\delta_\pm(u))v+r_A(\delta_\pm(v)) u + u\cdot_\pm v \Big)
\\&&=\delta_\pm(u)\circ \delta_\pm(v)- \delta_\pm \Big( l_A(\delta_\pm(u))v+r_A(\delta_\pm(v)) u + u\cdot v \mp 2l_A(\beta(u))v\Big) \\
&&=\delta_\pm(u)\circ \delta_\pm(v)- \delta_\pm \Big(
l_A(\alpha(u))v+r_A(\alpha(v)) u + u\cdot v \pm l_A(\beta(u))v \mp l_A(\beta(u))v
\pm r_A(\beta(v)u)  \mp l_A(\beta(u))v        \Big) \\
&&=\delta_\pm(u)\circ \delta_\pm(v)- \delta_\pm \Big( u*v \Big)
.
\end{eqnarray*}}
\end{proof}
\begin{cor}                             \mlabel{cor:r+-}
Let $(A,\circ )$ be a Novikov algebra and $(V,l_A,r_A)  $ be an $A$-bimodule Novikov algebra. Let $\beta : V \rightarrow A$ be a linear map such that $\beta$ is balanced and $A$-invariant of mass $\kappa \neq 0$. Then the following conclusions hold.
\begin{enumerate}
\item[(\romannumeral1)] $(V_\pm,\cdot_\pm,l_A,r_A)  $ are $A$-bimodule Novikov algebras, where $(V_\pm,\cdot_\pm ) $ are the new Novikov algebra structures on $V$ defined by
\begin{equation}                               \mlabel{V:+-def}
u \cdot_\pm v :=   \mp 2l_A\left( \beta(u) \right)v,
\qquad  u,v \in V.
\end{equation}
\item[(\romannumeral2)] Let $\alpha : V\rightarrow A$ be a linear map. Then $\alpha$ is an extended $\calo$-operator of weight $\lambda$ with extension $\beta$ of mass $(\kappa=-1, \pm \lambda) $ on $(A,\circ)$ associated to $(V,l_A,r_A)  $ if and only if $\alpha\pm\beta:V_\pm \rightarrow A$ is an $\calo$-operator of weight $1$ on $(A,\circ)$ associated to $(V_\pm,\cdot_\pm,l_A,r_A)$, where $\cdot_\pm$ are defined by Eq. (\mref{V:+-def}).
\end{enumerate}
\end{cor}
\begin{proof}
It follows from  Theorem \mref{thm:r+-} directly.
\end{proof}

\mlabel{ss:Baxter}
Finally, we apply the above conclusions to the case when $(M, \cdot, l_A, r_A)=(A,\circ,L_A,R_A)$. By Proposition \mref{pro:*}, Theorems \mref{delta +-} and \mref{thm:r+-}, we immediately obtain the following conclusion.
\delete{i.e. the following equality holds:
\begin{eqnarray}    \mlabel{bimodule homo}
\beta(x)\circ y= x\circ\beta(y),\;\;\kappa \beta(x\circ y)=\kappa \beta(x)\circ y,\quad  x,y\in A.
\end{eqnarray}
\yy{Please check.}}
\begin{pro}\label{pro:cons}
Let $(A,\circ)$ be a Novikov algebra and $T,\beta :A\rightarrow A$ be two linear maps. Let $\beta$ be balanced and $A$-invariant of mass $\kappa$  from $(A, L_A, R_A)$ to $(A, L_A, R_A)$. Suppose that $T$ is an extended $\calo$-operator of weight $\lambda$ with extension $\beta$ of mass $(\kappa, 0)$, i.e., the following equality holds:
\begin{eqnarray}    \label{ex Bax}
T(x)\circ T(y)-T \Big( T(x)\circ y+x\circ T(y) +\lambda x\circ y \Big)
=\kappa \beta(x) \circ \beta(y)
, \quad x,y\in A.
\end{eqnarray}
Then the binary operation
$$x\circ_T y=T(x)\circ y+x\circ T(y)+\lambda x\circ y,\quad  x,y\in A,$$
defines a Novikov algebra structure on $A$.~In particular, if $\kappa \neq 0$, then $(A_\pm,\circ_\pm,L_A,R_A)$ are $A$-bimodule Novikov algebras, where $(A_\pm,\circ_\pm)$ are the new Novikov algebra structures defined by
\begin{equation}                               \mlabel{A-A:+-def}
x \circ_\pm y := \lambda x\circ y \mp 2  \beta(x) \circ y,
\qquad  x,y \in A.
\end{equation}
Moreover, $T$ is an extended $\calo$-operator of weight $\lambda$ with extension $\beta$ of mass $(\kappa=-1,0) $ on $(A, \circ)$ associated to $(A, \circ, L_A, R_A)$, i.e.,  Eq.~(\mref{ex Bax}) holds for $\kappa=-1$, if and only if $T\pm \beta : A_\pm \rightarrow A$ are $\calo$-operator of weight $1$ on $(A, \circ)$ associated to $(A_\pm,\circ_\pm,L_A,R_A)$, where $\circ_\pm$ are defined by Eq.~(\mref{A-A:+-def}).
\end{pro}

\begin{ex}                                             
Let $(A,\circ)$ be a two-dimensional Novikov algebra in \cite{BaiMeng} with a basis $\lbrace e_1,e_2\rbrace$ whose multiplication is given by
$$e_1\circ e_1=e_1,~e_1\circ e_2=e_2,~e_2\circ e_1=e_2,~e_2\circ e_2=0.$$
Let $T$ and $\beta : A \rightarrow A$ be two linear maps satisfying the following conditions:
$$ T(e_1)=-2e_1+4e_2,
~T(e_2)=e_2,
~\beta(e_1)=e_1+3e_2,
~ \beta(e_2)=e_2.$$
Then it is direct to verify that $\beta$ is balanced  and $A$-invariant of mass $\kappa$, and $T$ is an extended $\calo$-operator of weight  $\lambda=1$ with extension $\beta$ of mass $\kappa=-2$ on $(A, \circ)$ associated to $(A,\circ,L_A,R_A)$.\delete{ i.e. the following equality holds:
\begin{eqnarray*}
T(x)\circ T(y)-T \Big( T(x)\circ y+x\circ T(y) +x \circ y \Big)
=-2\beta(x) \circ \beta(y)
, \quad x,y\in A.
\end{eqnarray*}}
Moreover, the binary operation
$$e_1\circ_T e_1=-3e_1+8e_2,
~e_1\circ_T e_2=0,
~e_2\circ_T e_1=0,
~e_2\circ_T e_2=0,$$ defines a Novikov algebra structure on $A$.
On the other hand, $(A_\pm,\circ_\pm,L_A,R_A)$ are $A$-bimodule Novikov algebra, where $(A_\pm,\circ_\pm)$ are the new Novikov algebra structures defined by
$$e_1\circ_\pm e_1=e_1 \mp2(e_1+3e_2),
~e_1\circ_\pm e_2=e_2 \mp2e_2,
~e_2\circ_\pm e_1=e_2 \mp 2e_2,
~e_2\circ_\pm e_2=0.
$$
\end{ex}
\delete{\begin{proof}
$$e_1\circ_T e_1=T(e_1)\circ e_1 +e_1 \circ T(e_1) +e_1\circ e_1=(-2e_1+4e_2)\circ e_1+e_1\circ (-2e_1+4e_2)+e_1       =-3e_1+8e_2,$$
$$e_1\circ_T e_2=T(e_1)\circ e_2 +e_1 \circ T(e_2)+ e_1\circ e_2=(-2e_1+4e_2)\circ e_2 + e_1\circ e_2 +e_2          =0,$$
$$e_2\circ_T e_1=T(e_2)\circ e_1 +e_2 \circ T(e_1)+e_2\circ e_1=  e_2\circ e_1+e_2\circ (-2e_1+4e_2)   +e_2        =0,$$
$$e_2\circ_T e_2=T(e_2)\circ e_2 +e_2 \circ T(e_2)+e_2\circ e_2= e_2\circ e_2+e_2\circ e_2       =0$$
On the other hand,
$$e_1\circ_\pm e_1 =e_1\circ e_1 \mp2\beta(e_1)\circ e_1 =e_1 \mp 2(e_1+3e_2) \circ e_1=e_1 \mp2(e_1+3e_2),$$
$$e_1\circ_\pm e_2 =e_1\circ e_2 \mp2\beta(e_1)\circ e_2 =e_2\mp 2(e_1+3e_2)\circ e_2=e_2 \mp2e_2$$
$$e_1\circ_\pm e_2 =e_2\circ e_1 \mp2\beta(e_2)\circ e_1 =e_2\mp 2e_2\circ e_1=e_2 \mp 2e_2$$
$$e_2\circ_\pm e_2 =e_2\circ e_2 \mp2\beta(e_2)\circ e_2 =\mp 2e_2\circ e_2=0.$$
\end{proof}}

\delete{\begin{rmk}
Let $(A,\circ)$ be a Novikov algebra. If a linear endomorphism $\beta$ of $A$ is a balanced $A$-bimodule homomorphism of mass $\kappa \neq 0$, i.e. it satisfies Eq.~(\mref{bimodule homo}) for $\kappa \neq 0$, then it is called an {\bf averaging operator}, that is,
\begin{equation}                               \mlabel{averaging operator}
\beta(x)\circ\beta(y) = \beta(x\circ \beta(y))=\beta(\beta(x)\circ y),
\qquad  x,y \in A.
\end{equation}
and it is also a {\bf Nijenhuis operator}, namely,
\begin{equation}                               \mlabel{Nijenhuis operator}
\beta(x)\circ\beta(y) +\beta^2(x\circ y)= \beta(x\circ \beta(y) +\beta(x)\circ y),
\qquad  x,y \in A.
\end{equation}
\end{rmk}}

Let $(A,\circ)$ be a Novikov algebra and $(A,\circ,L_A,R_A)$ be the $A$-bimodule Novikov algebra. It is obvious that $\beta=\id: A\rightarrow A$ is balanced, $A$-invariant of mass of $\kappa \neq 0$ and equivalent of mass $\mu$. In this case, Eq.~(\mref{a ext b}) is of the following form
\begin{equation}                       \mlabel{T and hkappa}
T(x)\circ T(y)-T \Big( T(x)\circ y+ x\circ T(y)+\lambda x\circ y \Big)
=\kappa x\circ y +\mu x\circ y
, \quad x,y\in A.
\end{equation}
Let $\hat{\kappa}=\kappa+\mu$. Thus, by Theorem \mref{thm:r+-}, we obtain the following corollary.

\begin{cor}                             \mlabel{cor:Bax}
Let $\hat{\kappa}=-1\pm\lambda$. Then $T$ satisfies Eq.~(\mref{T and hkappa}) if and only if $T\pm \id$ is a Rota-Baxter operator of weight $\lambda \mp 2$.
\end{cor}
Note that when $\lambda=0$, Eq.~(\mref{T and hkappa}) is of the following form:
\begin{equation}                       \mlabel{T and kappa}
T(x)\circ T(y)-T \Big( T(x)\circ y+ x\circ T(y) \Big)
=\hat{\kappa} x\circ y
, \quad x,y\in A.
\end{equation}
When  $\hat{\kappa}=-1$, Eq.~(\mref{T and kappa}) becomes
\begin{eqnarray}    \label{Bax bb=id k=-1}
T(x)\circ T(y)-T \Big( T(x)\circ y+x\circ T(y) \Big)
=- x \circ y
, \quad x,y\in A.
\end{eqnarray}
A Novikov algebra equipped with a linear endomorphism satisfying Eq.~(\mref{Bax bb=id k=-1}) is called a {\bf Baxter Novikov algebra}. Furthermore, $T$ satisfies Eq.~(\mref{Bax bb=id k=-1}) if and only if $T\pm \id$ is a Rota-Baxter operator of weight $\mp2$. 

\begin{pro} Let $(A,\circ, T)$ be a Baxter Novikov algebra. Then there is a post-Novikov algebra structure on $A$ given by
\begin{align}
x\odot y:= x\circ y,~
x\rhd y:= \dfrac{(T\pm \id)}{\mp2}(x)\circ y,~
x\lhd y:= x\circ  \dfrac{(T\pm \id)}{\mp2}(y),\;\;x,y\in A.
\end{align}
\end{pro}
\begin{proof}
It follows immediately from Corollary \mref{post-Rota}, since $\dfrac{(T\pm \id)}{\mp2}$ is a Rota-Baxter operator of weight $1$.
\end{proof}
\delete{
\begin{ex}                                             
Let $(A,\circ)$ be a two-dimensional Novikov algebra in \cite{BaiMeng} with a basis $\lbrace e_1,e_2\rbrace$ whose multiplication is given by
$$e_1\circ e_1=e_1,~e_1\circ e_2=e_2,~e_2\circ e_1=e_2,~e_2\circ e_2=0.$$
Let $T$ and $\beta : A \rightarrow A$ be two linear maps satisfying the following conditions:
$$ T(e_1)=-2e_1+4e_2,
~T(e_2)=e_2,
~\beta(e_1)=e_1+3e_2,
~ \beta(e_2)=e_2.$$
Then it is direct to verify that $\beta$ is balanced  and $A$-invariant of mass $\kappa$, and $T$ is an extended $\calo$-operator of weight  $\lambda=1$ with extension $\beta$ of mass $\kappa=-2$ on $(A, \circ)$ associated to $(A,\circ,L_A,R_A)$.\delete{ i.e. the following equality holds:
\begin{eqnarray*}
T(x)\circ T(y)-T \Big( T(x)\circ y+x\circ T(y) +x \circ y \Big)
=-2\beta(x) \circ \beta(y)
, \quad x,y\in A.
\end{eqnarray*}}
Moreover, the binary operation
$$e_1\circ_T e_1=-3e_1+8e_2,
~e_1\circ_T e_2=0,
~e_2\circ_T e_1=0,
~e_2\circ_T e_2=0,$$ defines a Novikov algebra structure on $A$.
On the other hand, $(A_\pm,\circ_\pm,L_A,R_A)$ are $A$-bimodule Novikov algebra, where $(A_\pm,\circ_\pm)$ are the new Novikov algebra structures defined by
$$e_1\circ_\pm e_1=e_1 \mp2(e_1+3e_2),
~e_1\circ_\pm e_2=e_2 \mp2e_2,
~e_2\circ_\pm e_1=e_2 \mp 2e_2,
~e_2\circ_\pm e_2=0.
$$
\end{ex}
\delete{\begin{proof}
$$e_1\circ_T e_1=T(e_1)\circ e_1 +e_1 \circ T(e_1) +e_1\circ e_1=(-2e_1+4e_2)\circ e_1+e_1\circ (-2e_1+4e_2)+e_1       =-3e_1+8e_2,$$
$$e_1\circ_T e_2=T(e_1)\circ e_2 +e_1 \circ T(e_2)+ e_1\circ e_2=(-2e_1+4e_2)\circ e_2 + e_1\circ e_2 +e_2          =0,$$
$$e_2\circ_T e_1=T(e_2)\circ e_1 +e_2 \circ T(e_1)+e_2\circ e_1=  e_2\circ e_1+e_2\circ (-2e_1+4e_2)   +e_2        =0,$$
$$e_2\circ_T e_2=T(e_2)\circ e_2 +e_2 \circ T(e_2)+e_2\circ e_2= e_2\circ e_2+e_2\circ e_2       =0$$
On the other hand,
$$e_1\circ_\pm e_1 =e_1\circ e_1 \mp2\beta(e_1)\circ e_1 =e_1 \mp 2(e_1+3e_2) \circ e_1=e_1 \mp2(e_1+3e_2),$$
$$e_1\circ_\pm e_2 =e_1\circ e_2 \mp2\beta(e_1)\circ e_2 =e_2\mp 2(e_1+3e_2)\circ e_2=e_2 \mp2e_2$$
$$e_1\circ_\pm e_2 =e_2\circ e_1 \mp2\beta(e_2)\circ e_1 =e_2\mp 2e_2\circ e_1=e_2 \mp 2e_2$$
$$e_2\circ_\pm e_2 =e_2\circ e_2 \mp2\beta(e_2)\circ e_2 =\mp 2e_2\circ e_2=0.$$
\end{proof}}}

\section{Tensor forms of extended $\calo$-operators}
\mlabel{sec:nybe}

In this section, we introduce the definition of extended Novikov Yang-Baxter equations, which is a generalization of Novikov Yang-Baxter equations defined in \cite{HBG}.  The relationships between extended Novikov Yang-Baxter equations and extended $\calo$-operators in several cases are investigated. Moreover, we show that there are also close relationships between Novikov Yang-Baxter equations and extended $\calo$-operators.



\subsection{Extended Novikov Yang-Baxter equation}
\mlabel{ss:NYBE}               

We first recall the definition of Novikov Yang-Baxter equations given in \cite{HBG}.


\delete{
\begin{pro}          \mlabel{Nov:bialg}
Let $(A,\circ)$ be a Novikov algebra and $r\in A\otimes A$. In the
following, we consider a special class of Novikov bialgebras $(A,
\circ, \Delta_r)$ when $\Delta_r: A\rightarrow A\otimes A$ is
defined by
\begin{eqnarray}    \label{co1}     
\Delta_r(a)\coloneqq (L_A(a)\otimes \id+\id\otimes L_{A,\star}(a))r\quad \tforall  a\in A.
\end{eqnarray}
\end{pro}}

Let $(A, \circ)$ be a Novikov algebra.
Let $r=\sum_i x_i \otimes y_i \in A\otimes A$ and
$r'=\sum_j x_j '\otimes y_j '\in A\otimes A$. Set
$$r_{12}\circ r'_{13}=\sum_{i,j}x_i \circ x_j'\otimes y_i\otimes
y_j',\;\; r_{12}\circ r'_{23}=\sum_{i,j} x_i\otimes y_i\circ
x_j'\otimes y_j',\;\;r_{13}\circ r'_{23}=\sum_{i,j} x_i\otimes
x_j'\otimes y_i\circ y_j',$$
$$r_{13}\circ r'_{12}=\sum_{i,j}x_i\circ x_j'\otimes y_j'\otimes
y_i,\;\;r_{23}\circ r'_{13}=\sum_{i,j} x_j'\otimes x_i\otimes
y_i\circ y_j',$$
$$r_{12}\star r'_{23}=\sum_{i,j} x_i\otimes y_i\star
x_j'\otimes y_j',\;\;r_{13}\star r'_{23}=\sum_{i,j} x_i\otimes
x_j'\otimes y_i\star y_j'.$$

\delete{
\zhushi{\begin{lem}\mlabel{coblem1}\mlabel{lem:cob2}
Let $(A,\circ)$ be a Novikov algebra and $r\in A\otimes A$. Define $\Delta_r: A\rightarrow A\otimes A$ by Eq.~\meqref{co1}. Then $(A,\circ, \Delta_r)$ is a Novikov bialgebra if and only if the following equalities hold.

\begin{enumerate}
\item                 \mlabel{it:coba}
$(\id\otimes (L_A(b\circ a)+L_A(a)L_A(b))+L_{A,\star}(a)\otimes L_{A,\star}(b))(r+\tau r)=0\;\;\tforall  a,b\in A.$
\item               \mlabel{it:cobb}
$(L_{A,\star}(a)\otimes L_{A,\star}(b)-L_{A,\star}(b)\otimes L_{A,\star}(a))(r+\tau r)=0\;\;\tforall  a,b\in A.$
\item               \mlabel{it:cobc}
{\small

                 $
\Big(-L_{A,\star}(b)\otimes R_A(a)+L_{A,\star}(a)\otimes R_A(b)+R_A(a)\otimes L_A(b)-R_A(b)\otimes L_A(a)+\id\otimes \big(L_A(a)L_A(b)\\
 \ -L_A(b)L_A(a)\big)-\big(L_A(a)L_A(b)-L_A(b)L_A(a)\big)\otimes \id\Big)(r+\tau r)=0 \tforall  a,b\in A.$
}
\item \mlabel{it:cobd}
{\small
$
\Big(L_A(a)\otimes \id\otimes \id-\id\otimes L_A(a)\otimes \id\Big)\Big((\tau r)_{12}\circ r_{13}+r_{12}\circ r_{23}+r_{13}\star r_{23}\Big)\nonumber\\
  \hspace{0.3cm}+\Big((\id\otimes L_A(a)\otimes
\id)(r+\tau r)_{12}\Big)\circ r_{23}-\Big((L_A(a)\otimes \id\otimes \id)r_{13}\Big)\circ (r+\tau r)_{12}+\Big(\id\otimes \id\otimes
L_{A,\star}(a)\Big)\\
  \hspace{0.3cm}\Big(r_{23}\circ r_{13}-r_{13}\circ
r_{23}-(\id\otimes\id\otimes \id-\tau\otimes \id)(r_{13}\circ
r_{12}+r_{12}\star r_{23})\Big)=0~~\tforall  a\in A.
$
}
\item \mlabel{it:cobe}

          $  (\id\otimes \id\otimes \id-\id\otimes \tau)(\id\otimes \id\otimes
            L_{A,\star}(a))(r_{13}\circ (\tau r)_{23} -r_{12}\star
            r_{23}-r_{13}\circ r_{12})=0\;\;\tforall  a\in A.$

\end{enumerate}
\end{lem}
}

\smallskip

\noindent

\zhushi{
\begin{thm}\mlabel     {thm:bialg}
Let $(A,\circ)$ be a Novikov algebra and $r\in A\otimes A$. Define
$\Delta_r: A\rightarrow A\otimes A$ by Eq.~{\rm (\mref{co1})}. Then
 $(A,\circ, \Delta_r)$ is a Novikov bialgebra if and
only if Eqs.~(\mref{it:coba})-(\mref{it:cobe}) hold.
\vspb
\end{thm}
}

\begin{cor}\mlabel     {cor:bialg}
Let $(A,\circ)$ be a Novikov algebra and $r\in A\otimes A$ be
skewsymmetric. Define $\Delta_r: A\rightarrow A\otimes A$ by
Eq.~{\rm (\mref{co1})}. Then
 $(A,\circ, \Delta_r)$ is a Novikov bialgebra if and only if the following
equalities hold.
{\small
\begin{eqnarray}
&&(L_A(a)\otimes \id\otimes \id-\id\otimes L_A(a)\otimes
\id)(\id\otimes \tau) (r\diamond r) +(\id\otimes
\id\otimes L_{A,\star}(a))(r\diamond r-(\tau\otimes \id)(r\diamond r))=0,\label{cob8}\\
&&\label{cob9}(\id\otimes \id\otimes \id- \id\otimes \tau)(\id\otimes
\id\otimes L_{A,\star}(a))(r\diamond r)=0\;\;  a\in A,
\end{eqnarray}}
where
\vspd
\begin{eqnarray} \notag
r\diamond r\coloneqq r_{13}\circ r_{23} +r_{12}\star r_{23}+r_{13}\circ
r_{12}.
\vspb
\end{eqnarray}
In particular, if $r\diamond r=0$, then  $(A,\circ, \Delta_r)$ is
a Novikov bialgebra.
\end{cor}
}

\begin{defi}   \cite{HBG}
Let $(A,\circ)$ be a Novikov algebra and $r\in A\otimes A$. The
equation
\begin{eqnarray}  \label{NYBE}
r\diamond r\coloneqq r_{13}\circ r_{23} +r_{12}\star r_{23}+r_{13}\circ
r_{12}=0
\end{eqnarray}
is called the {\bf Novikov Yang-Baxter equation (NYBE)} in $A$.
\end{defi}

\zhushi{
\begin{defi}
Let $(A,\circ)$ be a Novikov algebra. Fix $\epsilon \in \bf k $. The
equation
\vspb        
\begin{eqnarray}  \label{ENYBE}
r_{13}\circ r_{23} +r_{12}\star r_{23}+r_{13}\circ r_{12}=
\epsilon(r_{13} + r_{31}) \circ (r_{23} + r_{32})
\vspb
\end{eqnarray}
is called the {\bf extended Novikov Yang-Baxter equation (ENYBE)} in $A$.
\end{defi}
Note that when $\epsilon = 0$ or $r$ is skew-symmetric in the sense that $\tau(r)=-r$, then the ENYBE of mass $\epsilon$ is the same as the NYBE in Eq.~(\mref{NYBE})
}


\begin{pro}   \cite[Proposition 2.3]{BGNGCYBE}   \mlabel{hat and dual}
Let $V$ and $W$ be two vector spaces.
\begin{enumerate}
\item
Define the  natural linear isomorphisms as follows:
\begin{eqnarray*}
\wedge=\wedge_{V,W} :  V\ot W \rightarrow \Hom_{\bf k}(V^*,W),\;\; \langle w^\ast, \hr(v^\ast)\rangle= \langle v^\ast\otimes w^\ast,r\rangle,  \\
\vee=\vee_{\Hom_{\bf k}(V,W)}:  \Hom_{\bf k}(V,W)\rightarrow V^*\ot W,\;\; \langle \caa, v\otimes w^\ast\rangle=\langle w^\ast, \alpha(v)\rangle,
\end{eqnarray*}
where $r\in V\ot W$, $ v^\ast\in V^\ast$, $w^\ast\in W^\ast$, $\alpha \in \Hom_{\bf k}(V,W)$ and $v\in V$.
 Then $\wedge_{V,W}$ and $\vee_{\Hom_{\bf k}(V^\ast,W)}$ are the inverse of each other.
 \item For $\alpha\in\Hom_{\bf k}(V^*,W)$, let $\alpha^* :W^\ast\rightarrow V^{\ast\ast}\cong V$ be the corresponding dual linear map of $\alpha$. Then for any $r\in V\ot W$, we have $\widehat{\tau(r)}=\hr^*$.
\end{enumerate}
\end{pro}

Using the above notations, for a vector space $A$, we set
\begin{eqnarray}
r_\pm :=(r\pm \tau(r))/2,\;\; \alpha_\pm :=(\alpha\pm\alpha^*)/2,\;\; \;\;
r\in A\ot A,~\alpha\in \Hom_{\bf k}(A^*,A).
\end{eqnarray}
Moreover, for any $r\in A\ot A$, we have $(\hr)_\pm=\hat{r_\pm}$. 

By Proposition \mref{hat and dual}, we can identify $r\in A\ot A$ with the linear map $\hr: A^*\rightarrow A$ by
\begin{equation}\mlabel{rhat}
\langle \hr (a^*), b^*\rangle=\langle a^* \otimes b^*, r\rangle,
\quad  a^*,b^*\in A.
\end{equation}
Similarly define a linear map $\hr^t : A^*\rightarrow A$ by
\begin{equation}\mlabel{rt}
\langle a^*, \hr^t(b^*)\rangle=\langle a^* \otimes b^*, r\rangle,
\quad  a^*,b^*\in A,
\end{equation}
where $\hr^t$ is induced by $\tau(r)$.
Note that  $r$ is skew-symmetric or symmetric in $ A\ot A $ if and only if $\hr=-\hr^t$ or $\hr=\hr^t$. Define
\begin{equation}         \mlabel{r:a and b}
\alpha=\hr_-=(\hr-\hr^t)/2,\quad \beta=\hr_+=(\hr+\hr^t)/2,
\end{equation}
which are called the {\bf skew-symmetric part} and the {\bf symmetric part} of $\hr$ respectively. Then $\hr=\alpha + \beta=\hr_-+\hr_+$ and $\hr^t=-\alpha+\beta=-\hr_-+\hr_+$.


\begin{lem}\label{lem:r}
Let $(A,\circ)$ be a Novikov algebra and $s\in A\otimes A $ be symmetric. Then the following conditions are equivalent.

\begin{enumerate}
\item[(\romannumeral1)] s is {\bf invariant}, i.e.,
\begin{equation}\label{r:con1}
(L_A(x)\otimes \id+\id\otimes L_{A,\star}(x))s=0, \quad   x\in A.
\end{equation}
\item[(\romannumeral2)] $\hs:(A^*,L_{A,\star}^*, -R_A^*)\rightarrow A$ is balanced. \delete{i.e.
\begin{equation}\label{r:con2}
L_{A,\star}^*(\hs(a^*))b^*=-R_A^*(\hs(b^*)) a^*,\quad   a^*,b^*\in A^*.
\end{equation}}
\item[(\romannumeral3)] $\hs:(A^*,L_{A,\star}^*, -R_A^*)\rightarrow (A,L_A,R_A)$ is an $A$-bimodule homomorphism.
\delete{ i.e.
\begin{equation}\label{r:con3}
\hs(L_\star^*(x) a^*)=x\circ \hs(a^*),\quad \hs(-R^*(x) a^*)=\hs(a^*)\circ x,\quad   a^*\in A^*,x\in A.
\end{equation}}
\end{enumerate}
\vspb
\end{lem}

\begin{proof}
First, we show that Item (\romannumeral1) is equivalent to Item (\romannumeral2). Since $s\in A\ot A $ is symmetric, for all $x\in A$, $a^*,b^* \in A^*$, we get
\begin{eqnarray*}
\langle (L_A(x)\otimes \id+\id\otimes L_{A,\star}(x))s ,a^*\ot b^*  \rangle &&=
-\langle s,L_A^*(x) a^* \ot b^* \rangle-\langle s, a^* \ot L_{A,\star}^*(x)b^* \rangle
\\&&=\langle x\circ \hs(b^*),a^*\rangle + \langle x\circ \hs(a^*)+ \hs(a^*)\circ x, b^* \rangle
\\&&=\langle R_A(\hs(b^*))x,a^*\rangle + \langle L_{A,\star}\hs(a^*))x, b^* \rangle
\\&&=-\langle x,R_A^*(\hs(b^*))a^*\rangle -\langle x, L_{A,\star}^*(\hs(a^*))b^* \rangle
\\&&=\langle x,(-R_A^*)(\hs(b^*))a^* -L_{A,\star}^*(\hs(a^*))b^*\rangle .
\end{eqnarray*}
Thus $s$ is invariant if and only if $\hs:(A^*,L_{A,\star}^*, -R_A^*)\rightarrow A$ is balanced.

Next, we show that Item (\romannumeral1) is equivalent to Item (\romannumeral3). For all $x\in A$, $a^*,b^* \in A^*$, we obtain
\begin{eqnarray*}
\langle (L_A(x)\otimes \id+\id\otimes L_{A,\star}(x))s ,a^*\ot b^*  \rangle &&=
-\langle s,L_A^*(x) a^* \ot b^* \rangle-\langle s, a^* \ot L_{A,\star}^*(x)b^* \rangle
\\&&=-\langle \hs(b^*),L_A^*(x)a^*\rangle - \langle \hs(L_{A,\star}^*(x)b^*), a^* \rangle
\\&&=\langle x\circ \hs(b^*)- \hs(L_{A,\star}^*(x)b^*), a^* \rangle,
\end{eqnarray*}
and
\begin{eqnarray*}
\langle (L_A(x)\otimes \id+\id\otimes L_{A,\star}(x))s ,a^*\ot b^*  \rangle &&=
-\langle s,L_A^*(x) a^* \ot b^* \rangle-\langle s, a^* \ot L_{A,\star}^*(x)b^* \rangle
\\&&=-\langle \hs(L_A^*(x)a^*),b^*\rangle - \langle \hs(a^*), L_{A,\star}^*(x)b^* \rangle
\\&&=\langle -\hs(L_A^*(x)a^* )+ x\circ \hs(a^*)+\hs(a^*)\circ x, b^* \rangle.
\end{eqnarray*}
\delete{by the symmetry of $s$. Moreover, the third assertion $``\hs(-R^*(x) a^*)=\hs(a^*)\circ x"$ follows immediately according to the definition of $L_{\star}^*=L_A^*+R_A^*$.} Then it is easy to see that $s$ is invariant if and only if $\hs:(A^*,L_{A,\star}^*,$ $ -R_A^*)\rightarrow (A,L_A,R_A)$ is an $A$-bimodule homomorphism.
\end{proof}



\begin{lem}                         \mlabel     {lem:a^*}
Let $(A,\circ)$ be a Novikov algebra and $r\in A\otimes A$. Let $\alpha$, $\beta$ be the linear maps defined by Eq.~(\mref{r:a and b}) and $\cbb$ be invariant.\delete{ which also implies $\beta$ is a balanced $A$-bimodule homomorphism.} Then the following two statements are equivalent.
\begin{enumerate}
\item[(\romannumeral1)] $\alpha$ is an extended $\calo$-operator of weight $0$ with extension $\beta$ of mass $(-1,0)$   on $(A, \circ)$ associated to $(A^*,L_{A,\star}^*,-R_A^*)$.\delete{ i.e.,
\begin{equation}        
\alpha(a^*)\circ \alpha(b^*)-\alpha \Big( L_\star^*(\alpha(a^*))b^*-R^*(\alpha(b^*)) a^*\Big)
=- \beta(a^*) \circ \beta(b^*)
, \;\;\; a^*,b^*\in A^*.
\end{equation}}
\item[(\romannumeral2)] $\hr$ (resp. $-\hr^t$) is an $\calo$-operator of weight $1$ on $(A, \circ)$ associated to a new $A$-bimodule Novikov algebra $(A^*,\circ_+,L_{A,\star}^*,-R_A^*)$~(resp. $(A^*,\circ_-,L_{A,\star}^*,-R_A^*)$),\delete{ i.e.,
\begin{equation}                               \mlabel{cor:r:+}
\hr(a^*)\circ \hr(b^*)=  \hr \Big( L_\star^*(\hr(a^*))b^*-R^*(\hr(b^*)) a^* +a^*\circ_+b^* \Big)
,\;\;\;  a^*,b^*\in A^*,
\end{equation}
(resp.
\begin{equation}                                \mlabel{cor:r:-}
(-\hr^t)(a^*)\circ (-\hr^t)(b^*)=  (-\hr^t) \Big( L_\star^*((-\hr^t)(a^*))b^*-R^*((-\hr^t)(b^*)) a^* +a^*\circ_-b^* \Big)
, \;\;\; a^*,b^*\in A^*.
\end{equation}
), } where the Novikov algebra multiplications $\circ_\pm$ on $A^*$ are defined by
\begin{equation}                       \mlabel{cor:r:yunsuan}
a^* \circ_\pm b^* = \mp 2 L_{A,\star}^*(  \beta(a^*)  )b^*
, \;\;\; a^*,b^*\in A^*.
\end{equation}
\end{enumerate}
\end{lem}
\begin{proof}
The conclusion follows directly from Corollary \mref{cor:r+-} and Lemma \ref{lem:r}. 
\end{proof}
The following theorem yields a close relationship between extended $\calo$-operators on a Novikov algebra $(A,\circ)$ and solutions of the NYBE in $A$.

\begin{thm}               \label     {thm:r-ENYBE}
Let $(A,\circ)$ be a Novikov algebra and $r\in A\otimes A$. 
\begin{enumerate}
\item[(\romannumeral1)] Then $r$ is a solution of the NYBE in $A$ if and only if $\hr$ satisfies
\begin{equation}\label{o-NYBE}
\hr(a^*)\circ \hr(b^*)=\hr\Big( L_{A,\star}^*(\hr(a^*))b^*-(-R_A)^*(\hr^t(b^*))a^* \Big),\quad   a^*,b^*\in A.
\end{equation}
\item[(\romannumeral2)] Let $\hr=\alpha +\beta$, where $\alpha$ and $\beta$ are defined by Eq.~(\mref{r:a and b}). Assume that the symmetric part $\cbb$ of $r$ is invariant. Then $r$ is a solution of the following equation: 
\begin{equation}     \label{o-ENYBE}
r_{13}\circ r_{23} +r_{12}\star r_{23}+r_{13}\circ r_{12}=\frac{\kappa +1}{4}
(r_{13} + (\tau r)_{13}) \circ (r_{23} + (\tau r)_{23})
\vspb
\end{equation}
if and only if $\alpha$ is an extended $\calo$-operator with extension $\beta$
of mass $(\kappa, 0)$ on $(A, \circ)$ associated to $(A^\ast, L_{A,\star}^\ast, -R_A^\ast)$.
\delete{\begin{equation}\notag
\alpha(a^*)\circ \alpha(b^*)-\alpha\Big( L_{A,\star}^*(\alpha(a^*))b^*+(-R_A^*)(\alpha(b^*))a^* \Big)=\kappa \beta(a^*)\circ\beta(b^*),\quad \;\; a^*,b^*\in A^\ast.
\end{equation}}
\end{enumerate}
\vspb
\end{thm}

\begin{proof}
(\romannumeral1)~Let $r=\sum_i u_i \otimes v_i \in A\otimes A$. For all $a^*$, $b^*$, $c^* \in A^*$, we have
\begin{eqnarray*}
\langle r_{13}\circ r_{23}, a^*\ot b^* \ot c^*\rangle
&&=\sum_{i,j}\langle u_i,a^*\rangle \langle u_j,b^*\rangle \langle v_i \circ v_j,c^*\rangle=\sum_{j}\langle u_j,b^*\rangle\langle \hr(a^*) \circ v_j, c^*\rangle
\\&&=\langle \hr(a^*)\circ \hr(b^*),c^* \rangle,\\
\langle r_{12}\star r_{23}, a^*\ot b^* \ot c^*\rangle
&&=\sum_{i,j} \langle u_i,a^*\rangle      \langle v_i \star u_j,b^*\rangle
\langle v_j, c^*\rangle
=\sum_{j} \langle  L_{A,\star}(u_j) \hr(a^*),b^*\rangle       \langle v_i, c^*\rangle
\\&&=\sum_{j} -\langle  u_j, L_{A,\star}^* (\hr(a^*))b^*\rangle       \langle v_j, c^*\rangle
=\Big\langle - \hr\Big(L_{A,\star}^* (\hr(a^*))b^* \Big),c^*\Big\rangle,\\
\langle r_{13}\circ r_{12}, a^*\ot b^* \ot c^*\rangle
&&=\sum_{i,j} \langle u_j\circ u_i,a^*\rangle    \langle v_i,b^*\rangle
\langle v_j,c^*\rangle
\\&&=\sum_{j} \langle u_j\circ \hr^t(b^*), a^*\rangle \langle v_j,c^*\rangle
=\Big\langle  \hr\Big(-R_A^* (\hr^t(b^*))a^* \Big),c^*\Big\rangle.
\end{eqnarray*}
Thus $r$ is a solution of the NYBE in $A$ if and only if $\hr$ satisfies Eq.~(\mref{o-NYBE}).
\smallskip

\noindent
(\romannumeral2) By Lemma \mref{lem:r}, the proof of Item (\romannumeral1) and Eq. (\mref{r:a and b}), for all $a^*$, $b^*$, $c^* \in A^*$, we have
\begin{eqnarray*}
&&\Big\langle \alpha(a^*)\circ \alpha(b^*)-\alpha\Big( L_{A,\star}^*(\alpha(a^*))b^*+(-R_A^*)(\alpha(b^*))a^* \Big)-\kappa \beta(a^*)\circ\beta(b^*) ,c^*\Big\rangle \\
&&=\Big\langle
\alpha(a^*)\circ\alpha(b^*)
+\Big( \alpha(a^*)\circ\beta(b^*)-\beta(L_{A,\star}^*(  \alpha(a^*)) b^* \Big)
+\Big( \beta(a^*)\circ \alpha(b^*) +\beta( R_A^*(  \alpha(b^*)) a^* ) )  \Big)\\
&&\quad + \Big( \alpha( -R_A^*(\beta(b^*)) a^*) -\alpha( L_{A,\star}^*(\beta(a^*)) b^*)  \Big)
+\Big(  -\beta( R_A^*(  \beta(b^*) )a^*) - \beta( L_{A,\star}^*(  \beta(a^*)) b^*  \Big)
\\&&\quad +\Big( \alpha( R_A^*(\alpha(b^*)) a^*) -\alpha( L_{A,\star}^*(\alpha(a^*)) b^*) \Big)
+\beta(a^*)\circ\beta(b^*) -(\kappa +1)\beta(a^*)\circ\beta(b^*)
,c^*\Big\rangle \\
&&=\Big\langle  \Big( \alpha(a^*)\circ\alpha(b^*)+            \alpha(a^*)\circ\beta(b^*) +\beta(a^*)\circ \alpha(b^*)       +        \beta(a^*)\circ \beta(b^*)   \Big)
        -(\kappa +1) \beta(a^*)\circ\beta(b^*)\\
&&\quad +  \Big( -\alpha( R_A^*(\beta(b^*)) a^*)  -\beta( R_A^*(  \beta(b^*) )a^*  )
 +\alpha( R_A^*(\alpha(b^*)) a^*)  + \beta( R^*(  \alpha(b^*)) a^* )  \Big) \\
&&\quad +\Big(-\alpha( L_{A,\star}^*(\alpha(a^*)) b^*)- \beta( L_{A,\star}^*(\alpha(a^*) )b^*  )
-\alpha( L_{A,\star}^*(\beta(a^*)) b^*)  - \beta( L_{A,\star}^*(  \beta(a^*)) b^* ) \Big)
,c^* \Big\rangle
\\&&=\Big\langle \hr(a^*)\circ \hr(b^*)-\hr\Big( L_{A,\star}^*(\hr(a^*))b^*\Big)+
\hr\Big( (-R_A^*)(\hr^t(b^*))a^* \Big)-(\kappa +1) \beta(a^*)\circ\beta(b^*) ,c^*\Big\rangle \\
&&= \Big\langle  r_{13}\circ r_{23} +r_{12}\star r_{23}+r_{13}\circ r_{12}
 -(\kappa +1) \beta_{13}\circ \beta_{23}
,a^*\ot b^*\ot c^*\Big\rangle
\\&&= \Big\langle r_{13}\circ r_{23} +r_{12}\star r_{23}+r_{13}\circ r_{12}
 -\frac{\kappa +1}{4} (r_{13} + (\tau r)_{13}) \circ (r_{23} + (\tau r)_{23})
,a^*\ot b^*\ot c^*\Big\rangle.
\end{eqnarray*}
 Therefore, $r$ satisfies Eq.~(\mref{o-ENYBE}) if and only if $\alpha$ is an extended $\calo$-operator with extension $\beta$ of mass  $( \kappa, 0)$ on $(A, \circ)$ associated to $(A^\ast, L_{A,\star}^\ast, -R_A^\ast)$.
\vspb
\end{proof}

Now, we introduce the definition of extended Novikov Yang-Baxter equations.
\begin{defi}
Let $(A,\circ)$ be a Novikov algebra and $\epsilon \in \bf k $. The
equation
\vspb        
\begin{eqnarray}  \label{ENYBE}
r_{13}\circ r_{23} +r_{12}\star r_{23}+r_{13}\circ r_{12}=
\epsilon(r_{13} + (\tau r)_{13}) \circ (r_{23} + (\tau r)_{23})
\vspb
\end{eqnarray}
is called the {\bf extended Novikov Yang-Baxter equation (ENYBE) of mass $\epsilon $} in $A$.
\end{defi}
\begin{rmk}
Note that when $\epsilon = 0$ or $r$ is skew-symmetric \delete{in the sense that $\tau(r)=-r$}, the ENYBE of mass $\epsilon$ is the same as the NYBE.
\end{rmk}

\begin{cor}      \label{cor:o-ENYBE}
Let $(A,\circ)$ be a Novikov algebra and $r\in A\otimes A$. Let $\hr=\alpha +\beta$, where $\alpha$ and $\beta$ are defined by Eq.~(\mref{r:a and b}).
If $\cbb$ is invariant, then the following statements are equivalent.
	\begin{enumerate}
	\item\label{item1} $r$ is a solution of the NYBE in $A$.
	\item \label{item2} $\hr$ (resp. $-\hr^t$) is an $\calo$-operator of weight $1$ on $(A,\circ)$ associated to the $A$-bimodule Novikov algebra $(A^*,\circ_+,L_{A,\star}^*,-R_A^*)$ (resp. $ (A^*,\circ_-,L_{A,\star}^*,-R_A^*)$), where the Novikov algebra multiplication $\circ_+$ (resp. $\circ_-$) on $A^*$ are defined by Eq. (\mref{cor:r:yunsuan}).
	\item \label{item3} $\alpha$ is an extended $\calo$-operator of weight $0$ with extension $\beta$ of mass $(-1,0)$ on $(A,\circ)$ associated to $ (A^*,L_{A,\star}^*,-R_A^*)$.
	\item \label{item4} For all $a^*$, $b^* \in A$,
	\begin{equation}                    \mlabel{* and quan}
	( \alpha\pm\beta) (a^* \ast b^*)=( \alpha\pm\beta)(a^*)\circ
	( \alpha\pm\beta)(b^*),
	\end{equation}
	where
	\begin{equation*}
	a^* \ast b^* = L_{A,\star}^*(\hr(a^*))b^*-(-R_A)^*(\hr^t(b^*))a^*,\;\;  a^*,b^*\in A^\ast.
	\end{equation*}
	\end{enumerate}
\end{cor}

\begin{proof}
Note that $\hr^t=2\beta-\hr$. By Lemma \mref{lem:r}, we have
\begin{eqnarray*}
&&\hr(a^*)\circ \hr(b^*)- \hr \Big( L_{A,\star}^*(\hr(a^*))b^*+R_A^*(\hr^t(b^*)) a^* \Big)\\
&&=\hr(a^*)\circ \hr(b^*)- \hr \Big( L_{A,\star}^*(\hr(a^*))b^*-R_A^*(\hr(b^*)) a^*-2L_{A,\star}^*(\beta(b^*)a^* \Big) \\
&&=\hr(a^*)\circ \hr(b^*)- \hr \Big( L_{A,\star}^*(\hr(a^*))b^*-R_A^*(\hr(b^*)) a^* +a^*\circ_+b^* \Big) ,
\end{eqnarray*}
where the multiplication $\circ_+$ on $A^*$ is defined by Eq. (\mref{cor:r:yunsuan}). Then by Lemma \mref{lem:a^*} and Theorem \ref{thm:r-ENYBE}, $r$ is a solution of the NYBE in $A$ if and only if Item (\ref{item2}) holds if and only if 
Item (\ref{item3}) holds. Moreover, it is easy to see that Item (\ref{item2}) holds if and only if Item (\ref{item4}) holds. Then this conclusion holds.
\end{proof}

\begin{rmk}
In particular, if $r$ is skew-symmetric, then $r$ is a solution of the NYBE in $A$ if and only if $\hr :~A^* \rightarrow A$ is an $\calo$-operator of weight $0$ on $(A,\circ)$ associated to $ (A^*,L_{A,\star}^*,-R_A^*)$, which follows from $\cbb=0$. This is \cite[Theorem 3.27]{HBG}.
\end{rmk}

\zhushi{

In this part, we want to give the connection between the extended $\calo$-operator and the extended ENYBE. For this, we want to use Corollary \mref{cor:r+-} with new notations on triple $(A,L_\star^*,-R^*)$.

\begin{lem}                         \mlabel     {lem:a^*}
Let $(A,\circ)$ be a Novikov algebra and $r\in A\otimes A$. Let $\caa$, $\beta$ be defined by Eq.~(\mref{r:a and b}) and $\beta$ be invariant, which also implies $\beta$ is a balanced $A$-bimodule homomorphism as a map in $\Hom_{\bf k}(A^*,A)$. Then the following two assertions are equivalent.
\begin{enumerate}
\item[(\romannumeral1)] The map $\alpha$ is an extended $\calo$-operator with extension $\beta$ of mass -1, i.e.,
\begin{equation}        
\alpha(a^*)\circ \alpha(b^*)-\alpha \Big( L_\star^*(\alpha(a^*))b^*+-R^*(\alpha(b^*)) a^*\Big)
=- \beta(a^*) \circ \beta(b^*)
, \tforall a^*,b^*\in A^*.
\end{equation}
\item[(\romannumeral2)] The map r (resp. $-r^t$) is an $\calo$-operator of weight 1 associated to a new $A$-bimodule Novikov algebra $(A^*,\circ_+,L_\star^*,-R^*)$~(resp. $(A^*,\circ_-,L_\star^*,-R^*)$), i.e.,
\begin{equation}                               \mlabel{cor:r:+}
r(a^*)\circ r(b^*)=  r \Big( L_\star^*(r(a^*))b^*-R^*(r(b^*)) a^* +a^*\circ_+b^* \Big)
, \tforall a^*,b^*\in A^*,
\end{equation}
(resp.
\begin{equation}                                \mlabel{cor:r:-}
(-r^t)(a^*)\circ (-r^t)(b^*)=  (-r^t) \Big( L_\star^*((-r^t)(a^*))b^*-R^*((-r^t)(b^*)) a^* +a^*\circ_-b^* \Big)
, \tforall a^*,b^*\in A^*.
\end{equation}
), where the Novikov algebra multiplication $\circ_\pm$ on $A^*$ are defined by
\begin{equation}                       \mlabel{cor:r:yunsuan}
a^* \circ_\pm b^* = \mp 2 L_\star^*(  \beta(a^*)  )b^*
, \tforall a^*,b^*\in A^*.
\end{equation}
\end{enumerate}
\end{lem}

The following theorem yields a close relationship between extended $\calo$-operators on a Novikov algebra $(A,\circ)$ and solutions of the NYBE in $A$.

\begin{thm}               \label     {thm:r-ENYBE}
Let $(A,\circ)$ be a Novikov algebra and $r\in A\otimes A$, which can also be identified as a linear map in $\Hom_{\bf k}(A^*,A)$.
\begin{enumerate}
\item[(\romannumeral1)] Then $r$ is a solution of the NYBE in $A$ if and only if $r$ satisfies
\begin{equation}\label{o-NYBE}
r(a^*)\circ r(b^*)=r\Big( L_\star^*(r(a^*))b^*-(-R)^*(r^t(b^*))a^* \Big),\quad \forall  a^*,b^*\in A.
\end{equation}
\item[(\romannumeral2)] Define $r=\alpha +\beta$ by Eq.~(\mref{r:a and b}). Assume that the symmetric part $\beta$ of $r$ is invariant. Then $r$ is a solution of ENYBE of mass $\frac{\kappa +1}{4}$:
\begin{equation}     \label{o-ENYBE}
r_{13}\circ r_{23} +r_{12}\star r_{23}+r_{13}\circ r_{12}=\frac{\kappa +1}{4}
(r_{13} + r_{31}) \circ (r_{23} + r_{32})
\vspb
\end{equation}
if and only if $\alpha$ is an extended $\calo$-operator with extension $\beta$
of mass $(\varrho\neq 0,\kappa,0)$:
\begin{equation}\notag
\alpha(a^*)\circ \alpha(b^*)-r\Big( L_\star^*(\alpha(a^*))b^*+(-R^*)(\alpha(b^*))a^* \Big)=\kappa \beta(a^*)\circ\beta(b^*)\quad \forall  a^*,b^*\in A.
\end{equation}
\end{enumerate}
\vspb
\end{thm}

\begin{proof}
(\romannumeral1)~Let $r=\sum_i u_i \otimes v_i \in A\otimes A$. For any $a^*$, $b^*$, $c^* \in A^*$, we have
\begin{eqnarray*}
\langle r_{13}\circ r_{23}, a^*\ot b^* \ot c^*\rangle
&&=\sum_{i,j}\langle u_i,a^*\rangle \langle u_j,b^*\rangle \langle v_i \circ v_j,c^*\rangle=\sum_{j}\langle u_j,b^*\rangle\langle r(a^*) \circ v_j, c^*\rangle
\\&&=\langle r(a^*)\circ r(b^*),c^* \rangle,
\end{eqnarray*}
\begin{eqnarray*}
\langle r_{12}\star r_{23}, a^*\ot b^* \ot c^*\rangle
&&=\sum_{i,j} \langle u_i,a^*\rangle      \langle v_i \star u_j,b^*\rangle
\langle v_j, c^*\rangle
=\sum_{j} \langle  L_\star(u_j) r(a^*),b^*\rangle       \langle v_i, c^*\rangle
\\&&=\sum_{j} -\langle  u_j, L_\star^* (r(a^*))b^*\rangle       \langle v_j, c^*\rangle
=\Big\langle - r\Big(L_\star^* (r(a^*))b^* \Big),c^*\Big\rangle,
\end{eqnarray*}
\begin{eqnarray*}
\langle r_{13}\circ r_{12}, a^*\ot b^* \ot c^*\rangle
&&=\sum_{i,j} \langle u_j\circ u_i,a^*\rangle    \langle v_i,b^*\rangle
\langle v_j,c^*\rangle
=\sum_{j}- \langle r,L^*(u_j) a^*\ot b^*\rangle \langle v_j,c^*\rangle
\\&&=\sum_{j} \langle u_j\circ r^t(b^*), a^*\rangle \langle v_j,c^*\rangle
=\Big\langle  r\Big(-R^* (r^t(a^*))b^* \Big),c^*\Big\rangle,
\end{eqnarray*}
So $r$ is a solution of the NYBE in $A$ if and only if $r$ satisfies Eq.~(\mref{o-NYBE}).
\smallskip

\noindent
(\romannumeral2) Thanks to the proof of Item (\romannumeral1), for any $a^*$, $b^*$, $c^* \in A^*$, we have
\begin{eqnarray*}
&&\Big\langle \alpha(a^*)\circ \alpha(b^*)-\alpha\Big( L_\star^*(\alpha(a^*))b^*+(-R^*)(\alpha(b^*))a^* \Big)-\kappa \beta(a^*)\circ\beta(b^*) ,c^*\Big\rangle \\
&&=\Big\langle
\alpha(a^*)\circ\alpha(b^*)
+\Big( \alpha(a^*)\circ\beta(b^*)-\beta(L_\star^*(  \alpha(a^*)) b^* \Big)
+\Big( \beta(a^*)\circ \alpha(b^*) +\beta( R^*(  \alpha(b^*)) a^* ) )  \Big)\\
&&\quad + \Big( \alpha( -R^*(\beta(b^*)) a^*) -\alpha( L_\star^*(\beta(a^*)) b^*)  \Big)
+\Big(  -\beta( R^*(  \beta(b^*) )a^*) - \beta( L_\star^*(  \beta(a^*)) b^*  \Big)
\\&&\quad +\Big( \alpha( R^*(\alpha(b^*)) a^*) -\alpha( L_\star^*(\alpha(a^*)) b^*) \Big)
+\beta(a^*)\circ\beta(b^*) -(\kappa +1)\beta(a^*)\circ\beta(b^*)
,c^*\Big\rangle \\
&&=\Big\langle  \Big( \alpha(a^*)\circ\alpha(b^*)+            \alpha(a^*)\circ\beta(b^*) +\beta(a^*)\circ \alpha(b^*)       +        \beta(a^*)\circ \beta(b^*)   \Big)
        -(\kappa +1) \beta(a^*)\circ\beta(b^*)\\
&&\quad +  \Big( -\alpha( R^*(\beta(b^*)) a^*)  -\beta( R^*(  \beta(b^*) )a^*  )
 +\alpha( R^*(\alpha(b^*)) a^*)  + \beta( R^*(  \alpha(b^*)) a^* )  \Big) \\
&&\quad +\Big(-\alpha( L_\star^*(\alpha(a^*)) b^*)- \beta( L_\star^*(\alpha(a^*) )b^*  )
-\alpha( L_\star^*(\beta(a^*)) b^*)  - \beta( L_\star^*(  \beta(a^*)) b^* ) \Big)
,c^* \Big\rangle
\\&&=\Big\langle r(a^*)\circ r(b^*)-r\Big( L_\star^*(r(a^*))b^*\Big)+
r\Big( (-R^*)(r^t(b^*))a^* \Big)-(\kappa +1) \beta(a^*)\circ\beta(b^*) ,c^*\Big\rangle \\
&&= \Big\langle  r_{13}\circ r_{23} +r_{12}\star r_{23}+r_{13}\circ r_{12}
 -(\kappa +1) \beta_{13}\circ \beta_{23}
,a^*\ot b^*\ot c^*\Big\rangle
\\&&= \Big\langle r_{13}\circ r_{23} +r_{12}\star r_{23}+r_{13}\circ r_{12}
 -\frac{\kappa +1}{4} (r_{13} + r_{31}) \circ (r_{23} + r_{32})
,a^*\ot b^*\ot c^*\Big\rangle ,
\end{eqnarray*}
where the first equality follows $Lemma$~\mref{lem:r} and the last equality follows $Eq$.~(\mref{r:a and b}). Therefore, $r$ is a solution of the ENYBE of mass $(\kappa+1)/4$ if and only if $\alpha$ is an extended $\calo$-operator with extension $\beta$ of mass $\kappa$.
\vspb
\end{proof}

If $\kappa=-1$, we have

\begin{cor}      \label{cor:o-ENYBE}
Let $(A,\circ)$ be a Novikov algebra and $r\in A\otimes A$. Define $r=\alpha +\beta$ by Eq.~(\mref{r:a and b}).
If $\beta$ is invariant. Then the following statements are equivalent:
	\begin{enumerate}
	\item $r$ is a solution of NYBE in $A$.
	\item $r$ (resp. $-r^t$) satisfies Eq.~(\mref{cor:r:+}) (resp.~Eq. (\mref{cor:r:-})), i.e., $r$ (resp. $-r^t$) is an extended $\calo$-operator of~weight 1 associated to the $A$-bimodule Novikov algebra $(A^*,\circ_+,L_\star^*,-R^*)$\\(resp. $ (A^*,\circ_-,L_\star^*,-R^*)$), where the Novikov algebra multiplication $			\circ_+$ (resp. $\circ_-$) on $A^*$ are defined by Eq. (\mref{cor:r:yunsuan}).
	\item $\alpha$ is an $\calo$-operator with extension $\beta$ of mass -1.
	\item For any $a^*,~b^* \in A$,
	\begin{equation}                    \mlabel{* and quan}
	( \alpha\pm\beta) (a^* \ast b^*)=( \alpha\pm\beta)(a^*)\circ
	( \alpha\pm\beta)(b^*),
	\end{equation}
	where
	\begin{equation*}
	a^* \ast b^* = L_\star^*(r(a^*))b^*-(-R)^*(r^t(b^*))a^* \quad \forall  a^*,b^*\in A.
	\end{equation*}
	\end{enumerate}
\end{cor}

\begin{proof}
Since the symmetric part of $\beta$ of $r$ is invariant, then by lemma \mref{lem:r} and $r^t=2\beta-r$, for any $a^*,b^*\in A$, we have
\begin{eqnarray*}
&&r(a^*)\circ r(b^*)- r \Big( L_\star^*(r(a^*))b^*+R^*(r^t(b^*)) a^* \Big)\\
&&=r(a^*)\circ r(b^*)- r \Big( L_\star^*(r(a^*))b^*-R^*(r(b^*)) a^*-2L_\star^*(\beta(b^*)a^* \Big) \\
&&=r(a^*)\circ r(b^*)- r \Big( L_\star^*(r(a^*))b^*-R^*(r(b^*)) a^* +a^*\circ_+b^* \Big) ,
\end{eqnarray*}
where the multiplication $\circ_+$ on $A^*$ are defined by Eq. (\mref{cor:r:yunsuan}). Then Corollary \mref{cor:a^*} yields $r$ is a solution of NYBE in $A$ $\Leftrightarrow$ Item (b) $\Leftrightarrow$ 
Item(c) holds. Moreover, Eq. (\mref{* and quan}) is equivalent to Eq. (\mref{cor:r:+}) and Eq. (\mref{cor:r:-}). Thus Item(a) $\Leftrightarrow$ Item (d) holds.
\vspb
\end{proof}
Note that $r$ is a skew-symmetric solution of NYBE in $A$ if and only if $\hr :~A^* \rightarrow A$ is an $\calo$-operator of weight zero, which follows from $\beta=0$.\mcite{HBG}
}

\zhushi{
r\Big(L_\star^*(r(a^*))b^* + -R^*(r(b^*)a^*)\Big)
+\Big(-r(a^*)\circ\beta(b^*)+ \beta(L_\star^*(r(a^*))b^*  \Big)
\\&&+\Big(-\beta(a^*)\circ r(b^*)+ \beta(-R^*(r(b^*))a^*  \Big)
+ r \Big(-L_\star^*(\beta(a^*))b^* +R^*(\beta(b^*)a^*   ) \Big)
=
\begin{equation}        
\alpha(a^*)\circ \alpha(b^*)-\alpha \Big( L_\star^*(\alpha(a^*))b^*-R^*(\alpha(b^*)) a^*\Big)
=- \beta(a^*) \circ \beta(b^*)
, \tforall a^*,b^*\in A^*.
\end{equation}

\begin{equation}
r(a^*)\circ r(b^*)=  r \Big( L_\star^*(r(a^*))b^*-R^*(r(b^*)) a^* +a^*\circ_+b^* \Big)
, \tforall a^*,b^*\in A^*,
\end{equation}
(resp.
\begin{equation}
(-r^t)(a^*)\circ (-r^t)(b^*)=  (-r^t) \Big( L_\star^*((-r^t)(a^*))b^*-R^*((-r^t)(b^*)) a^* +a^*\circ_-b^* \Big)
, \tforall a^*,b^*\in A^*.
\end{equation}
), where the Novikov algebra multiplication $\circ_\pm$ on $A*$ are defined by
\begin{equation}
a^* \circ_\pm b^* = \mp 2 L_\star^*(  \beta(a^*)  )b^*
, \tforall a^*,b^*\in A^*.
\end{equation}
}

\begin{cor} Let $(A,\circ)$ be a Novikov algebra, $r\in A\ot A$ is a solution of the NYBE in $A$ and the symmetric part $\check{\beta}$ of $r$ is invariant. Then
\begin{align}  \notag
a^*\odot b^*:=-2 L_{A,\star}^*(  \beta(a^*)  )b^*,~
a^*\rhd b^*:=L_{A,\star}^*(\hr(a^*))b^* ,~
a^*\lhd b^*:=-R_A^*(\hr(b^*)) a^* ,~a^*,b^*\in A^*,
\end{align}
defines a post-Novikov algebra structure on $A^*$. In particular, if $\hr:A^*\rightarrow A$ is invertible, then there exists a compatible post-Novikov algebra structure on $A$ defined by
{\small
\begin{align}  \notag
x\odot y:=-2\hr\Bigg( L_{A,\star}^*\Big(  \beta(  \hr^{-1}(x)   ) \Big )\hr^{-1}(y)\Bigg),~
x\rhd y:=\hr\Big(  L_{A,\star}^*(x)\hr^{-1}(y)  \Big) ,~
x\lhd y:=-\hr\Big( R_A^*(y) \hr^{-1}(x)  \Big),~x,y\in A,
\end{align}}
\end{cor}
\begin{proof}
By Corollary \mref{cor:o-ENYBE}, we have $\hr(a^*)\circ \hr(b^*)= \hr \Big( L_{A,\star}^*(\hr(a^*))b^*-R_A^*(\hr(b^*)) a^* -2 L_{A,\star}^*(  \beta(a^*)  )b^*\Big)$ for all $a^\ast$, $b^\ast\in A^\ast$. Then this conclusion follows immediately from Theorem \mref{yibanpost}.
\end{proof}

\subsection{ENYBE on quadratic Novikov algebras}
\mlabel{ss:self-dual}               
\delete{In this part we focus on the relationship between $A$-bimodule Novikov algebras $(A,L,R)$ and $(A^*,L_\star^*,-R^*)$ about the extended $\calo$-operator and ENYBE. } We recall some concepts about quadratic Novikov algebras in \mcite{HBG}.

Let $(A,\circ)$ be a Novikov algebra and $\calb(\cdot,\cdot):A\times A \rightarrow \bfk$ be a nondegenerate bilinear form. Let $\varphi:A\rightarrow A^*$ be the induced invertible linear map by $\calb(\cdot,\cdot)$ given as
\begin{eqnarray}  \label{bilinear map}
\calb(a,b)=\langle  \varphi(a),b   \rangle,\quad  a,b\in A.
\end{eqnarray}

\begin{defi} \cite{HBG}  
\mlabel{Novbilinear}
Let $(A,\circ)$ be a Novikov algebra. A bilinear form
$\mathcal{B}(\cdot,\cdot)$  on $A$ is called {\bf invariant} if it satisfies
\vspa
\begin{eqnarray}\label{bilinear1}
\mathcal{B}(a\circ b,c)=-\mathcal{B}(b, a\star c)\;\;\
a,b,c\in A.
\vspa
\end{eqnarray}
A {\bf quadratic Novikov algebra}, denoted by $(A,
\circ,\mathcal{B}(\cdot,\cdot))$, is a Novikov algebra $(A,\circ)$
together with a nondegenerate symmetric invariant bilinear form
$\mathcal{B}(\cdot,\cdot)$.
\vspa
\end{defi}

Let $(A,\circ)$ be a Novikov algebra and $\calb(\cdot,\cdot):A\times A \rightarrow \bfk$ be a nondegenerate bilinear form. Suppose that $T:A\rightarrow A$ is a linear map. Then $T$ is called {\bf self-adjoint} (resp. {\bf skew-adjoint} ) with respect to $\calb(\cdot,\cdot)$ if
\begin{equation}    \label{self-adjoint}
\calb (T(a),b)=\calb (a,T(b))\quad (\text{resp.}~\calb (T(a),b)=-\calb (a,T(b)),
\quad
a,b\in A.
\end{equation}

\begin{lem}       \mlabel{operator dual}
Let $(A,\circ,\mathcal{B}(\cdot,\cdot))$ be a quadratic Novikov algebra and $\varphi$ be the invertible linear map $:A\rightarrow A^*$ defined by Eq. (\mref{bilinear map}). Suppose that $\beta:A\rightarrow A$ is a linear map that is self-adjoint with respect to $\calb(\cdot,\cdot)$. Then 
$\beta:A\rightarrow A$ is a balanced $A$-bimodule homomorphism \delete{of mass $\kappa$} from $(A, L_A, R_A)$ to $(A, L_A, R_A)$ \delete{i.e. it satisfies Eq.~(\mref{bimodule homo}) with $\kappa\neq 0$,} if and only if $\bbvar=\beta\varphi^{-1}:A^*\rightarrow A$ is a balanced $A$-bimodule homomorphism from  $(A, L_{A, \star}^\ast, -R_A^\ast)$ to $(A, L_A, R_A)$.\delete{ i.e.
\begin{eqnarray}\mlabel{tilde1}
 &&L_{A,\star}^*(\bbvar(a^*))b^*=
-R_A^*(\bbvar(b^*))a^*,\\
\mlabel{tilde2}
&&\bbvar(L_{A,\star}^*(x) a^*)= x\circ \bbvar(a^*), \quad \bbvar(-R_A^*(x) a^*)=\bbvar(a^*)\circ x,\quad   a^*, b^\ast\in A^*,\;x\in A.
\end{eqnarray} In particular, if $\beta = \id: A\rightarrow A$, then $\varphi^{-1}:A^*\rightarrow A$ is a balanced $A$-bimodule homomorphism.}
\end{lem}
\begin{proof}
 Let $a^*$, $b^*\in A^*$. Set $x=\varphi^{-1}(a^*)$ and $y=\varphi^{-1}(b^*) \in A$.
 Since $\calb(\cdot,\cdot)$ is symmetric and $\beta$ is self-adjoint with respect to $\calb(\cdot,\cdot)$, we have $\langle \beta(x),\varphi(y)\rangle  =\calb(\beta(x),y)=\calb(x,\beta(y))=
\langle \varphi(x),\beta(y)\rangle$. Therefore, we get $\langle \bbvar(a^*),b^*\rangle=\langle a^*,\bbvar(b^*)\rangle$. Thus $\check{\bbvar} \in A\ot A$ is symmetric.
By Lemma \mref{lem:r}, $\bbvar$ is balanced if and only if $\bbvar$ is an $A$-bimodule homomorphism. Therefore, we only need to prove that $\beta$ is a balanced $A$-bimodule homomorphism if and only if $\bbvar$ is balanced. 
Since $\calb(\cdot,\cdot)$ is symmetric and invariant and $\beta$ is self-adjoint with respect to $\calb(\cdot,\cdot)$, for all $z\in A$, we have
\begin{eqnarray*}
\langle L_{A,\star}^*(\bbvar(a^*))b^*,z\rangle
&&=\langle L_{A,\star}^*(\beta(x))\varphi(y),z\rangle
=-\langle \varphi(y),\beta(x)\star z\rangle
\\&&=-\calb(y,\beta(x)\star z )
=\calb(z\circ y,\beta(x))=\calb(\beta(z\circ y),x)  ,\\
\langle -R_A^*(\bbvar(b^*))a^*,z\rangle
&&=\langle -R_A^*((\beta(y))\varphi(x),z\rangle
=\langle \varphi(x),z\circ \beta(y)\rangle
\\&&=\calb(x,z\circ \beta(y)).
\end{eqnarray*}
Since $\calb(\cdot,\cdot)$ is nondegenerate, $ L_{A,\star}^*(\bbvar(a^*))b^*=
-R_A^*(\bbvar(b^*))a^*$ for all $a^\ast$, $b^\ast\in A^\ast$ if and only if $\beta(z\circ y)=z\circ \beta(y)$ for all $y$, $z\in A$.
\delete{Now we attack to different cases about $\kappa$. If $\kappa=0$, there is nothing to prove. If $\kappa\neq 0$, thanks to Lemma \mref{lem:r}, Eq. (\mref{tilde1}) and Eq. (\mref{tilde2}) are equivalent.} On the other hand,  we have
\begin{eqnarray*}
\langle L_{A,\star}^*(\bbvar(a^*))b^*,z\rangle
=-\calb(y, \beta(x)\star z ),\\
\langle -R_A^*(\bbvar(b^*))a^*,z\rangle
=\calb(x,z\circ \beta(y))=-\calb(\beta(y),z\star x)
=-\calb(y,\beta(z\star x)).
\end{eqnarray*}
Since $\calb(\cdot,\cdot)$ is nondegenerate, we get that $ L_{A,\star}^*(\bbvar(a^*))b^*=
-R_A^*(\bbvar(b^*))a^*$ for all $a^\ast$, $b^\ast\in A^\ast$ if and only if $\beta(x)\star z=\beta(z\star x)$ for all $x$, $z\in A$. Thus, by the definition of $\star$ and $\beta(z\circ y)=z\circ \beta(y)$ for all $y$, $z\in A$, we get that $ L_{A,\star}^*(\bbvar(a^*))b^*=
-R_A^*(\bbvar(b^*))a^*$ for all $a^\ast$, $b^\ast\in A^\ast$ if and only if $\beta(x)\circ z=\beta(x\circ z)$ for all $x$, $z\in A$. Then the proof is completed.
\end{proof}

\begin{pro}         \mlabel{dual exo}
Let $(A,\circ,\mathcal{B}(\cdot,\cdot))$ be a quadratic Novikov algebra
and $\varphi:A\rightarrow A^*$ be the invertible linear map  defined as Eq. (\mref{bilinear map}). Suppose that $T$ is a linear endomorphism of $A$, $\beta$ is a balanced $A$-bimodule homomorphism from $(A, L_A, R_A)$ to $(A, L_A, R_A)$ and $\beta$ is self-adjoint with respect to $\calb(\cdot,\cdot)$. Then the following conclusions hold.
\begin{enumerate}
\item[(\romannumeral1)] $T$ is an extended $\calo$-operator of weight $0$ with extension $\beta$ of mass $(\kappa,0)$ on $(A, \circ)$ associated to $(A,\circ, L_A, R_A)$ if and only if $\Tvar=T\varphi^{-1}:A^*\rightarrow A$ is an extended $\calo$-operator of weight $0$ with extension $\bbvar=\beta \varphi^{-1}:A^*\rightarrow A$ of mass $(\kappa,0)$.\delete{, i.e. $\bbvar$ satisfies Eq. (\mref{tilde1}) and Eq. (\mref{tilde2}) and $\Tvar$ and $\bbvar$ satisfy
\begin{equation*}        
\Tvar(a^*)\circ \Tvar(b^*)-\Tvar \Big( L_\star^*(\Tvar(a^*))b^*-R^*(\Tvar(b^*)) a^*\Big)
=\kappa \bbvar(a^*) \circ \bbvar(b^*)
, \quad a^*,b^*\in A^*.
\end{equation*}}
\item[(\romannumeral2)] Suppose that in addition, $T$ is skew-adjoint with respect to $\calb(\cdot,\cdot)$. Then $\check{\delta_{\pm}}=\check{\Tvar} \pm \check{\bbvar}$ regarded as an element of $A\ot A$ is a solution of the ENYBE of mass $\dfrac{\kappa +1}{4}$ if and only if $T$ is an extended $\calo$-operator of weight $0$ with extension $\beta$ of mass $(\kappa,0)$ on $(A, \circ)$ associated to $(A,\circ, L_A, R_A)$. In particular, we have
    \begin{itemize}
\item If $\kappa=-1$, then $\check{\delta_\pm}=\check{\Tvar} \pm \check{\bbvar}$ is a solution of the NYBE in $A$ if and only if $T$ is an extended $\calo$-operator of weight $0$ with extension $\beta$ of mass $(-1,0)$ on $(A, \circ)$ associated to $(A,\circ, L_A, R_A)$.
\item If $\kappa=0$, then $\check{\delta_\pm}=\check{\Tvar} \pm \check{\bbvar}$ is a solution of the NYBE in $A$ if and only if $T$ is a Rota-Baxter operator of weight $0$.
\end{itemize}
\end{enumerate}
\end{pro}

\begin{proof}
(\romannumeral1) By Lemma \mref{operator dual}, $\beta:A\rightarrow A$ is a balanced $A$-bimodule homomorphism from $(A, L_A, R_A)$ to $(A, L_A, R_A)$ 
if and only if $\bbvar=\beta\varphi^{-1}:A^*\rightarrow A$ is a balanced $A$-bimodule homomorphism from  $(A, L_{A, \star}^\ast, -R_A^\ast)$ to $(A, L_A, R_A)$. Moreover, by \cite[Proposition 2.8]{DH}, we obtain that $\varphi(L_A(x)y)=L_{A,\star}^*(x) \varphi(y)$ and $\varphi(R_A(y)x)=-R_A^*(y) \varphi(x)$ for all $ x,y\in A$.~Let $x$, $y \in A$.~Set $a^*=\varphi(x)$ and $b^*=\varphi(y)$. Then we have
\begin{eqnarray*}    
&&\Tvar\varphi(x)\circ \Tvar\varphi(y)-\Tvar\varphi \Big( L_A(\Tvar \varphi(x)) y+ R_A(\Tvar\varphi(y) )x\Big)
=\kappa \bbvar\varphi(x) \circ \bbvar\varphi(y)
\\&&\Leftrightarrow
\Tvar(a^*)\circ \Tvar(b^*)-\Tvar \varphi \Big( L_A(\Tvar(a^*)) \varphi^{-1}(b^*)+ R_A(\Tvar(b^*) )\varphi^{-1}(a^*)\Big)
=\kappa \bbvar(a^*) \circ \bbvar(b^*)
\\&&\Leftrightarrow
\Tvar(a^*)\circ \Tvar(b^*)-\Tvar\Big( L_{A,\star}^*(\Tvar(a^*))b^*-R_A^*(\Tvar(b^*)) a^*\Big)
=\kappa \bbvar(a^*) \circ \bbvar(b^*).
\end{eqnarray*}
Therefore, this conclusion holds. \delete{$T$ is an extended $\calo$-operator with extension $\beta$ of mass $(0,\kappa,0)$ if and only if $\Tvar=T\varphi^{-1}:A^*\rightarrow A$ is an extended $\calo$-operator with extension $\bbvar$ of mass $(0,\kappa,0)$.}

(\romannumeral2) In particular, if $T$ is skew-adjoint with respect to $\calb(\cdot,\cdot)$, then we have $\langle T(x),\varphi(y)\rangle+  \langle \varphi(x),T(y) \rangle=0$. for all $x$, $y\in A$. Set $a^*=\varphi(x)$ and $b^*=\varphi(y)$. Then we have $\langle \Tvar(a^*),b^*\rangle  +  \langle a^*,\Tvar(b^*) \rangle=0$ for all $a^\ast$, $b^\ast\in A^\ast$. Therefore, $\check{\Tvar} \in A \ot A$ is skew-symmetric. Then Item (\romannumeral2) follows from Item (\romannumeral1) and Theorem \mref{thm:r-ENYBE}.
\vspb
\end{proof}
\delete{In fact, since $\calb$ is symmetric and invariant, for any $x,y,z \in A$, we have
\begin{eqnarray*}
&&\mathcal{B}(x\circ y,z)=-\mathcal{B}(y, x\star z)
\Leftrightarrow \langle \varphi(x\circ y),z \rangle =-\langle \varphi(y),x\star z \rangle
\\&&\Leftrightarrow \langle \varphi(L(x)y),z \rangle =\langle L_\star^*(x) \varphi(y), z \rangle
\Leftrightarrow \varphi(L(x)y)=L_\star^*(x) \varphi(y).
\end{eqnarray*}
Similarly, for any $x,y,z \in A$, we obtain
\begin{eqnarray*}
&&\mathcal{B}(z,x\circ y)=-\mathcal{B}(x\star z,y)
\Leftrightarrow \langle \varphi(z),x\circ y \rangle =-\langle \varphi(x\star z) ,y \rangle
\\&&\Leftrightarrow -\langle L^*(x)\varphi(z),y \rangle =-\langle\varphi(x\star z),y \rangle
\Leftrightarrow L^*(x)\varphi(z)=\varphi(x\star z).
\end{eqnarray*}
By the definition of $\star$, we also get $\varphi(R(y)x)=-R^*(y) \varphi(x)$, for any $x,y \in A$.}
\delete{Note that if we take specific value of $\kappa$ in Proposition \mref{dual exo}.(\romannumeral2) and  $\check{\delta_\pm}=\check{\Tvar} \pm \check{\bbvar} \in A\ot A$, then we have:

\begin{cor}
\begin{enumerate}
\item[(\romannumeral1)] If $\kappa=-1$, $\check{\delta_\pm}=\check{\Tvar} \pm \check{\bbvar}$ a solution of NYBE if and only if $T$ is an extended $\calo$-operator with extension $\beta$ of mass $(0,-1,0)$.
\item[(\romannumeral2)] If $\kappa=0$, $\check{\delta_\pm}=\check{\Tvar} \pm \check{\bbvar}$ is a solution of NYBE if and only if $T$ is a Rota-Baxter operator of weight zero.
\end{enumerate}
\end{cor}}

\begin{cor}\label{cor-qN}
Let $(A,\circ,\mathcal{B}(\cdot,\cdot))$ be a quadratic Novikov algebra, $\varphi:A\rightarrow A^*$ be the linear map defined by Eq.~(\mref{bilinear map}) and $r\in A\otimes A$. Define $\hr=\alpha+\beta$  where $\alpha$ and $\beta$ are defined by Eq.~(\mref{r:a and b}) and suppose that $\cbb \in A\otimes A$ is invariant. Then we get

\begin{enumerate}
\item[(\romannumeral1)] 
$r$ is a solution of the ENYBE of mass $\dfrac{\kappa +1}{4}$ in $A$ if and only if $\bar{\alpha} = \alpha\varphi :A\rightarrow A$ is an extended $\calo$-operator of weight $0$ with extension $\bar{\beta}=\beta\varphi :A\rightarrow A$ of mass $(\kappa,0)$ on $(A,\circ)$ associated to $(A,\circ, L_A, R_A)$.

\item[(\romannumeral2)] $r$ is a solution of the NYBE in $A$ if and only if $\bar{\alpha} = \alpha\varphi :A\rightarrow A$ is an extended $\calo$-operator of weight $0$ with extension $\bar{\beta}=\beta\varphi :A\rightarrow A$ of mass $(-1,0)$ on $(A,\circ)$ associated to $(A,\circ, L_A, R_A)$. In particular, if $r$ is skew-symmetric, then $r$ is a solution of the NYBE in $A$ if and only if $\bar{\alpha} : A\rightarrow A$ is a Rota-Baxter operator of weight $0$.
\end{enumerate}
\end{cor}

\begin{proof}
\delete{We claim that $\bar{\alpha}=\alpha\varphi$ is skew-adjoint with respect to $\mathcal{B}(~,~)$ and $\bar{\beta}=\beta\varphi$ is self-adjoint with respect to $\mathcal{B}(~,~)$, whereupon the proof is apparent thanks to the Proposition \mref{dual exo}. In fact,} Since $\caa \in A\otimes A$ is skew-symmetric,  we have
\begin{eqnarray*}
&&\langle \caa,\varphi(x)\ot \varphi(y)  \rangle = \langle -\caa,\varphi(y)\ot \varphi(x) \rangle
\\&&\Leftrightarrow  \langle \alpha(\varphi(x)),\varphi(y)  \rangle
=-\langle \alpha(\varphi(y)),\varphi(x) \rangle
\\&&\Leftrightarrow \mathcal{B}(\bar{\alpha}(x), y)=-\mathcal{B}(\bar{\alpha}(y), x), \quad\text{ $x,y\in A$}.
\end{eqnarray*}
Thus $\bar{\alpha}$ is skew-adjoint with respect to $\mathcal{B}(\cdot,\cdot)$. Similarly, we can obtain that $\bar{\beta}=\beta\varphi$ is self-adjoint with respect to $\mathcal{B}(\cdot,\cdot)$. Then this conclusion follows directly from Proposition \mref{dual exo}.
\end{proof}

\subsection{Extended $\calo$-operators in general and ENYBE}
\mlabel{ss:general o-operators}               
We now investigate the relationship between an extended $\calo$-operator $\alpha:V\rightarrow A$ in general and ENYBE.

\delete{We would like to end this section with a proof that an extend $\calo$-operators $\alpha:V\rightarrow A$ give rise to an extended $\calo$-operators on $\tA=A\ltimes_{l_\star^*,-r^*} V^*$ associated to the dual bimodule  $\tA^*$
both with finite $\bfk$-dimension.}

Let $(A,\circ)$ be a Novikov algebra and $(V,l_A,r_A)$ be an $A$-bimodule. Set $l_{A,\star}^\ast=l_A^\ast+r_A^\ast$. Then $(V^*,l_{A,\star}^*,-r_A^*)$ is an $A$-bimodule. Set $\tA=A\ltimes_{l_{A,\star}^*,-r_A^*} V^*$ by Proposition \ref{pp:dualrep}. It is easy to verify that the following diagram is commutative: for any $\gamma$ in $\Hom_{\bf k}(V,A)$,

\delete{\begin{eqnarray} \label{injective}
        \xymatrix{
 \txt{$\Hom_{\bf k}(V,A)$}\atop\text{${\gamma}$}\ar@{->}[rr]^_{\wedge}  \ar[d]_{\rm }&&
\txt{$V^\ast \ot A$} \atop {  \text{$\check{\gamma}=\sum_i v_i^\ast\otimes a_i$}} \ar[d]^(.6){}\\
 \txt{$\Hom_{\bf k}(V\oplus A^*,V^*\oplus A)$}\atop\text{$\mathcal{P}:={\gamma}^{\mathcal{A}}=\iota_2\circ \gamma \circ p_1$} \ar@{->}[rr]^_{$\wedge$}&& \txt{$(V^*\oplus A)\otimes  (V^*\oplus A)$}\atop\text{$\check{\mathcal{P}}:=\check{\tga}=\sum_i(v_i^\ast,0)\otimes (a_i,0) $}  },
\end{eqnarray}}
\begin{eqnarray} \label{injective}
        \xymatrix{
 \txt{$\Hom_{\bf k}(V,A)$}\atop\text{${\gamma}$}\ar@{->}[rr]^-{\vee}  \ar[d]_{\rm }&&
\txt{$V^\ast \ot A$} \atop {  \text{$\check{\gamma}=\sum_i v_i^\ast\ot a_i$}  } \ar[d]^(.6){}\\
 \txt{$\Hom_{\bf k}(V\oplus A^*,V^*\oplus A)$}\atop\text{$\mathcal{P}:={\gamma}^{\mathcal{A}}=\iota_2\circ \gamma \circ p_1$} \ar@{->}[rr]^-{\vee}&& \txt{$(V^*\oplus A)\otimes (V^*\oplus A) $}\atop\text{$\check{\mathcal{P}}:=\check{\tga}=\sum_i v_i^\ast\ot a_i= \check{\gamma}$}  },
\end{eqnarray}
where $\iota_2:A\rightarrow V^*\oplus A$ is the natural inclusion and $p_1: V\oplus A^*\rightarrow V$ is the usual projection.
Denote
\begin{eqnarray}  \label{aa+-bb}
\check{\mathcal{P}}_\pm:=\check{\tga_\pm}=\check{\gamma}\pm\tau(\check{\gamma}).
\vspa
\end{eqnarray}

\begin{lem}       \mlabel{bb and tbb}
Let $(A,\circ)$ be a Novikov algebra, $(V,l_A,r_A)$ be an $A$-bimodule, $\tA=A\ltimes_{l_{A,\star}^*,-r_A^*} V^*$ and $\beta \in \Hom_{\bf k}(V,A)$. Define $\mathcal{Q}=\beta^\tA$ in $\tA\ot\tA$ by Eq. (\mref{injective}) and $\check{\tbb}_+$ by Eq.~(\mref{aa+-bb}). Then $\tbb_+:\tA^*\rightarrow\tA$ is a balanced $\tA$-bimodule homomorphism  from $(\tA^*,L_{\tA,\star}^*,-R_{\tA}^*)$ to $(\tA,L_{\tA},R_{\tA})$ if and only if $\beta :V\rightarrow A$ is a balanced $A$-bimodule homomorphism  from $(V,l_A,r_A)$ to $(A,L_A,R_A)$.
\end{lem}
\begin{proof}
\delete{Before the proof, we claim that: for any $a^*\in A^*$ and $u\in V$, we have $\tbb_+(a^*)=\beta^*(a^*)$ and $\tbb_+(u)=\beta(u)$ where $\beta^*:A^*\rightarrow V^*$ is the dual linear map associated to $\beta$. In fact,} Since $\tbb_+=\tbb + \tbb^*$, for all $a^*\in A$ and $u\in V$, we have
\begin{eqnarray*}
\tbb(a^*)=\iota_2\circ \beta \circ p_1(a^*)=0,\;\;\tbb(u)=\iota_2\circ \beta \circ p_1(u)=\beta(u).
\end{eqnarray*}
Thus for all $b^*\in A^*$, $v\in V$, we obtain
\begin{eqnarray*}
&&\langle \tbb ^*(a^*),v\rangle=\langle a^*,\tbb(v)\rangle=\langle a^*,\beta(v)\rangle=\langle\beta^*(a^*),v\rangle,
\\
&&\langle \tbb ^*(a^*),b^*\rangle=\langle a^*,\tbb(b^*)\rangle=0,
\\
&&\langle \tbb ^*(u),a^*\rangle=\langle u,\tbb(a^*)\rangle=0,
\\
&&\langle \tbb ^*(u),v\rangle=\langle u,\tbb(v)\rangle=\langle u,\beta(v)\rangle=0.
\end{eqnarray*}
Therefore, one gets $\tbb_+(a^*)=(\tbb + \tbb^*)(a^*)=\beta^*(a^*)$ and $\tbb_+(u)=(\tbb + \tbb^*)(u)=\beta(u)$.

\delete{Now we can attack to the lemma. If $\kappa =0$, there is nothing to prove. Otherwise we just need to prove the case $\kappa =1$ thanks to the fact that $\bfk$ is a field.} Assume that $\beta :(V,l_A,r_A)\rightarrow (A, L_A, R_A)$ is a balanced $A$-bimodule homomorphism. Let $b^*\in A^*$ and $v\in V$. Then we obtain
\begin{align*}
L_{\tA,\star}^*\Big(\tbb_+(a^*+u)\Big)(b^*+v)
&=L_{\tA,\star}^*\Big(\tbb_+(a^*)\Big)b^*+L_{\tA,\star}^*\Big(\tbb_+(a^*)\Big)v
+L_{\tA,\star}^*\Big(\tbb_+(u)\Big)b^* +L_{\tA,\star}^*\Big(\tbb_+(u)\Big)v
\\
&=L_{\tA,\star}^*\Big(\beta^*(a^*)\Big)b^*+L_{\tA,\star}^*\Big(\beta^*(a^*)\Big)v
+L_{\tA,\star}^*\Big(\beta(u)\Big)b^* +L_{\tA,\star}^*\Big(\beta(u)\Big)v,
\end{align*}
and
\begin{equation*}
-R_{\tA}^*\Big(\tbb_+(b^*+v)\Big)(a^*+u)
=-R_{\tA}^*\Big(\beta^*(b^*)\Big)a^*-R_{\tA}^*\Big(\beta^*(b^*)\Big)u
-R_{\tA}^*\Big(\beta(v)\Big)a^* -R_{\tA}^*\Big(\beta(v)\Big)u.
\end{equation*}
On the other hand, for all $x\in A$ and $w^*\in V^*$, we obtain
{\small
\begin{eqnarray*}
&&\langle
L_{\tA,\star}^*\Big(\beta^*(a^*)\Big)b^*+R_{\tA}^*\Big(\beta^*(b^*)\Big)a^*   ,x
\rangle
=-\langle b^*,\beta^*(a^*)\bullet x + x\bullet\beta^*(a^*) \rangle -
 \langle a^*, x\bullet \beta^*(b^*)  \rangle=0,                       \\
&&\langle
L_{\tA,\star}^*\Big(\beta^*(a^*)\Big)b^*+R_{\tA}^*\Big(\beta^*(b^*)\Big)a^*  ,w^*
\rangle
=-\langle b^*,\beta^*(a^*)\bullet w^* + w^*\bullet\beta^*(a^*) \rangle -
 \langle a^*, w^*\bullet\beta^*(b^*)  \rangle=0,                  \\
&&\langle
L_{\tA,\star}^*\Big(\beta^*(a^*)\Big)v+R_{\tA}^*\Big(\beta(v)\Big)a^*   ,x
\rangle
  =-\langle v,\beta^*(a^*)\bullet x + x\bullet\beta^*(a^*) \rangle -
  \langle a^*, x\bullet\beta(v)  \rangle\\
 &&\quad=\langle -r_A(x)v,\beta^*(a^*)\rangle +\langle l_{A,\star}(x)v,\beta^*(a^*)\rangle
  -\langle a^*, x\circ \beta(v)  \rangle
  =\langle \beta(l_A(x)v)-x\circ\beta(v)            ,a^* \rangle  =0,     \\
&&\langle
L_{\tA, \star}^*\Big(\beta^*(a^*)\Big)v+R_{\tA}^*\Big(\beta(v)\Big)a^*   ,w^*
\rangle
  =-\langle v,\beta^*(a^*)\bullet w^* + w^*\bullet\beta^*(a^*) \rangle -
  \langle a^*, w^*\cdot\beta(v)  \rangle =0,\\
&&\langle
L_{\tA,\star}^*\Big(\beta(u)\Big)b^*+R_{\tA}^*\Big(\beta^*(b^*)\Big)u ,x\rangle
=-\langle b^*,\beta(u)\star x  \rangle -
 \langle u, x\bullet\beta^*(b^*)  \rangle
\\&&\quad= \langle b^*, -\beta(u)\star x +\beta(l_{A,\star}(x)u)\rangle =0,     \\
&&\langle
L_{\tA, \star}^*\Big(\beta(u)\Big)b^*+R_{\tA}^*\Big(\beta^*(b^*)\Big)u ,w^*
\rangle
=-\langle b^*,\beta(u)\bullet w^* + w^*\bullet\beta(u) \rangle -
 \langle a^*, w^*\bullet\beta^*(b^*)  \rangle=0,         \\
&&\langle
L_{\tA,\star}^*\Big(\beta(u)\Big)v+R_{\tA}^*\Big(\beta(v)\Big)u   ,x
\rangle
=-\langle v,\beta(u)\bullet x + x\bullet\beta(u) \rangle -
 \langle u, x\bullet\beta(v)  \rangle=0,         \\
&&\langle
L_{\tA,\star}^*\Big(\beta(u)\Big)v+R_{\tA}^*\Big(\beta(v)\Big)u  ,w^*
\rangle
=-\langle v,\beta(u)\bullet w^* + w^*\bullet\beta(u) \rangle -
 \langle u, w^*\bullet\beta(v)  \rangle
\\&&\quad =\langle   l_{A,\star} (\beta(u))v-r_A(\beta(u))v-r_A(\beta(v))u , w^*\rangle
    =\langle  l_A(\beta(u))v-r_A(\beta(v))u   \rangle=0.
\end{eqnarray*}}
Therefore, $L_{\tA,\star}^*\Big(\tbb_+(a^*+u)\Big)(b^*+v)
=-R_{\tA}^*\Big(\tbb_+(b^*+v)\Big)(a^*+u)$, which means that $\tbb$ is balanced. Furthermore, since $\check{\tbb}_+ \in \tA\otimes\tA $ is symmetric, by Lemma \mref{lem:r}, $\tbb_+$ is a balanced $\tA$-bimodule homomorphism.

Conversely, if $\tbb_+$ is a balanced $\tA$-bimodule homomorphism, then for all $u,v\in V$ and $x\in A$, we obtain
\begin{eqnarray*}
L_{\tA,\star}^*\Big(\tbb_+(u)\Big)(v)=-R_{\tA}^*\Big(\tbb_+(v)\Big)(u)
~\Leftrightarrow~
l_A(\beta(u))v=r_A(\beta(v))u,
\\
\tbb_+ \Big(  L_{\tA,\star}^*(x)v  \Big)= x\cdot \tbb_+(v)
~\Leftrightarrow~
\beta(l_A(x)v)=x\circ \beta(v),
\\
\tbb_+ \Big(  -R_{\tA}^*(x)u  \Big)=  \tbb_+(u)\cdot x
~\Leftrightarrow~
\beta(r_A(x)u)= \beta(u)\circ x.
\end{eqnarray*}
Therefore, $\beta :(V,l_A,r_A)\rightarrow (A, L_A, R_A)$ is a balanced $A$-bimodule homomorphism.
\end{proof}

\begin{thm}   \label{thm:taa exo}
Let $(A,\circ)$ be a Novikov algebra, $(V,l_A,r_A)$ be an $A$-bimodule  and $\tA=A\ltimes_{l_{A,\star}^*,-r_A^*} V^*$. Let $\alpha$, $\beta \in \Hom_{\bf k}(V,A)$ and  $\taa=\alpha^\tA$ and $\tbb=\beta^\tA$ by Eq. (\ref{injective})\delete{, which are identified as linear maps from $\tA^*\rightarrow\tA$}. Suppose that $\tbb_+:\tA^*\rightarrow\tA$ is a balanced $\tA$-bimodule homomorphism  from $(\tA^*,L_{\tA,\star}^*,-R_{\tA}^*)$ to $(\tA,L_{\tA},R_{\tA})$. Then $\alpha$ is an extended $\calo$-operator of weight $0$ with extension $\beta$ of mass $(\kappa,0)$ on $(A, \circ)$ associated to $(V, l_A, r_A)$ if and only if $\taa_-$ is an extended $\calo$-operator of weight $0$ with extension $\tbb_+$ of mass $(\kappa,0)$ on $(\tA, \bullet)$ associated to $(\tA^\ast, L_{\tA,\star}^\ast, -R_\tA^\ast)$.
\end{thm}

\begin{proof}
Similar to the proof of Lemma \mref{bb and tbb}, we can obtain that for all $a^*\in A$ and $u\in V$, $\taa_-(a^*)=-\alpha^*(a^*)$ and $\taa_-(u)=\alpha(u)$, where $\alpha^*:A^*\rightarrow V^*$ is the dual linear map of $\alpha$. \delete{In fact, we have $\taa_-=\taa - {\taa}^*$. Moreover, for any $a^*\in A$ and $u\in V$,
\begin{eqnarray*}
\taa(a^*)=\iota_2\circ \alpha \circ p_1(a^*)=0,\\
\taa(u)=\iota_2\circ \alpha \circ p_1(u)=\alpha(u).
\end{eqnarray*}
Where, for vector spaces $V_i,~i=1,2$, $\iota_i:V_i\rightarrow V_1\oplus V_2$ is the usual inclusion and  $p_i:V_1\oplus V_2\rightarrow V $ is the usual projection.
Hence for any $b^*\in A^*$, $v\in V$,
\begin{eqnarray*}
&&\langle {\taa} ^*(a^*),v\rangle=\langle a^*,\taa(v)\rangle=\langle a^*,\alpha(v)\rangle=\langle\alpha^*(a^*),v\rangle;
\\
&&\langle {\taa} ^*(a^*),b^*\rangle=\langle a^*,\taa(b^*)\rangle=0;
\\
&&\langle {\taa} ^*(u),a^*\rangle=\langle u,\taa(a^*)\rangle=0;
\\
&&\langle {\taa} ^*(u),v\rangle=\langle u,\taa(v)\rangle=\langle u,\alpha(v)\rangle=0.
\end{eqnarray*}
Thus our first assertion holds. }Suppose that $\alpha$ is an extended $\calo$-operator of weight $0$ with extension $\beta$ of mass $(\kappa,0)$ on $(A, \circ)$ associated to $(V, l_A, r_A)$ . Then for all $a^*$, $b^*\in A^*$ and $u$, $v\in V$, we have

\begin{eqnarray*}
&&\quad  \taa_-(u+a^*) \bullet \taa_-(v+b^*)  -
\taa_-\bigg(      L_{\tA,\star}^*\Big(\taa_-(u+a^*)\Big)(v+b^*)
            -R_{\tA}^*\Big(\taa_-(v+b^*)\Big)(u+a^*)        \bigg)
\\&&
=\alpha(u)\bullet\alpha(v)-\taa_-\bigg( L_{\tA,\star}^*\Big(\alpha(u)\Big)v -R_{\tA}^*\Big(\alpha(v)\Big)u \bigg)
-\alpha(u)\bullet\alpha^*(b^*)-\alpha^*(a^*)\bullet\alpha(v)
\\&&+\alpha^*(a^*)\bullet\alpha^*(b^*)
-\taa_-\bigg(
L_{\tA,\star}^*\Big(\alpha(u)\Big)b^* -L_{\tA,\star}^*\Big(\alpha^*(a^*)\Big)v -L_{\tA,\star}^*\Big(\alpha^*(a^*)\Big)b^*
     \\&&  -R_{\tA}^*\Big(\alpha(v)\Big)a^*
+R_{\tA}^*\Big(\alpha^*(b^*)\Big)u +R_{\tA}^*\Big(\alpha^*(b^*)\Big)a^*
 \bigg)
\\&&
=\alpha(u)\circ\alpha(v)-\alpha \Big( l_A(\alpha(u))v+r_A(\alpha(v)) u \Big)
+\alpha^*\bigg( L_{\tA,\star}^*\Big(\alpha(u)\Big)b^*\bigg)
-\alpha(u)\bullet\alpha^*(b^*)  \\&&-\alpha^*\bigg( -R_{\tA}^*\Big(\alpha^*(b^*)\Big)u\bigg)
-\alpha^*(a^*)\bullet\alpha(v) -\alpha^*\bigg(L_{\tA,\star}^*\Big(\alpha^*(a^*)\Big)v\bigg)+\alpha^*\bigg(-R_{\tA}^*\Big(\alpha(v)\Big)a^* \bigg).
\end{eqnarray*}
On the other hand, for all $w\in V$, we have
\begin{eqnarray*}
&&\Big\langle
\alpha^*\bigg( L_{\tA,\star}^*\Big(\alpha(u)\Big)b^*\bigg)
-\alpha(u)\bullet\alpha^*(b^*)  -\alpha^*\bigg( -R_{\tA}^*\Big(\alpha^*(b^*)\Big)u\bigg)                               ,w\Big\rangle
\\&&=\Big\langle
\alpha^*\Big( L_{\tA,\star}^*\Big(\alpha(u)\Big)b^*\Big)
-l_{A,\star}^*(\alpha(u))\alpha^*(b^*) -\alpha^*\Big( -R_{\tA}^*\Big(\alpha^*(b^*)\Big)u \Big)                 ,w\Big\rangle
\\&&=\Big\langle b^*,-\alpha(u)\star\alpha(w)+\alpha\Big( l_{A,\star}(\alpha(w))u \Big) +\alpha\Big( l_{A,\star}(\alpha(u))w \Big)
\Big\rangle
=-\Big\langle b^*, \kappa\beta(u)\star\beta(w)\Big\rangle
\\&&=-\Big\langle b^*, \kappa\beta\Big(  l_\star(   \beta(u)   )w         \Big)
\Big\rangle
=\Big\langle \kappa l_\star^*(   \beta(u)   )(\beta^*(b^*)), w
\Big\rangle .
\end{eqnarray*}
Therefore, one gets
$$\alpha^*\bigg( L_{\tA,\star}^*\Big(\alpha(u)\Big)b^*\bigg)
-\alpha(u)\bullet\alpha^*(b^*)  -\alpha^*\bigg( -R_{\tA}^*\Big(\alpha^*(b^*)\Big)u\bigg)=
\kappa l_{A,\star}^*(   \beta(u)   )(\beta^*(b^*)).$$
Similarly, we obtain
\begin{eqnarray*}
&&\Big\langle
-\alpha^*(a^*)\bullet \alpha(v) -\alpha^*\bigg(L_{\tA,\star}^*\Big(\alpha^*(a^*)\Big)v\bigg)+\alpha^*\bigg(-R_{\tA}^*\Big(\alpha(v)\Big)a^* \bigg),   w
\Big\rangle
\\&&=\Big\langle  a^*,-\alpha\Big(  r_A(\alpha(v)) w  \Big)-\alpha\Big(  l_A(\alpha(w)) v  \Big) +\alpha(w)\circ\alpha(v)
\Big\rangle
=\Big\langle a^*, \kappa\beta(w)\circ\beta(v)\Big\rangle
\\&&=\Big\langle a^*, \kappa\beta\Big(  r_A(   \beta(v)   )w         \Big)
\Big\rangle
=-\Big\langle \kappa r_A^*(   \beta(v)   )(\beta^*(a^*)), w
\Big\rangle .
\end{eqnarray*}
Hence, we obtain
$$-\alpha^*(a^*)\bullet\alpha(v) -\alpha^*\bigg(L_{\tA,\star}^*\Big(\alpha^*(a^*)\Big)v\bigg)+\alpha^*\bigg(-R_{\tA}^*\Big(\alpha(v)\Big)a^* \bigg)=
-\kappa r^*(   \beta(v)   )(\beta^*(a^*)).$$
So
\begin{eqnarray*}
&&\quad\taa_-(u+a^*) \bullet\taa_-(v+b^*)  -
\taa_-\bigg(      L_{\tA,\star}^*\Big(\taa_-(u+a^*)\Big)(v+b^*)
            -R_{\tA}^*\Big(\taa_-(v+b^*)\Big)(u+a^*)        \bigg)
\\&&=\kappa\beta(u)\circ\beta(v) + \kappa l_{A,\star}^*(   \beta(u)   )(\beta^*(b^*))
-\kappa r_A^*(   \beta(v)   )(\beta^*(a^*))
\\&&=\kappa\tbb_+(u) \bullet \tbb_+(v) +\kappa\tbb_+(u)\bullet \tbb_+^*(b^*) + \kappa\tbb_+^*(a^*)\bullet \tbb_+(v)  =\kappa \tbb_+(u+a^*) \bullet  \tbb_+(v+b^*).
\end{eqnarray*}


\delete{If $\kappa =0$, the above equation implies that $\taa_-$ is an extended $\calo$-operator of weight zero. Otherwise use} By Lemma \mref{bb and tbb}, we obtain that $\tbb_+$ from $(\tA^*,L_{\tA,\star}^*,-R_{\tA}^*)$ to $(\tA, L_{\tA}, R_{\tA})$
is a balanced $\tA$-bimodule homomorphism. Therefore, $\taa_-$ is an extended $\calo$-operator of weight $0$ with extension $\tbb_+$ of mass $(\kappa, 0)$  on $(\tA, \bullet)$ associated to $(\tA^\ast, L_{\tA,\star}^\ast, -R_\tA^\ast)$.

Conversely, if $\taa_-$ is an extended $\calo$-operator of weight $0$ with extension $\tbb_+$ of mass $(\kappa, 0)$  on $(\tA, \bullet)$ associated to $(\tA^\ast, L_{\tA,\star}^\ast, -R_\tA^\ast)$, then for all $u$, $v\in V$, we have $$\taa_-(u) \bullet \taa_-(v)  -
\taa_-\bigg(      L_{\tA,\star}^*\Big(\taa_-(u)\Big)(v)
            -R_{\tA}^*\Big(\taa_-(v)\Big)(u)        \bigg)
=\kappa \tbb_+(u) \bullet\tbb_+(v).$$
Hence
$$\alpha(u)\circ \alpha(v)-\alpha \Big( l_A(\alpha(u))v+r_A(\alpha(v)) u \Big)
=\kappa \beta(u) \circ \beta(v).$$
\delete{Similarly, if $\kappa =0$, the above equation implies that $\alpha$ is an $\calo$-operator. If $\kappa\neq 0$, considering Lemma \mref{bb and tbb}, $\beta$ is a balanced $A$-bimodule homomorphism of mass $\kappa$ and thus $\alpha$ is an extended $\calo$-operator with extension $\beta$ of mass $\kappa$  on $(\tA, \bullet)$ associated to $(\tA^\ast, L_{\tA,\star}^\ast, -R_\tA^\ast)$.}
Therefore, this conclusion holds.

\end{proof}


\begin{cor}     \label{cor-gN}
Let $(A,\circ)$ be a Novikov algebra, $(V,l_A,r_A)$ be an $A$-bimodule  Novikov algebra and $\tA=A\ltimes_{l_{A,\star}^*,-r_A^*} V^*$. Let $\alpha$, $\beta \in \Hom_{\bf k}(V,A)$ and  $\taa=\alpha^\tA$ and $\tbb=\beta^\tA$ defined by Eq. (\ref{injective}). Suppose that $\beta$  is a balanced $A$-bimodule homomorphism  from $(V^*,l_{A,\star}^*,-r_{A}^*)$ to $(A,L_{A},R_{A})$. Then the following conclusions hold.
\begin{enumerate}
\item[(\romannumeral1)]
\begin{enumerate}
  \item$\alpha$ is an  $\calo$-operator of weight $0$ with extension $\beta$ of mass $(\kappa,0)$ on $(A,\circ)$ associated to $(V, l_A, r_A)$ if and only if $\check{\taa}_-\pm \check{\tbb}_+ $is a solution of the ENYBE of mass $\dfrac{\kappa+1}{4}$ in $A\ltimes_{l_{A,\star}^*,-r_A^*} V^*$.
  \item  \cite[Theorem 3.29]{HBG} $\alpha$ is an $\calo$-operator of weight $0$ on $(A,\circ)$ associated to $(V, l_A, r_A)$ if and only if $\check{\taa}_-$ is a skew-symmetric solution of the NYBE in $A\ltimes_{l_{A,\star}^*,-r_A^*} V^*$. In particular, a linear endomorphism $T:A\rightarrow A$ is a Rota-Baxter operator of weight zero if and only if $\check{\calt}_-=\check{T}-\tau(\check{T})$ is a skew-symmetric solution of the NYBE in $A\ltimes_{L_{A,\star}^*,-R_A^*} A^*$.
  \item$\alpha$ is an extended $\calo$-operator of weight $0$ with extension $\beta$ of mass $(-1,0)$ on $(A,\circ)$ associated to $(V, l_A, r_A)$ if and only if $\check{\taa}_-\pm \check{\tbb}_+$ is a solution of the NYBE in $A\ltimes_{l_{A,\star}^*,-r_A^*} V^*$.
  \end{enumerate}

\item[(\romannumeral2)] Let $(V,l_A,r_A)=(A, L_A, R_A)$ and $\beta=\id$. Then we have
\begin{enumerate}
\item $\alpha$ satisfies Eq. (\mref{Bax bb=id k=-1}) if and only if $\check{\taa}_- \pm \check{\id}_+$ is a solution of the NYBE in $A\ltimes_{L_{A,\star}^*,-R_A^*} A^*$.
\item Let $T: A\rightarrow A$ be a linear map and $\calt=T^{\tA}$ defined by Eq. (\ref{injective}). Then $T$ is a Rota-Baxter operator of weight $\lambda \neq 0$ if and only if both $\dfrac{2}{\lambda}(\check{\calt}-\tau(\check{\calt}))+2\check{\id}$ and $\dfrac{2}{\lambda}(\check{\calt}-\tau(\check{\calt}))-2\check{\id}$ are solutions of the NYBE in $A\ltimes_{L_{A,\star}^*,-R_A^*} A^*$.
\end{enumerate}


\end{enumerate}
\end{cor}

\begin{proof}
(\romannumeral1)( a)
This follows from Theorems \mref{thm:r-ENYBE} and \mref{thm:taa exo}.
\smallskip

\noindent
(\romannumeral1)(b) This follows from Theorem \mref{thm:r-ENYBE} for $\kappa =0$ or $\beta=0$ and Corollary \mref{cor:o-ENYBE}.
\smallskip

\noindent
(\romannumeral1)(c) This follows from Theorem \mref{thm:r-ENYBE} for $\kappa =-1$ and Corollary \mref{cor:o-ENYBE}.
\smallskip

\noindent
(\romannumeral2)(a) This follows from  Item (\romannumeral1) directly.
\smallskip

\noindent
(\romannumeral2)(b) By the discussion after Corollary \mref{cor:Bax}, $T$ is a Rota-Baxter operator of weight $\lambda \neq 0$ if and only if $\dfrac{2T}{\lambda} +\id$ is an extended $\calo$-operator of weight $0$ with extension $\beta=\id$ of mass $(\kappa=-1,0)$ on $(A,\circ)$ associated to $(A, L_A, R_A)$, i.e., $\dfrac{2T}{\lambda} +\id$ satisfies Eq.~(\mref{Bax bb=id k=-1}). Then the conclusion follows from Item (\romannumeral2)(a).
\vspb
\end{proof}

\delete{\begin{defi}
Let $(A,\circ)$ be a Novikov algebra. If $r\in A\ot A$ is a solution of the NYBE in $A$ and the symmetric part $\check{\beta}$ of $r$ is invariant, then the induced Novikov bialgebra $(A,\circ, \Delta_r)$ is called {\bf quasitriangular Novikov bialgebra}.
\end{defi}
\yy{The definition of Novikov bialgebras? $\Delta_r$?}}


\section{Extended  $\calo$-operators and generalized Novikov Yang-Baxter equations}
\mlabel{sec:gnybe and operator}
In this section, we will introduce the notion of generalized Novikov Yang-Baxter equations and establish the relationship between extended $\calo$-operators and  generalized Novikov Yang-Baxter equations.

We first introduce the definition of generalized Novikov Yang-Baxter equations.

\begin{defi}
Let $(A,\circ)$ be a Novikov algebra. An element $r\in A\otimes A$ is called a solution of the
{\bf generalized Novikov Yang-Baxter equations (GNYBES)} in  $A$, if it satisfies
\vspa
{\small
\begin{eqnarray}
&&\Big(L_A(a)\otimes \id\otimes \id-\id\otimes L_A(a)\otimes \id\Big)\Big((\tau r)_{12}\circ r_{13}+r_{12}\circ r_{23}+r_{13}\star r_{23}\Big)\nonumber\\
&&\hspace{0.3cm}+\Big((\id\otimes L_A(a)\otimes
\id)(r+\tau r)_{12}\Big)\circ r_{23}-\Big((L_A(a)\otimes \id\otimes \id)r_{13}\Big)\circ (r+\tau r)_{12}+\Big(\id\otimes \id\otimes
L_{A,\star}(a)\Big)\label{cob6}\\
&&\hspace{0.3cm}\Big(r_{23}\circ r_{13}-r_{13}\circ
r_{23}-(\id\otimes\id\otimes \id-\tau\otimes \id)(r_{13}\circ
r_{12}+r_{12}\star r_{23})\Big)=0,
    \nonumber
\end{eqnarray}
\begin{equation}        \label{cob7}
     (\id\otimes \id\otimes \id-\id\otimes \tau)(\id\otimes \id\otimes
     L_{A,\star}(a))(r_{13}\circ (\tau r)_{23} -r_{12}\star
     r_{23}-r_{13}\circ r_{12})=0,\;\;  a\in A.
\end{equation}
}
\vspa
\end{defi}

\begin{rmk}\label{def-gnybe}
Let $(A, \circ)$ be a Novikov algebra and $r\in A\otimes A$. Define $\Delta_r: A\rightarrow A\otimes A$ by
\begin{eqnarray*}
\Delta_r(a):=(L_A(a)\otimes \id+\id\otimes L_{A, \star}(a))r,\;\;a\in A.
\end{eqnarray*}
By \cite[Theorem 3.22]{HBG}, $(A, \circ, \Delta_r)$ is a Novikov bialgebra defined in \cite{HBG} if and only if $r$ is a solution of the GNYBES in $A$ and the following equalities hold for all $a$, $b\in A$:
\begin{eqnarray*}
&&(\id\otimes (L_A(b\circ a)+L_A(a)L_A(b))+L_{A,\star}(a)\otimes L_{A,\star}(b))(r+\tau r)=0,\\
&&(-L_{A,\star}(b)\otimes R_A(a)+L_{A,\star}(a)\otimes R_A(b)+R_A(a)\otimes L_A(b)-R_A(b)\otimes L_A(a)\\
&&\qquad+\id\otimes (L_A(a)L_A(b)-L_A(b)L_A(a))-(L_A(a)L_A(b)-L_A(b)L_A(a))\otimes \id)(r+tr)=0.
\end{eqnarray*}
Let $\circ_\Delta$ be the binary operation on $A^*$ induced by $\Delta_r$, i.e.,
\begin{equation} \label{duct on A^*}
\langle a^*\qsan b^* ,x \rangle=\langle  a^*\otimes b^*,\Delta_r(x)\rangle, \quad   x\in A,a^*,b^*\in A^*.
\end{equation}
Then by \cite[Lemma 3.21]{HBG}, $r$ is a solution of the  GNYBES in $A$ if and only if
$(A^\ast, \qsan)$ is a Novikov algebra.
\end{rmk}

\begin{lem}\mlabel{product on A^*}
Let $(A,\circ)$ be a Novikov algebra and $r\in  A\ot A $. \delete{Let $\circ_\Delta$ be the multiplication on $A^*$ induced by Eq.(\mref{co1}), defined by
\begin{equation} \label{duct on A^*}
\langle a^*\qsan b^* ,x \rangle=\langle  a^*\otimes b^*,\Delta_r(x)\rangle, \quad   x\in A,\;\; a^*, b^*\in A^*.
\end{equation}
Regarding $r$ as a linear map $:A^*\rightarrow A$,} Then the binary operation $\qsan$ on $A^\ast$ defined by Eq. (\ref{duct on A^*}) is also given by
\begin{equation} 
 a^*\qsan b^* =-\Big( L_{A,\star}^*(\hr(a^*))b^*+R_A^*(\hr^t(b^*))a^*  \Big),\quad   a^*,\;b^*\in A^*.
\end{equation}
Let $\alpha=\hr_-$ and $\beta=\hr_+$. Suppose that $r_+$ is invariant. Then we have
\small{
\begin{equation} \mlabel{sym:duct on A^*}
 a^*\qsan b^* = -\Big(L_{A,\star}^*(\alpha(a^*))b^*-R_A^*(\alpha(b^*))a^*\Big)
 =-\Big(L_{A,\star}^*(\hr_-(a^*))b^*-R_A^*(\hr_-(b^*))a^*\Big),  \;a^*,b^*\in A^*.
\end{equation}
}
\end{lem}
\begin{proof} 
Let
$\{e_1, \ldots, e_n\}$ be a basis of $A$ and $\{e_1^\ast, \ldots,
e_n^\ast\}$ be its dual basis.
Set $r=\sum_{i,j}a_{ij}e_i\ot e_j$ and $e_i\circ e_j=\sum_{k=1}^nc_{i,j}^k e_k$. Then we have
\vspb
\begin{eqnarray*}
e_k^*\qsan e_l^*&&=\sum_{s=1}^n\langle e_k^*\ot e_l^*,\Delta_r(e_s)\rangle e_s^*=\sum_{s=1}^n\langle e_k^*\ot e_l^*,(L_A(e_s)\ot \id+\id\ot L_{A,\star}(e_s))r\rangle e_s^*
\\&&=\sum_{s=1}^n\langle e_k^*\ot e_l^*,\sum_{i,j}a_{ij} \sum_{t=1}^nc_{s,i}^t e_t \ot e_j   + \sum_{i,j}a_{ij} e_i\ot  \sum_{r=1}^n(c_{s,j}^r + c_{j,s}^r) e_r \rangle e_s^*
\\&&=\sum_{s,i} (a_{il}c_{si}^k)e_s^* +\sum_{s,j} a_{kj}(c_{sj}^l+c_{js}^l)e_s^* =\sum_{s,r} (a_{rl}c_{sr}^k  +  a_{kr}(c_{sr}^l+c_{rs}^l)    )e_s^*
\\&&=(-R_A^*)(\hr^t(e_l^*))e_k^* -L_{A,\star}^*(\hr(e_k^*))e_l^*
=-\Big(L_{A,\star}^*(\hr(e_k^*))e_l^*+R_A^*(\hr^t(e_l^*))e_k^*   \Big)   .
\end{eqnarray*}
Therefore the first conclusion holds. Note that $\widehat{r}^t=-\alpha+\beta $. Then the second conclusion follows immediately  by Lemma \mref{lem:r}.
\end{proof}

\begin{cor}       \mlabel{solution of GNYBE}
Let $(A,\circ)$ be a Novikov algebra. Suppose that $r\in A\otimes A$ is skew-symmetric. Then $r$ is a solution of the GNYBES in $A$ if and only if $(A, \qsan)$ is a Novikov algebra where $\qsan$ is defined by Eq.~(\mref{sym:duct on A^*}).
\end{cor}
\begin{proof}
It follows directly from Remark \ref{def-gnybe} and Lemma \mref{product on A^*}.
\end{proof}

\begin{pro}     \mlabel{pro:gnybe}
Let $(A,\circ)$ be a Novikov algebra and $r\in A\ot A $. Define $\hr=\alpha+\beta$ by Eq.~(\mref{r:a and b}) and suppose that $\check{\beta} \in A\otimes A$ is invariant. \delete{that is, $\beta$ is a balanced $A$-bimodule homomorphism as a liner map.} If $\alpha$ is an extended $\calo$-operator of weight $0$ with extension $\beta$ of mass $(\kappa,0)$ on $(A, \circ)$ associated to $(A^\ast, L_{A,\star}^\ast, -R_A^\ast)$, then the binary operation defined by Eq.~(\mref{sym:duct on A^*}) defines a Novikov algebra structure on $A^*$ and $r$ is a solution of the GNYBES in $A$.
\end{pro}
\begin{proof}
It follows directly from Theorem \mref{delta +-}, Lemma \mref{product on A^*} and Remark \ref{def-gnybe}.
\delete{Use Theorem \mref{delta +-} to the $A$-bimodule Novikov algebra $(A^*,L_\star^* ,-R^*)$ with zero multiplication and the first assertion follows immediately. Moreover, thanks to the Lemma \mref{product on A^*}, $r$ is a solution of GNYBES.}
\end{proof}

\begin{cor}
With the assumption as in Proposition \mref{pro:gnybe}, a solution of the ENYBE of mass $\kappa \in \bfk$ in $A$ is also a solution of the GNYBES in $A$.
\end{cor}

\begin{proof}
Let $r$ be a solution of the ENYBE of mass $\kappa$ in $A$ and define $\hr=\alpha+\beta$ by Eq.~(\mref{r:a and b}). By Theorem \mref{thm:r-ENYBE}, $\alpha$ is an extended $\calo$-operator of weight $0$ with extension $\beta$ of mass $( 4\kappa-1,0)$ on $(A, \circ)$ associated to $(A^\ast, L_{A,\star}^\ast, -R_A^\ast)$. Then by Proposition \mref{pro:gnybe}, $r$ is a solution of the GNYBES in $A$.
\end{proof}

\begin{pro}                                             \mlabel{pro:*v}
Let $(A,\circ)$ be a Novikov algebra and $(V,l_A,r_A) $ be an $A$-bimodule Novikov algebra. Let $\alpha : V \rightarrow A $ be a linear map and define a binary operation on $V$ as follows:
\begin{equation} 
u*v:=l_A(\alpha(u))v+r_A(\alpha(v))u,\quad  u, v \in V.
\end{equation}
Then $(V,\ast)$ is a Novikov algebra if and only if the following equalities hold:
\begin{eqnarray}                                        \mlabel{*vcon1}
&&l_A\Big(\alpha(u)\circ\alpha(v)- \alpha(u*v)\Big)w=l_A\Big(\alpha(u)\circ\alpha(w)-\alpha(u*w)\Big)v,\\
             \mlabel{*vcon2}
&&l_A\Big( \alpha(u)\circ\alpha(v) -  \alpha(u*v)\Big)w  -
l_A \Big(  \alpha(v)\circ\alpha(u)-  \alpha(v*u)\Big) w
\\&&= \notag
r_A\Big(    \alpha(v)\circ\alpha(w)  -  \alpha(v*w)   \Big) u-
r_A\Big(    \alpha(u)\circ\alpha(w)   -\alpha(u*w)  \Big) v  ,
\qquad  u, v, w \in V.
\end{eqnarray}
\end{pro}
\begin{proof}
It follows from Proposition \mref{pro:*} directly when $(M, \cdot, l_A, r_A)=(V, l_A, r_A)$.
\end{proof}

\begin{thm}  \mlabel{Goper con}
Let $(A,\circ)$ be a Novikov algebra, $(V,l_A,r_A)$ be an $A$-bimodule and $\tA=A\ltimes_{l_{A,\star},r_A}V^\ast$. Let $\alpha \in \Hom_{\bf k}(V,A)$ and $\taa=\alpha^\tA$ by Eq. (\mref{injective}). Then $\check{\taa}_-\in \tA\ot\tA$ is a skew-symmetric solution of the GNYBES in $\tA$ if and only if Eqs.~(\mref{*vcon1}), (\mref{*vcon2}) and the following equalities hold:
\begin{eqnarray}                      \mlabel{Goper con3}
&&x\star B_\alpha(u,w)= B_\alpha(l_A(x)w,u)+ B_\alpha(u,l_A(x)w),\\
&&                  \mlabel{Goper con4}
B_\alpha(l_A(x)w,v) = B_\alpha(l_A(x)v,w),\\
&&                   \mlabel{Goper con5}
B_\alpha(l_A(x)w,u)+B_\alpha(u,l_A(x)w)=
x\circ B_\alpha(u,w)  +B_\alpha(r_A(x)u,w),\\
&&                  \mlabel{Goper con6}
B_\alpha(r_A(x)u,v) +B_\alpha(v,r_A(x)u)
=B_\alpha(r_A(x)v,u)+B_\alpha(u,r_A(x)v),
\end{eqnarray}
for all $u,v,w\in V$ and $x\in A$, where
\begin{equation}
B_\alpha(u,w)=\alpha(u)\circ\alpha(v) -  \alpha\Big( l_A(\alpha(u))v+r_A(\alpha(v)) u \Big), \quad\quad u,v\in V.
\end{equation}
\end{thm}
If a linear operator $\alpha: V\rightarrow A$ satisfies Eqs. (\ref{Goper con3})-(\ref{Goper con6}), then $\alpha$ is called a {\bf generalized $\calo$-operator} on $(A, \circ)$ associated to $(V, l_A, r_A)$.
\begin{proof}
By Lemma \mref{product on A^*}, Corollary \mref{solution of GNYBE} and Proposition \mref{pro:*v}, $\check{\taa}_- \in \tA\ot\tA$ is a skew-symmetric solution of the GNYBES in $\tA$ if and only if for all $u,v,w\in V$ and $a^*,b^*,c^*\in A^*$, the following equalities hold:
\small{
\begin{eqnarray*}
(\ast)\;\;L_{\tA,\star}^*\Bigg(\taa_-(u+a^*) \bullet \taa_-(v+b^*)  -
\taa_-\bigg(      L_{\star,\tA}^*\Big(\taa_-(u+a^*)\Big)(v+b^*)
    -R_{\tA}^*\Big(\taa_-(v+b^*)\Big)(u+a^*)     \bigg)\Bigg)(w+c^*)
\\=L_{\tA,\star}^*\Bigg(\taa_-(u+a^*) \bullet \taa_-(w+c^*)  -
\taa_-\bigg(      L_{\tA,\star}^*\Big(\taa_-(u+a^*)\Big)(w+c^*)
     -R_{\tA}^*\Big(\taa_-(w+c^*)\Big)(u+a^*)        \bigg)\Bigg)(v+b^*),\nonumber\\
\label{g-eq-2}
(\ast\ast)\;\;L_{\tA,\star}^*\Bigg(\taa_-(u+a^*) \bullet \taa_-(v+b^*)  -
\taa_-\bigg(      L_{\tA,\star}^*\Big(\taa_-(u+a^*)\Big)(v+b^*)
    -R_{\tA}^*\Big(\taa_-(v+b^*)\Big)(u+a^*)     \bigg)\Bigg)(w+c^*)
\\\qquad-L_{\tA,\star}^*\Bigg(\taa_-(v+b^*) \bullet \taa_-(u+a^*)  -
\taa_-\bigg(      L_{\tA,\star}^*\Big(\taa_-(v+b^*)\Big)(u+a^*)
    -R_{\tA}^*\Big(\taa_-(u+a^*)\Big)(v+b^*)     \bigg)\Bigg)(w+c^*)
\nonumber\\=
-R_{\tA}^*\Bigg(\taa_-(v+b^*) \bullet\taa_-(w+c^*)  -
\taa_-\bigg(      L_{\tA,\star}^*\Big(\taa_-(v+b^*)\Big)(w+c^*)
     -R_{\tA}^*\Big(\taa_-(w+c^*)\Big)(v+b^*)        \bigg)\Bigg)(u+a^*)
\nonumber\\\qquad+
R_{\tA}^*\Bigg(\taa_-(u+a^*) \bullet \taa_-(w+c^*)  -
\taa_-\bigg(      L_{\tA,\star}^*\Big(\taa_-(u+a^*)\Big)(w+c^*)
     -R_{\tA}^*\Big(\taa_-(w+c^*)\Big)(u+a^*)        \bigg)\Bigg)(v+b^*).\nonumber
\end{eqnarray*}
}

On the one hand, by the proof of Theorem \mref{thm:taa exo}, Eq. ($\ast$) is equivalent to
\small{
\begin{eqnarray*}
&& L_{\star,\tA}^*\Bigg[
\alpha(u)\circ\alpha(v)-\alpha \Big( l_A(\alpha(u))v+r_A(\alpha(v)) u \Big)
+\alpha^*\bigg( L_{\tA,\star}^*\Big(\alpha(u)\Big)b^*\bigg)
-\alpha(u)\bullet \alpha^*(b^*)  \\&&-\alpha^*\bigg( -R_{\tA}^*\Big(\alpha^*(b^*)\Big)u\bigg)
-\alpha^*(a^*)\bullet\alpha(v) -\alpha^*\bigg(L_{\tA,\star}^*\Big(\alpha^*(a^*)\Big)v\bigg)+\alpha^*\bigg(-R_{\tA}^*\Big(\alpha(v)\Big)a^* \bigg)
                          \Bigg]w
\\&&+ L_{\star,\tA}^*\Bigg[
\alpha(u)\circ\alpha(v)-\alpha \Big( l_A(\alpha(u))v+r_A(\alpha(v)) u \Big)
+\alpha^*\bigg( L_{\star,\tA}^*\Big(\alpha(u)\Big)b^*\bigg)
-\alpha(u)\bullet\alpha^*(b^*)  \\&&-\alpha^*\bigg( -R_{\tA}^*\Big(\alpha^*(b^*)\Big)u\bigg)
-\alpha^*(a^*)\cdot\alpha(v) -\alpha^*\bigg(L_{\tA,\star}^*\Big(\alpha^*(a^*)\Big)v\bigg)+\alpha^*\bigg(-R_{\tA}^*\Big(\alpha(v)\Big)a^* \bigg)
                          \Bigg]c^*
\\&&\hspace{-.3cm}=L_{\tA,\star}^*\Bigg[
\alpha(u)\circ\alpha(w)-\alpha \Big( l_A(\alpha(u))w+r(\alpha(w)) u \Big)
+\alpha^*\bigg( L_{\tA,\star}^*\Big(\alpha(u)\Big)c^*\bigg)
-\alpha(u)\bullet\alpha^*(c^*)  \\&&-\alpha^*\bigg( -R_{\tA}^*\Big(\alpha^*(c^*)\Big)u\bigg)
-\alpha^*(a^*)\cdot\alpha(w) -\alpha^*\bigg(L_{\tA,\star}^*\Big(\alpha^*(a^*)\Big)w\bigg)+\alpha^*\bigg(-R_{\tA}^*\Big(\alpha(w)\Big)a^* \bigg)
                          \Bigg]v
\\&&+L_{\tA,\star}^*\Bigg[
\alpha(u)\circ\alpha(w)-\alpha \Big( l_A(\alpha(u))w+r_A(\alpha(w)) u \Big)
+\alpha^*\bigg( L_{\tA,\star}^*\Big(\alpha(u)\Big)c^*\bigg)
-\alpha(u)\bullet\alpha^*(c^*)  \\&&-\alpha^*\bigg( -R_{\tA}^*\Big(\alpha^*(c^*)\Big)u\bigg)
-\alpha^*(a^*)\bullet\alpha(w) -\alpha^*\bigg(L_{\tA,\star}^*\Big(\alpha^*(a^*)\Big)w\bigg)+\alpha^*\bigg(-R_{\tA}^*\Big(\alpha(w)\Big)a^* \bigg)
                          \Bigg]b^*.
\end{eqnarray*}
}
It is easy to see that the equality above holds if and only if the following equalities hold:
\begin{eqnarray}          \mlabel{abc*=0}
&& L_{\tA,\star}^*\Bigg(
\alpha(u)\circ\alpha(v)-\alpha \Big( l_A(\alpha(u))v+r_A(\alpha(v)) u \Big)
                          \Bigg)w
\\&& \notag =L_{\tA,\star}^*\Bigg(
\alpha(u)\circ\alpha(w)-\alpha \Big( l_A(\alpha(u))w+r_A(\alpha(w)) u \Big)
                          \Bigg)v,\\
       \mlabel{v ac*=0}
&& L_{\tA,\star}^*\Bigg(
\alpha^*\bigg( L_{\tA,\star}^*\Big(\alpha(u)\Big)b^*\bigg)
-\alpha(u)\bullet\alpha^*(b^*)  -\alpha^*\bigg( -R_{\tA}^*\Big(\alpha^*(b^*)\Big)u\bigg)   \Bigg)w
\\&&=L_{\tA,\star}^*\Bigg(   \alpha(u)\circ\alpha(w)-\alpha \Big( l_A(\alpha(u))w+r_A(\alpha(w)) u \Big)           \Bigg)b^*, \notag\\
         \mlabel{u bc*=0}
&& L_{\tA,\star}^*\Bigg(
-\alpha^*(a^*)\bullet\alpha(v) -\alpha^*\bigg(L_{\tA,\star}^*\Big(\alpha^*(a^*)\Big)v\bigg)+\alpha^*\bigg(-R_{\tA}^*\Big(\alpha(v)\Big)a^* \bigg)
\Bigg)w
\\&&=L_{\tA,\star}^*\Bigg(  -\alpha^*(a^*)\bullet\alpha(w) -\alpha^*\bigg(L_{\tA,\star}^*\Big(\alpha^*(a^*)\Big)w\bigg)+\alpha^*\bigg(-R_{\tA}^*\Big(\alpha(w)\Big)a^* \bigg)  \Bigg)v,
\notag
\\               \mlabel{vw a*=0}
&& L_{\tA,\star}^*\Bigg( \alpha^*\bigg( L_{\tA,\star}^*\Big(\alpha(u)\Big)b^*\bigg)
-\alpha(u)\bullet\alpha^*(b^*)  -\alpha^*\bigg( -R_{\tA}^*\Big(\alpha^*(b^*)\Big)u\bigg)  \Bigg)c^*
\\&&=
L_{\tA,\star}^*\Bigg( \alpha^*\bigg( L_{\tA,\star}^*\Big(\alpha(u)\Big)c^*\bigg)
-\alpha(u)\bullet\alpha^*(c^*)  -\alpha^*\bigg( -R_{\tA}^*\Big(\alpha^*(c^*)\Big)u\bigg)  \Bigg)b^*,
\notag\\        \mlabel{uw b*=0}   
&& L_{\tA,\star}^*\Bigg( -\alpha^*(a^*)\bullet\alpha(v) -\alpha^*\bigg(L_{\tA,\star}^*\Big(\alpha^*(a^*)\Big)v\bigg)+\alpha^*\bigg(-R_{\tA}^*\Big(\alpha(v)\Big)a^* \bigg) \Bigg)c^*
=0.
\end{eqnarray}
Since for all $s^*\in V^\ast$ and $a^*\in A^*$, we have $L_{\tA,\star}^*(s^*)a^*=0$. Hence, Eqs.~(\mref{vw a*=0}) and (\mref{uw b*=0}) naturally hold. Then we shall prove that
$Eq.~(\mref{abc*=0})\Leftrightarrow Eq.~(\mref{*vcon1})$,
$Eq.~(\mref{v ac*=0})\Leftrightarrow Eq.~(\mref{Goper con3})$,
$Eq.~(\mref{u bc*=0})\Leftrightarrow Eq.~(\mref{Goper con4})$. The proofs of these statements are similar. Next, we prove that Eq.~(\mref{u bc*=0}) holds for all $a^\ast\in A^\ast$ and $v$, $w\in V$ if and only if Eq.~(\mref{Goper con4}) holds for all $x\in A$, $v$, $w\in V$. Let $LHS$ and $RHS$ denote the left-hand side and right-hand of Eq.~(\mref{u bc*=0}) respectively. Then for all $x\in A$ and $s^*\in V^\ast$, we have $$\langle LHS,s^*\rangle =\langle RHS,s^*\rangle=0.$$ Furthermore, we obtain
\begin{align*}
\langle LHS,x\rangle
&=-\Big\langle w,L_{\tA,\star}\Bigg(
-(-r_A^*)(\alpha(v))\alpha^*(a^*) -\alpha^*\bigg(L_{\tA,\star}^*\Big(\alpha^*(a^*)\Big)v\bigg)+\alpha^*\bigg(-R_{\tA}^*\Big(\alpha(v)\Big)a^* \bigg)
\Bigg)x\Big\rangle.
\\&=-\Big\langle w,(-r_A^*(x)+l_{A,\star}^*(x))
\Bigg(-(-r_A^*)(\alpha(v))\alpha^*(a^*) -\alpha^*\bigg(L_{\tA,\star}^*\Big(\alpha^*(a^*)\Big)v\bigg)+\alpha^*\bigg(-R_{\tA}^*\Big(\alpha(v)\Big)a^* \bigg)
\Bigg)\Big\rangle
\\&=\Big\langle l_A(x)w, r_A^*(\alpha(v))\alpha^*(a^*) -\alpha^*\bigg(L_{\tA,\star}^*\Big(\alpha^*(a^*)\Big)v\bigg)+\alpha^*\bigg(-R_\tA^*\Big(\alpha(v)\Big)a^* \bigg)\Big\rangle
\\&=-\Big\langle \alpha(r_A(\alpha(v))(l_A(x)w)) -\alpha(l_A(x)w)\circ\alpha(v),a^*\Big\rangle +
\langle L_{\tA,\star,}(\alpha^*(a^*))\alpha(l_A(x)w),v\rangle
\\&=-\Big\langle \alpha(r_A(\alpha(v))(l_A(x)w))
-\alpha(l_A(x)w)\circ\alpha(v),a^*\Big\rangle +
\langle l_A^*(\alpha(l_A(x)w))(\alpha^*(a^*)),v\rangle
\\&=\Big\langle -\alpha(r_A(\alpha(v))(l_A(x)w)) +\alpha(l_A(x)w)\circ\alpha(v)-\alpha(l_A(\alpha(l_A(x)w))v),a^*\Big\rangle .
\end{align*}
Similarly, one gets
\begin{eqnarray*}
\langle RHS,x\rangle
=\Big\langle -\alpha(r_A(\alpha(w))(l_A(x)v)) +\alpha(l_A(x)v)\circ\alpha(w)-\alpha(l_A(\alpha(l_A(x)v))w),a^*\Big\rangle .
\end{eqnarray*}
Therefore, Eq.~(\mref{u bc*=0}) holds for all $a^\ast\in A^\ast$ and $v$, $w\in V$ if and only if Eq.~(\mref{Goper con4}) holds for all $x\in A$, $v$, $w\in V$. Similarly, one can check that Eq.~(\mref{abc*=0}) holds if and only if Eq.~(\mref{*vcon1}) holds, Eq.~(\mref{v ac*=0}) holds for all $u$, $w\in V$ and $b^\ast\in A^\ast$ if and only if Eq.~(\mref{Goper con3}) holds for all $u$, $w\in V$ and $x\in A$.

On the other hand, by the same way, with Eqs. (\mref{Goper con3}) and (\mref{Goper con4}), one can check that Eq. ($\ast\ast$) holds if and only if Eqs. (\mref{Goper con5}) and (\mref{Goper con6}) hold. Then the proof is completed.
\end{proof}

\begin{cor}
Let $(A,\circ)$ be a Novikov algebra.
\begin{enumerate}
\item[(\romannumeral1)]Let $(M,\cdot,l_A,r_A)$ be an $A$-bimodule Novikov algebra and $\tA=A\ltimes_{l_{A,\star}^*,-r_A^*} M^*$. Let $\alpha,\beta :M\rightarrow A$ be two linear maps and define $\taa=\alpha^\tA$ by Eq.~(\mref{injective}). 
Suppose that $\alpha$ is an extended $\calo$-operator of weight $\lambda$ with extension $\beta$ of mass $(\kappa,\mu)$ on $(A,\circ)$ associated to $(M,\cdot,l_A,r_A)$. Then $\check{\taa}_-$ is a skew-symmetric solution of the GNYBES in $\tA$ if and only if the following equalities hold for all $u,v,w \in M$ and $x\in A$:
\begin{eqnarray}             \mlabel{Goper cor-a1}
&&\lambda l_A(\alpha(u \cdot v))w= \lambda l_A(\alpha(u \cdot w))v,\\
\mlabel{Goper cor-a2}
&&\lambda l_A(\alpha(u \cdot v))w-\lambda l_A(\alpha(v \cdot u))w
=\lambda r_A(\alpha(v \cdot w))u-\lambda r_A(\alpha(u \cdot w))v,\\
 &&           \mlabel{Goper cor-a3}
\lambda x\star \alpha(u \cdot w)=\lambda \alpha( (l_A(x)w)\cdot u)
+\lambda \alpha(u \cdot (l_A(x)w)),\\
            \mlabel{Goper cor-a4}
&&\lambda \alpha( (l_A(x)w)\cdot v) = \lambda \alpha((l(x)v)\cdot w),\\
&&             \mlabel{Goper cor-a5}
\lambda \alpha( (l_A(x)w)\cdot v)  +\lambda \alpha( v\cdot(l_A(x)w))=\lambda x\circ \alpha(v\cdot w)
+\lambda \alpha((r_A(x)v)\cdot w) ,\\
&& \mlabel{Goper cor-a6}
 \lambda \alpha((r_A(x)u)\cdot v)+ \lambda \alpha(v\cdot (r_A(x)u))=\lambda \alpha((r_A(x)v)\cdot u)+\lambda \alpha(u\cdot (r_A(x)v))\Big).
\end{eqnarray}
\begin{enumerate}
\item If $\alpha$ is an extended $\calo$-operator of weight $0$ with extension $\beta$ of mass $(\kappa,\mu)$, then $\check{\taa}_-$ is a skew-symmetric solution of the GNYBES in $\tA$.
\delete{ Further, if $\alpha$ is an extended $\calo$-operator with extension $\beta$ of mass $(\kappa=0,\mu)$, $\beta$ is naturally an $A$-bimodule Novikov homomorphism. Then $\check{\taa}_-$ is a skew-symmetric solution of GNYBES.}
\item When $\beta=0$ or $(\kappa,\mu)=(0,0)$, i.e., $\alpha$ is just an extended $\calo$-operator of weight $\lambda$ on $(A, \circ)$ associated to $(M,\cdot,l_A,r_A)$, $\check{\alpha}-\tau(\check{\alpha})$
is a skew-symmetric solution of the GNYBES in $\tA$ if and only if Eqs. (\ref{Goper cor-a1})-(\ref{Goper cor-a6}) hold.
\end{enumerate}
\item[(\romannumeral2)]Let $(V,l_A,r_A)$ be a $A$-bimodule Novikov algebra. Let $\alpha,\beta :V\rightarrow A$ be two linear maps such that $\alpha$ is an extended $\calo$-operator of weight $\lambda$ with extension $\beta$ of mass $(\kappa, \mu)$ on $(A,\circ)$ associated to $(V,l_A,r_A)$. Then $\check{\taa}_-\in (A\ltimes_{l_{A,\star}^*,-r_A^*} V^*)\ot (A\ltimes_{l_{A,\star}^*,-r_A^*} V^*)$ is a skew-symmetric solution of the GNYBES in $A\ltimes_{l_{A,\star}^*,-r_A^*} V^*$.
\item[(\romannumeral3)]
Let $\alpha:A\rightarrow A$ be an extended $\calo$-operator of weight $0$ with extension $\beta=\id$ of mass $(\kappa, \mu)$ on $(A,\circ)$ associated to $(A, L_A, R_A)$, i.e., Eq.~(\mref{T and kappa}) holds. Then $\check{\alpha}-\tau(\check{\alpha}) \in (A\ltimes_{L_{A,\star}^*,-R_A^*} A^*)\ot (A\ltimes_{L_{A,\star}^*,-R_A^*} A^*)$ is a skew-symmetric solution of the GNYBES in $A\ltimes_{L_{A,\star}^*,-R_A^*} A^*$.
\end{enumerate}
\end{cor}

\begin{proof}
(\romannumeral1) Since $\alpha$ is an extended $\calo$-operator of weight $\lambda$ with extension $\beta$ of mass $(\kappa,\mu)$ on $(A,\circ)$ associated to $(M,\cdot,l_A,r_A)$, for all $u,v\in M$, we have $$B_\alpha(u,w)=\alpha(u)\circ\alpha(v) -  \alpha\Big( l_A(\alpha(u))v+r_A(\alpha(v)) u \Big)=\lambda \alpha(u\cdot v)+\kappa \beta(u)\circ\beta(v)+\mu\beta(u)\cdot \beta(v).$$ Since $\beta$ is balanced, $A$-invariant of mass $\kappa$ and equivalent of mass $\mu$, by  Theorem \mref{Goper con}, one can directly check that Eqs. (\ref{*vcon1})-(\ref{Goper con6}) are equivalent to Eqs. (\ref{Goper cor-a1})-(\ref{Goper cor-a6}).

The other conclusions follow directly from Item (\romannumeral1).
\end{proof}

\noindent {\bf Acknowledgments.} This research is supported by
NSFC (No. 12171129) and  the Zhejiang
Provincial Natural Science Foundation of China (No. Z25A010006).

\smallskip

\noindent
{\bf Declaration of interests. } The authors have no conflicts of interest to disclose.

\smallskip

\noindent
{\bf Data availability. } Data sharing is not applicable to this article as no new data were created or analyzed in this study.

\end{document}